\documentclass[mnsc,nonblindrev]{informs3noheader} % current default for 

\usepackage{bbm}
 \usepackage{setspace}
 \usepackage{algorithm}
 \usepackage{accents}
 \usepackage{booktabs}

\usepackage{multirow}
\usepackage{soul}
\usepackage{float}

\DoubleSpacedXI

\usepackage{natbib}
 \bibpunct[, ]{(}{)}{,}{a}{}{,}%
 \def\bibfont{\small}%
\TheoremsNumberedThrough     % Preferred (Theorem 1, Lemma 1, Theorem 2)
\ECRepeatTheorems

\EquationsNumberedThrough    % Default: (1), (2), ...
\MANUSCRIPTNO{} % When the article is logged in and DOI assigned to it,
\newcommand{\lref}[2][]{\hyperref[#2]{#1~\ref*{#2}}}
\newcommand{\eq}[1]{\hyperref[#1]{Eq.~(\ref*{#1})}}
\newcommand{\eqbare}[2][]{\hyperref[#2]{#1~(\ref*{#2})}} % Ref. to equation number only
\newcommand{\multieqbare}[3][]{\hyperref[#2]{#1~(\ref{#2}--\ref{#3})}}
\numberwithin{equation}{section}

\newcommand{\eqbareset}[3][]{\hyperref[#3]{#1~(\ref*{#2}-\ref*{#3})}}

\newcommand{\sectlabel}[1]{\label{#1}}
\newcommand{\eqlabel}[1]{\label{#1}}

\newtheorem{theorem*}{Theorem}
\newtheorem{proposition*}{Proposition}
\newtheorem{corollary*}{Corollary}

\newcommand{\ca}{\mathcal}
\def\phcomments{0}
\newcounter{note}[section]
\renewcommand{\thenote}{\thesection.\arabic{note}}
\newcommand{\note}[2]{\refstepcounter{note}\marginpar{\small\bf \textcolor{red}{#1~\thenote}}$\ll${\sf \textcolor{red}{#1's
 Comment~\thenote:}} {\em \textcolor{red}{#2}}$\gg$}

\ifnum\phcomments=1
\newcommand{\notePH}[1]{\note{PH}{#1}}
\else
\newcommand{\notePH}[1]{}
\fi

\ifnum\phcomments=1
\newcommand{\noteSS}[1]{\note{SS}{#1}}
\else
\newcommand{\noteSS}[1]{}
\fi

\newcommand{\highlightcolor}{black} 

\begin{document}
%%%%%%%%%%%%%%%%

%\setcounter{page}{1}
%\DOI{10.1287/xxxx.0000.0000}%


%\RUNTITLE{}

% \TITLE{Bundling Information Goods of Decreasing Value}
\TITLE{An Optimistic-Robust Approach for Dynamic Positioning of Omnichannel Inventories}

\ARTICLEAUTHORS{%
\AUTHOR{Pavithra Harsha, Shivaram Subramanian, Ali Koc, Mahesh Ramakrishna, Brian Quanz, Dhruv Shah, Chandra Narayanaswami}
\AFF{IBM T. J Watson Research Center, Yorktown Heights, NY 10570s\\ \EMAIL{Author Contact: pharsha@us.ibm.com}}
%\AUTHOR{Shivaram Subramanian}
%\AFF{IBM Research, \EMAIL{}, \URL{}}
} % end of the block

\ABSTRACT{%
We introduce a new class of data-driven and distribution-free optimistic-robust {\em bimodal} inventory optimization (BIO) strategy to effectively allocate inventory across a retail chain to meet time-varying, uncertain omnichannel demand. 
\textcolor{\highlightcolor}{The bimodal nature of BIO stems from its ability to balance downside risk, as in traditional Robust Optimization (RO), which focuses on worst-case adversarial demand, with upside potential to enhance average-case performance. This enables BIO to remain as resilient as RO while capturing benefits that would otherwise be lost due to endogenous outliers.
%While prior Robust optimization (RO) methods emphasize the downside, i.e., worst-case adversarial demand, BIO also considers the upside to remain resilient like RO while also reaping the rewards of improved average-case performance by overcoming the presence of endogenous outliers. This bimodal 
Omnichannel inventory planning provides a suitable problem setting for analyzing the effectiveness of BIO's bimodal} strategy in managing the tradeoff between lost sales at stores and cross-channel e-commerce fulfillment costs,
%strategy is particularly valuable for balancing the tradeoff between lost sales at the store and the costs of cross-channel e-commerce fulfillment, 
%which lies at the core of our inventory optimization model. These 
factors that are inherently asymmetric due to channel-specific behaviors.
%These factors are asymmetric due to the heterogeneous behavior of the channels, with lost sales costs disproportionately affecting stores, while e-commerce fulfillment depends on network effects.
%with a bias towards the former in terms of lost sales cost and a dependence on network effects for the latter. 
We provide structural insights about the BIO solution and how it can be tuned to achieve a preferred tradeoff between robustness and the average-case performance. %Our experiments show that significant benefits can be achieved by rethinking traditional approaches to inventory management, which are siloed by channel and location. \textcolor{red}{Do we need this?} 
\textcolor{\highlightcolor}{Using a real-world dataset from a large American omnichannel retail chain, a business value assessment during a peak period indicates that BIO outperforms pure RO by 27\% in terms of realized average profitability and surpasses other competitive baselines under imperfect distributional information by over 10\%. This demonstrates that BIO provides a novel, data-driven, and distribution-free alternative to traditional RO that achieves strong average performance while carefully balancing robustness.
%, at least a 10\% profitability gain compared to RO and other competitive baselines leveraging imperfect distributional information, while also preserving the (practical) worst-case performance.
}}%
%additionally project a 15\% profitability gain for BIO over RO on a real-world dataset from a mid-sized American omnichannel retailer.
%\KEYWORDS{deterministic inventory theory; infinite linear programming duality; 
%  existence of optimal policies; semi-Markov decision process; cyclic schedule}

\HISTORY{}

\maketitle
%%%%%%%%%%%%%%%%%%%%%%%%%%%%%%%%%%%%%%%%%%%%%%%%%%%%%%%%%%%%%%%%%%%%%%

% Samples of sectioning (and labeling) in MSOM
% NOTE: (1) \section and \subsection do NOT end with a period
%       (2) \subsubsection and lower need end punctuation
%       (3) capitalization is as shown (title style).
%
%\section{Introduction.}\label{intro} %%1.
%\subsection{Duality and the Classical EOQ Problem.}\label{class-EOQ} %% 1.1.
%\subsection{Outline.}\label{outline1} %% 1.2.
%\subsubsection{Cyclic Schedules for the General Deterministic SMDP.}
%  \label{cyclic-schedules} %% 1.2.1
%\section{Problem Description.}\label{problemdescription} %% 2.

% Text of your paper here
\section{Introduction}
An increasingly digital economy has led to a proliferation of virtual shopping channels, making omnichannel shopping ubiquitous among today's consumers, a trend that extends across diverse product categories as well as emerging markets that were previously dominated by brick-and-mortar retailing~\citep{neilson}. The immediacy of this omnichannel consumer demand has driven today’s supply chains to be agile and responsive. For example, retail chains now offer various delivery options in addition to serving walk-in customers. These options often come with service guarantees, such as same-day delivery, free two-day shipping, and buy-online-pick-up-in-store (BOPUS) within two hours. To serve this diversified consumer demand, retailers have opened multiple facilities to minimize fulfillment costs and expedite delivery. They also employ cross-fulfillment methods such as serving demand from distant warehouses or utilize ship-from-store (SFS) to reduce lost sales opportunities and/or meet on-time delivery targets. 

 The dynamics in the retail industry in 2020 due to the pandemic have forced retailers to think much more broadly and invest more rapidly in their omnichannel capabilities~\citep{Bourlier2020}. {In fact, the digital transformation is prompting companies to also pool their omnichannel capabilities, including physical assets like warehouse space or last-mile transportation, as well as information, leading to mutual profitability and propelling them further ahead in their omnichannel journey~\citep{HBR2022}.}

Omnichannel inventory systems (OIS) are an essential component of this rapid shift from product-centric to customer-centric supply chains. Optimal inventory positioning across the retail chain is a critical aspect of OIS and aims to place the right amount of inventory in the right location to serve omnichannel demand. The traditional planning methods that are siloed by channel and/or location {fall short of this goal due to conflicting assumptions on how the inventory is used for downstream fulfillment. This may lead to} stock-outs, particularly among walk-in and BOPUS customers, low margins due to expensive fulfillment costs for online demand, excessive carrying costs, and price markdowns.
%The traditional planning methods are siloed by channel and/or location with conflicting assumptions on how the inventory is used for downstream fulfillment leading to excessive amount of inventory system-wide, and a dearth or a excess locally depending on how the inventory is eventually fulfilled. 
To mitigate these challenges, network-based planning models that incorporate fulfillment flexibility of the e-commerce demand {(and simultaneously the inflexibility in meeting walk-in demand)}
are key to tap into the efficiencies of cross-channel and cross-location resource pooling. 

% Omnichannel inventory systems (OIS) are an essential component of this rapid shift from product-centric to customer-centric supply chains. Optimal inventory positioning across the retail chain is a critical aspect of OIS and aims to place the right amount of inventory in the right location to serve omnichannel demand. Failing to do so may lead to stock-outs, particularly among walk-in and BOPUS customers, low margins due to expensive fulfillment costs for online demand, excessive carrying costs, and price markdowns. 
% %
% The traditional planning methods are siloed by channel and/or location with conflicting assumptions of on how the inventory is used for downstream fulfillment leading to excessive amount of inventory system-wide, and a dearth or a excess locally depending on how the inventory is eventually fulfilled. To mitigate these challenges, network-based planning models that incorporate fulfillment flexibility of the e-commerce demand %(and simultaneously the inflexibility of walk-in demand) 
% are key to tap into the efficiencies of cross-channel and cross-location pooling. 

Our work aims to develop a data-driven distribution-free network-based inventory management policy for the OIS that optimizes the quantity of inventory that is periodically purchased %from the supplier and thereafter 
and positioned in the retail chain to dynamically match the supply with time-varying and uncertain omnichannel demands to maximize the retail chain's total profit margins with the dual goals of maximizing sales and minimizing fulfillment costs.

Robust inventory
planning, which is inherently data-driven and distribution free, primarily focuses on the \textit{downside} objective, i.e., worst-case demand scenarios and mitigating risks associated with
adversarial demand patterns. However, to be truly customer-centric, an OIS must equally be able to take advantage of the \textit{upside} of making inventory positioning decisions that drive long-term profitable sales and grow its customer base. 
%This will allow the OIS to remain resilient to potential negative impacts while also reaping the rewards of positive outcomes. This additional requirement of positioning to gain from potential upsides,  besides planning for worst case objectives, requires a \textit{bimodal} resource allocation strategy, which has been given less attention in the robust optimization (RO) literature and is a focus of this paper. 
\textcolor{\highlightcolor}{Achieving both resilience against negative outcomes and the ability to seize positive opportunities requires a \textit{bimodal} resource allocation strategy--one that not only safeguards against worst-case demand but also optimizes for potential gains. This dual approach, which retains data-driven and distribution-free properties, has received less attention in the robust optimization (RO) literature and is the focus of this paper.} We argue that a bimodal inventory management strategy is particularly valuable for balancing the tradeoff between lost sales at physical stores and the costs of cross-channel e-commerce fulfillment, which is at the core of our model. These tradeoffs are asymmetric due to the heterogeneous behavior of different channels, with store inventory facing higher lost sales costs, while e-commerce fulfillment depends on network effects. \textcolor{\highlightcolor}{Hence, omnichannel inventory planning is a challenging problem setting for analyzing the effectiveness of BIO's bimodal strategy.}
%{We believe that this bimodal inventory management strategy is particularly valuable for balancing the tradeoff between lost sales at the store and the costs of cross-channel e-commerce fulfillment, which is at the core of our model. These factors are asymmetric due to the heterogenous behavior of the channels, with a bias towards the former in terms of lost sales cost and a dependence on network effects for the latter.} 
The contributions of this paper are as follows:

\begin{enumerate}
\item We { propose a optimistic-robust replenishment policy based on a  data-driven, distribution-free Bimodal inventory optimization (BIO) model that optimizes the  allocations using a convex combination of best- and worst-case demand scenarios via an user-chosen optimism hyperparameter}. Pure adversarial models seek to misalign demands from inventory 
resulting in sales that are less responsive to increased (but potentially skewed) inventory allocations resulting in overly pessimistic replenishment decisions. BIO models overcome this limitation, specifically the presence of decision-dependent outliers, by simultaneously balancing the best- and worst-case scenarios.
%We include illustrative examples to compare and contrast the behavior of BIO and vanilla RO in various settings.
%\textcolor{red}{Dual goals in omnichannel maximizing sales and minimizing fulfillment cost}

\item We { provide insights into the structure of the optimal allocation of the BIO problem} by showing that under  {the optimal allocation
%the same user-selected convex-combination of the respective optima
follows the same blending of the best- and worst- case optima as the user-selected scenario blending.}  %\noteSS{not fully clear- what is 'it' and 'same' as what} 
%Furthermore, {we establish a one-to-one correspondence between well-known convex inventory management objectives (e.g., SAA) and the blending hyperparameter value.} \noteSS{should we update the previous sentence based on the updated corollary?} Together 
This establishes a relatively easy mechanism to identify the corresponding optimal allocation via a simple out-of-sample evaluation, enabling decision-makers to determine an allocation that achieves a desired balance between objectives, such as expected profit and robustness.
%, based on their personal preferences.

%\notePH{convex combination of uncertainities with given set; Optimistic was needed for heterogenous behaviours of the two channels}
\item {Next, we derive an exact linear reformulation of the adversarial bi-linear BIO sub-problem. This is achieved by demonstrating and leveraging the integrality of the extreme points within the selected uncertainty set, which is constructed using inputs from a demand forecasting model.} The derived sub-problem is then employed within a column and cut generation (CCG) algorithm. 
This algorithm is utilized to solve large-scale real-world BIO instances characterized by generalized fulfillment cost structures, operational restrictions, and omnichannel business constraints.

\item Last, we perform extensive Monte Carlo simulation experiments using real data of a retail chain that operates over 150 stores spread across the continental United States.
%during multiple weeks of a peak holiday period.
We observe that the incremental benefit of BIO models, including pure RO, is pronounced in the context of omnichannel operations %(where allocations are made both to online and store demands) 
over the baseline models (which are otherwise optimal in a pure walk-in setting). {This advantage arises from the effective utilization of cross-channel resource pooling.}  %We also observe that the BIO models can significantly enhance the performance of the pure RO models by tuning a single optimism hyper-parameter. ****  OR **** First we observe that the incremental benefit of BIO models, including pure RO, is most pronounced in the context of omnichannel operations %(where allocations are made both to online and store demands), 
%where the average profitability gains are in the range 5-15\% across different parameter and initial inventory settings over the baseline models (note that the latter are optimal in a pure walk-in setting). 
We observe that the BIO models can enhance the performance of the pure RO models on average by up to 10\% by tuning a single optimism hyper-parameter across different parameter and initial inventory settings. In a comprehensive business value assessment over a rolling horizon with a realistic fine-grained customer transaction level fulfillment engine, we observe that %pure RO outperform baselines by 8\% and 10\% in realized average profitability and inventory turn-over metrics, while 
%BIO outperforms pure RO by 15\% and 35\% in the realized average profitability and inventory turn-over metrics respectively.
\textcolor{\highlightcolor}{BIO outperforms pure RO by 27\% in terms of realized average profitability and surpasses other competitive baselines under imperfect distributional information by over 10\%. This demonstrates that BIO provides a novel, data-driven, and distribution-free alternative to traditional RO that achieves strong average performance while carefully balancing robustness.}

%\noteSS{#4 consider trimming/shortening. we quote 4 different improvement numbers here.. would 2 key numbers suffice and preserve focus?. Other e.g. not fully clear what "single replenishment decision.." means as we have not described that setting here yet.. }
% Comparing the single replenishment decision of the different standalone optimizers we observe that by inducing optimism, the BIO models enhance the performance of the pure RO models on average by up to 10\% by tuning a single optimism hyper-parameter.
% We also observe that the incremental benefit of BIO models, including pure RO, is most pronounced in the context of omnichannel operations (where allocations are made both to online and store demands), where the average profitability gains  are in the range 5-15\% over the baseline models (which are otherwise optimal in a pure walk-in setting). 
% Next we perform a business value assessment by comparing the standalone optimizers over a rolling horizon with a more realistic fine-grained customer transaction level fulfillment engine. We observe that pure RO outperform baselines by 8\% and 10\% in realized average profitability and inventory turn-over metrics, while BIO outperforms pure RO by 15\% and 35\% respectively. 
\end{enumerate}
%\notePH{Generalizability of the methods}

We would like to note that our modeling choices %are influenced by
address several practical considerations including:
(1) very low-rate of sales interspersed with brief but sharp sales peaks where the shared inventory system is stretched to capacity,
(2) the degree of uncertainty in multi-period, multi-channel demand forecast are inputs generated by AI/ML models,
(3) the cascading effects of decisions span across a large omnichannel network consisting of hundreds of nodes, 
(4) the manifold impact of lead times, cross-fulfillment costs, and lost sales opportunities, and finally
(5) to demonstrate practical effectiveness on real-world data in a multi-period rolling horizon setting.

% \notePH{ The term 'bimodal inventory optimization strategy'}
%  \notePH{ the 3 (or 4) key differentiators to drive the litrev comparisons are:
%   - our OCI modeling that considers tradeoff between store lost sales and cross-channel fulfilment costs that prior OCI or single channel do not consider.
%   - A bimodal resource allocation strategy that works with (independent) forecast-driven uncertainty sets, and how this naturally fits with above OCI goal.
%   - prior inventory devaluation that is critical in a practical application that prior papers ignore.
%   - a 4th differentiator would be problem size/network complexity with real data. I think we saw one paper that solved large networks, but did not consider lost sales?}

\section{Related Literature}~\label{literature}
We classify the related literature into three main areas: (1) Inventory management considering cross-fulfillment, %connections to
transshipment, and omnichannel retailing; (2) RO, distributional, and Pareto robust techniques; and (3) Kelly's principle.

%We categorize our related literature into three broad topics: (1) inventory management under cross-fulfillment considerations, the connections to the  transshipment literature and omnichannel retailing, in general; (2) RO, distributionally and pareto robust methods, and (3) the Kelly's principle. %, and the more recent optimistic-robust methods. 
%; (3) Kelly's principle and the more recent optimistic-pessimistic approaches.  %(multi-period rather than distributional assumptions as the latter requires covariance matrix)

Inventory management with cross-fulfillment  considerations is a topic of growing interest in the operations management community. The pioneering paper of \citet{acimovic2017mitigating} studies the periodic-review joint-replenishment problem for a group of fulfillment centers (FCs) serving e-commerce demand. They propose heuristic policies to mitigate costly demand spill overs across FCs assuming a myopic fulfillment policy in the presence of lead times. This paper and those we describe next advocate the need for integrated planning across locations rather than decentralized strategies, when cross-fulfillment is involved.
\citet{bandi2019sustainable} develop robust-periodic affine policies for newsvendor networks in a backorder setting for single and multiple items. \citet{lim2021integrating} presents a two-phase RO approach (binary decision variables in phase 1 with an elastic uncertainty set and continuous decision variables in phase 2) to the joint replenishment and fulfillment problem for e-commerce networks with drop shipment (no backlog or lost sales). %Binary decisions corresponding on replenishment arrival are identified in the first phase using user-specified worst-case cost targets followed by solving for the continuous replenishment and fulfilment decisions using affine rules to handle the adaptability. Binary decisions are identified in the first phase using an elastic box-demand uncertainty set that is adjustable within a range to optimally satisfy a feasible user-specified worst-case cost target. This is followed by solving for the continuous replenishment and fulfilment decisions using affine rules to handle the adaptability. 
\citet{govindarajan2021distribution} develop a semi-definite programming (SDP) based heuristic to solve the distributionally robust multi-location inventory allocation problem on a newsvendor network (single period, e-commerce channel) assuming a L-level nested fulfillment cost structure. %There is a vast literature related to inventory management for the traditional walk-in channel in networked settings, see~\citealt{jackson2019multiperiod} and the references therein. Single channel models, as described above, cannot be used for omnichannel precisely because we have to simultaneously handle the inability to cross-fulfill store demand and the flexibility to do so for online demand across multiple locations. 
These single channel e-commerce models as well as the traditional pure walk-in channel models (see~\citealt{jackson2019multiperiod} and the references therein), cannot be used for omnichannel setting precisely because we have to simultaneously account for the inability to cross-fulfill store demand and the flexibility to do so for online demand across multiple locations.

\citet{JolineNaval} develop a heuristic for the joint initial positioning and fulfillment of a seasonal item in a lost sales omnichannel setting. They require distributional assumptions at all locations and assume that (1) any unfulfilled store demand can be cross-fulfilled and (2) the fulfillment cost has an L-level nested structure. %Rigorously modelling the latter inherently imposes on our ability to handle inventory pooling and complex network fulfillment, while in the former one can get away with more local models and simple fulfillment 
%\citet{chen2019distributionally} develop a distributionally robust approach for optimal stocking in a grocery setting serving online and offline customers from a single location. 
%\notePH{Cannot use online models for omnichannel because of lost sales. Pure brick - we cannot fulfillment and shared inventories cannot be used}
Our paper also focuses on the multi-location omnichannel setting with lost sales, but considers the problem of periodic replenishment with lead times and general fulfillment cost structures, without relying on the aforementioned assumptions. {We explicitly consider the tradeoff between the lost walk-in sales at stores and cross-channel fulfilment costs that prior single channel or omnichannel models do not address.} 

%Methodologically, we propose a data-driven optimistic-robust framework that leverages uncertainty sets.  %Other papers that also study inventory allocation problem but not 

The reactive nature of fulfillment is analogous to a zero-transshipment setting with no lead time~\citep{yang2007capacitated}. For an extensive review of this topic, readers can refer to~\citet{paterson2011inventory}. Transshipment problems, in general, are intractable~\citep{tagaras1992pooling}, and by extension, it is the case with the replenishment problem with fulfillment that we study.

Omnichannel retailing more broadly (and encompassing inventory management) has received increasing attention in the literature with several works analyzing the impact
%also garnered increasing attention in recent years. Many other papers study of different aspects of omnichannel retailing and these include the value 
of inventory and product information sharing across channels~\citep{gallino2014integration, gao2017online}, fulfillment optimization~\citep{acimovic2015making, jasin2015lp, ali2017optimization}, fulfillment flexibility~\citep{devalve2023understanding}, omnichannel pricing~\citep{lei2018joint, harsha2019practical, OCPX}, and store pick-ups~\citep{gallino2017channel, gao2017omnichannel} and returns~\citep{nageswaran2020consumer}.

Robust Optimization (RO) is a widespread methodology for modeling decision-making problems under uncertainty and it seeks to optimize the worst-case performance across all scenarios within the pre-specified uncertainty set (see~\citealt{bertsimas2011theory, gorissen2015practical}, and the robust inventory management papers referenced above and the references there-in). %The additional requirement to gain from any potential upside requires a bimodal optimization (BIO) approach, which has been given less attention in the robust optimization (RO) literature and is a focus of this paper. %We first distinguish BIO from the pareto RO (PRO) and target-oriented RO introduced by~\citet{iancu2014pareto} and ~\cite{lim2017inventory} respectively.  
%In practice, RO problems usually admit multiple worst-case optimal solutions, the performance of which may differ substantially under non-worst-case uncertainty scenarios. So,
 A long standing criticism of RO, especially among practitioners is that it produces overly conservative decisions.  \citet{iancu2014pareto} propose models that can identify pareto-optimal robust (PRO) solutions that perform just as well in the worst case and same or better in non-worst case scenarios. One well-studied  approach to identify practical viable solutions is to use budget constraints that are derived from limit laws to contain the variability~\citep{bandi2012tractable}. A limitation of this approach we observe in our setting is the inability of RO to overcome endogenous outliers that continues to result in conservative solutions (for e.g., those that can only be eliminated by decision dependent limit-based uncertainty sets whose budget bounds are hard to characterize in closed form, see Section~\ref{model} for a motivating example). Another recent work motivates variable (box) uncertainty sets that can be tuned in accordance with user-specified cost target~\citep{lim2017inventory}. In this paper, we derive data-driven hierarchical budget constraints (in similar spirit to the limit laws) as an independent, fixed input from a demand forecasting model.

 Distributionally robust optimization (DRO) is yet another paradigm to mitigate conservatism by seeking to identify the worst-performing distribution from an ambiguity class~\citep{rahimian2019distributionally}. Recently \citet{gotoh2023data} introduce a class of distributionally optimistic optimization (DOO) models and show that it is always possible to outperform the (finite sample) SAA if ones considers both worst- and best-case methods, noting that the optimistic solutions more sensitive to model error than either worst-case or SAA optimizers. %\citet{DFO} introduce the notation of distributionally favorable optimization (DFO) to handle endogenous outliers. %\citet{cao2021contextual} leverages the best/worst case rewards to develop confidence bounds on the rewards and develops performance guarantees for robust contextual optimization in online and offline settings. 
%Distributionally Optimistic Optimiza- tion (DOO) models, and show that it is always possible to “beat” SAA out-of-sample if we consider not just worst-case (DRO) models but also best-case (DOO) ones. We also show, however, that this comes at a cost: Optimistic solutions are more sensitive to model error than either worst-case or SAA optimizers, and hence are less robust.
In contrast to the pure worst-case (RO, DRO) or pure best-case (DOO), the bi-modal inventory management strategy BIO that we introduce leverages an allied-adversary that simultaneously seeks to gain from potential upsides while managing an adversary across scenarios within a given uncertainty set.

Our approach is inspired by the seminal work of \citet{kelly1956new} that examines the problem of maximizing the growth rate of wealth over a series of independent trials to identify an optimal fraction of resources to bet under uncertainty %on an uncertain outcome 
assuming re-investment. 
%Under the assumptions of this model, there is a positive expectation of profit in every trial, which leads to a choice of $f = 1$ if we choose a ‘best case’ objective of maximizing the expected return. \notePH{f of 1 is expected return. In our case, it is the best case return. Also, maybe all of this can be moved upfront} On the other hand, the probability of ruin over many trials converges to 1.0, and the safest `$f$' to avoid this worst case is to bet as little as possible, which severely reduces profitability. 
The optimal `Kelly criterion' for this specific problem setting is an analytically derived value that lies between the twin extremes of maximizing expected gain {(invest as much as possible)} 
and minimizing the probability of ruin {(invest as little as possible)}
in order to achieve long-term gain~\citep{rotando1992kelly}. 
%Furthermore, the Kelly approach requires only the expected value of investment return and is a distribution-free alternative to portfolio selection ~\cite{thorp1975portfolio}, in contrast to the classic Markowitz portfolio optimization model (~\cite{markowitz7portfolio}). 
We see a parallel to Kelly’s bimodal decision-making approach within our multi-period robust inventory optimization problem where we seek a balance between procuring and positioning significant inventory levels spread uniformly across the network versus procuring optimistically low quantities and positioning them in alignment with confident point forecasts. \textcolor{\highlightcolor}{However, unlike Kelly’s framework, our setting does not involve periodic reinvestment and the resulting compounding gains or losses.}

\section{Problem setting, notation and assumptions}~\label{settings}
Consider an omnichannel retailer selling a single item and operating a network of retail stores and warehouses, both of which we refer to as nodes and denote by $\Lambda$. The retailer serves walk-in demand and e-commerce demand and uses the inventory in the stores and warehouses to fulfill the demands in various channels%\footnote{Note that the retailer can have other types of omnichannel demand like the buy-online-pick-up-in-store demand. For simplicity of exposition, we do not include them even though the formulations we propose in the paper can be extended to handle these other demand types. }
, in particular, the store inventory to meets its walk-in demand and the network inventory across nodes for online demand. The goal of the retailer is to purchase inventory from a supplier $S$ and distribute this %and any pre-existing inventory from different nodes of the network to other nodes of the network %(say warehouses say stores typically for fashion items \notePH{do we need to include this use case}) 
across the different nodes in the network
to maximize profitability over a finite forward looking time horizon $\ca{T}$, assuming that any unmet demand is lost. In Appendix~\ref{DCtoStores}, we provide the  modifications to the models if we additionally consider re-positioning of inventory between different nodes of the network. We explain the key variables, the parameters, the sequential decision programming framework and our assumptions next, while the detailed notations are described in Table~\ref{tab:notation}. 

\noindent {\bf State:} The state denoted by  ${\bf I}$ comprises of the current on-hand inventory and that in the pipeline for all $l \in \Lambda$. At time $t$ for location $l$, the state is $I_{tl}^{j}$ where $j =0,...,L_l $ where $L_l$ is the %maximum of the 
deterministic lead time to node $l$
%s $L_{l'l} $ across all pair of connected nodes $l'$ to $l$ where $l' \in \Lambda\cup S$, 
with index 0 referring to on-hand inventory. The initial inventory is ${\bf \bar{I}}$.

\noindent {\bf Action:}
The retailer aims to make a decision $x_{tl}$ for all $l \in \Lambda $ which is the quantity of inventory to purchase and move from the supplier node $S$ to node $l$ at time $t$. 
%The retailer aims to make a decision $x^t_{ll'} \ \forall \ l \in S\cup \Lambda $ which is how much inventory to purchase and move from the supplier $S$ or move from the nodes in the network ($l$) to other nodes in the network ($l'$) at time $t$. 
Although we model and solve for decisions across all the $\ca{T}$ periods, we only execute on the current decisions, i.e., decisions are made in a rolling horizon fashion. In practice, the demand forecast is periodically updated based on realized demands and the optimization model is re-solved. A multi-period problem is essential to account for the positive lead time, specifically the future impact of decisions taken in the current period and avoid trivial myopic solutions.
%Also, because of positive lead time, we have to model multiple periods ahead to capture the future impact of the actions taken in the current period and avoid trivial \noteSS{ myopic} solutions. 

%SS: not necessary, but if space permits, one line to say why multiperiod is better than myopic single-period approach.

\noindent {\bf Uncertainty:} We model the walk-in demand at the node level and the e-commerce demand at a region level that we refer to as zones $Z$. A zone is a geographically contiguous cluster of zip-codes that have the same cost of shipping from any given origin node. %(\noteSS{shipping cost is the same across all origin nodes- or shipping co a given origin node is the same for all zipcodes-i assumed the latter} .
% to any zip-code within a zone.  
We denote the uncertain walk-in and online demands at node $l\in \Lambda$ and zone $z\in Z$ by ${D}^b_{tl}$ and ${D}^o_{tz}$ respectively. %Later in this section,  we describe how we leverage AI/machine learning based demand forecasting models to better characterize these uncertain quantities and model them within the proposed optimization framework. 

\noindent {\bf Reward:} To compute the retailer's profitability, we are given (non-negative) prices and back-order penalties in the different channels in the down stream periods. They are denoted by $p^b_{tl}$, $b^b_{tl}$ for walk-in demand and $p^o_{t} $, $b^o_{t} $ for online demand. Note that the walk-in prices and penalties can vary by location while e-com prices are fixed across locations. We assume that $p^b_{tl} + b^b_{tl} > p^o_t +b^o_{t} - c_{lz} \ \forall i \in \Lambda, z \in Z, t \in \ca{T} $ where $c_{lz}$ is the fulfillment cost from node $l$ to any zone $z$. 
%This assumption, it is always more beneficial to serve a walk-in customer than meeting online demand. 
This assumption effectively enforces a service priority for a walk-in customer over an online order, thus avoiding stock hedging (hiding inventory in the backroom) for an potential online customer rather than serving an existing walk-in customer.
We also assume that prices and penalties are non-increasing over time to avoid a similar stock hedging phenomenon over time and ensure that it is preferable to fulfill an existing order now rather than wait until a future time period. Let $h_l$ denote the holding cost for left over inventory in node $l$ in any period and $C_{l}$ denote the combined purchase and transportation cost of moving inventory from the supplier to node $l$. %Also, let $C_{ll'}$ denote the purchase cum transportation cost of moving inventory from node $l$ to $l'$ where $ l \in S \cup \Lambda $ and $l' \in \Lambda$ (note that it is just the transportation cost when $l \in \Lambda$).

\begin{table}[t]~\label{tab:notation}
\caption{Notation for the model.}
%\centering
\begin{minipage}{0.5\textwidth}
\resizebox{0.97\textwidth}{!}{%
\begin{tabular}{rl}
\hline 
\multicolumn{2}{l}{\bf{Indices and Sets}}\\
$l, z, t $, & a node (store or warehouses), a zone, a period,\\
$o, b$, & the online channel, the brick-and-mortar channel, \\
%$S$, & supplier node\\
%\multicolumn{2}{l}{\bf{Sets}}\\
  $\Lambda, Z$, & all nodes, all zones, \\
  ${\ca T}$, & all time periods $\{0 \cdots T-1 \}$ where $T$ is the planning horizon,\\
  $S$, & supplier node,\\
  \multicolumn{2}{l}{\bf{Parameters}}\\
$p^b_{tl}$, $b^b_{tl}$, & price, lost sales penalty for store $l$ sales in period $t$,\\
$p^o_t$,$b^o_t$, & price, lost sales penalty for online sales in period $t$,\\
$h_l$, & holding cost at node $l$ of left over inventory in a period, \\
\hline
\end{tabular}}
\end{minipage}%
\hfill
\begin{minipage}{0.5\textwidth}
\resizebox{0.95\textwidth}{!}{%
\begin{tabular}{rl}
\hline 
$c_{lz}$ & unit cost of fulfilling an online order in zone
              $z$ from node $l$,\\
$C_{l}$, & unit purchase \& transportation cost from supplier $S$ to node $l$,\\
%$C_{l'l}$, & per-unit transportation or purchase when shipping
%                      from node or supplier $l'$ to node $l$,\\
%  $I_{DC}$, & initial inventory at DC,\\
% $R^{L}_l, R^{U}_l$, & lower and upper bounds on inventory at node
%                        $l$, \\
 $L_{l}$, & lead time between supplier node $S$ to node $l$,\\
% $L_{ll'}$, & lead time between nodes $l$ and $l'$,\\
%  $L_{l}$, & maximum lead time for node $l$, in particular, $\max_{l' \in \Lambda \cup S} L_{l'l}$,\\
%$\overline{D}^b_{tl}$, & the nominal demand in store node $l$ in period $t$,
%  \\
% $\overline{D}^{oP}_{tl}$, & the nominal buy online pickup demand in node $l$ in period $t$, \\ 
%$\overline{D}^o_{tz}$, & the nominal demand in the online 
%                         zone $z$ and period $t$,\\
$\bar{I}_{l}^j$, & initial inventory pipeline for node $l$ arriving in $j$ periods,\\
  \multicolumn{2}{l}{\bf{Decision Variables}}\\
  $I_{tl}^j$, & inventory for node $l$ in period $t$ arriving in $j$ periods,\\
$x_{tl}$, & inventory allocated from supplier in period $t$ for node $l$,\\
$s^b_{tl}$, & fulfilled, lost demand from store $l$ in
              period $t$, \\
%$s^{oP}_{tl}$, ,  $u^{oP}_{tl}$, & fulfilled, lost online pick up in store demand from store
%                 $l$ in period $t$, \\
$s^{oS}_{tz}$,  & fulfilled, lost online demand from zone $z$ in period $t$, \\
$y_{tlz}$, & fulfilled online demand in zone $z$ from node $l$ in period $t$. \\
\hline
\end{tabular}}
\end{minipage}
%\end{center}
%}
\end{table}

\noindent {\bf Sequence of events:} 
%We assume the following sequence of events: 
The inventory in the pipeline arrives and the retailer places an order with the supplier. \textcolor{\highlightcolor}{For nodes with zero lead time, the order arrives immediately.} %simultaneously moves the inventory across its nodes. 
Then, the omnichannel demand realizes, after which the retailer fulfills it. Finally, the rewards and costs are incurred. 

\noindent {\bf Fulfillment assumptions:} Note that for the purpose of making inventory purchase and distribution decisions,  we assume that the fulfillment decisions are batched within each period. In practice, fulfillment is realized exogenously (to the model) based on customer arrivals and handled by the retailer’s order management system (OMS). This is an independent system that determines the fulfilling store or warehouses on-the-fly for e-commerce customer orders based on a variety of metrics including operational costs,  capacity balancing across nodes, stock outs, markdowns and can also involve complex and dynamic rules that depend on store performance, traffic, and capacity, local weather, and order splitting rules, among other factors~\citep{ali2017optimization}. %Furthermore, these rules may be frequently adjusted by OMS users. 
%In contrast to fulfillment decisions which is based on customer arrivals,  
Unlike fulfilment, inventory decisions follow a more regular cadence (say bi-weekly or weekly) and are usually dependent on the transportation schedule of trucks.
%that are a part of the larger ecosystem tied to the wide range of items sold by a retailer. 
%We refer the reader to the discussion in~\citet{OCPX} wherein when fulfillment was similarly exogenous, modeling fulfillment via batching was preferred to predicting it via a machine learning model. 

\noindent {\bf Optimization paradigm:} %The last aspect prior to formulating the decision model of the retailer is about the characterization of the demand uncertainty. 
To optimize the expected cost using stochastic programming, the future demand uncertainty has to be characterized via high dimensional probabilistic models and/or via complex scenarios generation techniques that characterize the dependence across the random variables. Due to its challenging nature and the dependence of the solution on inexact estimations of the true distribution that can adversely affect real-life performance, we adopt a framework that builds on the data-driven robust optimization (RO) paradigm. %Here uncertainties are characterized in an uncertainty set.  

\noindent {\bf Uncertainty set:} In the RO paradigm, the retailer aims to maximize the worst case total profitability assuming the demand, ${D}_{tl}^b$ and ${D}_{tz}^o $, are uncertain and comes from a pre-specified polyhedral set $U$ derived from data. %In practice, demands are non-negative integers, and hence we work with the discrete values of demand in this set. Later in this paper, we tap into this discrete demand structure for developing an exact global optimization approach to solving the problem. We also discuss relaxations for the continuous uncertainty sets. Therefore, we introduce both these sets. 
Specifically, we assume the following structure for $U$:% and its discrete counter-part $U_{\mathbb{Z}}$:
% \begin{align}
% U_{\mathbbm{Z}}  &:=
% \left\lbrace \ {\bf {D}}^b_{t}, {\bf {D}}^o_{t} \in \mathbbm{Z}^+ \  \forall \ t \in \ca{T} \ \middle| \
%   \begin{tabular}{ll}
%    ${\bf \bar{D}}^{mL}_{t}\leq {\bf {D}}^m_{t} \leq {\bf \bar{D}}^{mU}_{t}$ &
%                                                                     $\quad \forall \ m \in  \{o,b\}, \  t \in \ca{T}$\\
%    $ \bar{\bar{D}}^{mL}_{t}\leq {\bf e}^T {\bf {D}}^m_{t} \leq \bar{\bar{D}}^{mU}_{t}$ & $\quad \forall \ m \in  \{o,b\},  \  t \in \ca{T}$\\
%  %$ D^m_{L}\leq \sum_{t=1}^T {\bf e}^T {\bf {D}}^m_{t} \leq D^m_{U}$ & $\quad m \in {o,b}$\\
% %    $ D^m_{L}\leq \sum_{t=1}^T\sum_{m \in{o,b}} {\bf e}^T {\bf
% %    {D}}^m_{t}  \leq D^m_{Ut}$ & \\
%   \end{tabular}
% \ \right\rbrace,  \label{eqn:uncertainty_set}
% \end{align}
\begin{align}
U  &:=
\left\lbrace \ {\bf {D}}^b_{t}, {\bf {D}}^o_{t} \  \forall \ t \in \ca{T} \ \middle| \
  \begin{tabular}{ll}
   ${\bf \bar{D}}^{mL}_{t}\leq {\bf {D}}^m_{t} \leq {\bf \bar{D}}^{mU}_{t}$ &
                                                                    $\quad \forall \ m \in  \{o,b\}, \  t \in \ca{T}$\\
   $ \bar{\bar{D}}^{mL}_{t}\leq {\bf e}^T {\bf {D}}^m_{t} \leq \bar{\bar{D}}^{mU}_{t}$ & $\quad \forall \ m \in  \{o,b\},  \  t \in \ca{T}$\\
 %$ D^m_{L}\leq \sum_{t=1}^T {\bf e}^T {\bf {D}}^m_{t} \leq D^m_{U}$ & $\quad m \in {o,b}$\\
%    $ D^m_{L}\leq \sum_{t=1}^T\sum_{m \in{o,b}} {\bf e}^T {\bf
%    {D}}^m_{t}  \leq D^m_{Ut}$ & \\
  \end{tabular}
\ \right\rbrace,  \label{eqn:uncertainty_set}%\\
 % U_{\mathbb{Z}}  &:= \left\lbrace  {\bf {D}}^b_{t}, {\bf {D}}^o_{t} \in \mathbbm{Z}^+ \cap U\right\rbrace.
\end{align}
Here ${\bf e}$ refers to a vector of 1 and ${\bf {D}}^m_{t} $ refers to the vector of demand across locations (nodes or zones depending on channel $m$). 
The first constraint in (\ref{eqn:uncertainty_set}) restricts the channel-location-period specific demand to a box set and the second, referred to as the budget constraint, requires the channel-period specific total demand across locations to fall within a lower and an upper bound.
\textcolor{\highlightcolor}{The budget bounds are derived through quantile modeling of demand distributions at both local and hierarchical levels in a data-driven manner. For example, this can be achieved by learning from historical demand data using quantile-based models—an entirely data-driven approach. }
%The bounds are derived based on quantile modeling of the demand distributions at the local and at the hierarchical level in a data-driven fashion.
In our experiments with real-data, we developed channel specific multi-task demand models %\noteSS{@brian reference} 
across locations (or zones for e-commerce) with a focus on accuracy and coherency across the location hierarchy. \textcolor{\highlightcolor}{Multi-task models learn shared/common effects (correlations) across different time-series by utilizing common features~\citep{salinas2019deepar,oreshkin2020nbeats,makridakis2022m5,jati2023hier}.}
 %An advantage of these models is that in addition to local forecasts, we also obtain demand forecasts at different levels of the location hierarchy that we use to specify the uncertainty sets. 

%Such coherent forecasts that are consistent across different levels of the hierarchy increased the client trust in the forecast model. An additional advantage with multi-tasked models is that the typical correlation effects across locations and time periods are captured via common features \noteSS{@brian reference?}. \notePH{Reduce to 1 line to capture the value of multitask models}

In the absence of a hierarchical forecast, one can fall back on central-limit theorem based location-scale normalized budget constraints~\citep{bertsimas2004price, bandi2012tractable}. Unlike a hierarchical forecast, the location-scale method requires a variance estimator at the local level. % (see~\citealt{bandi2019sustainable} for an example in an inventory optimization context). %\citet{?} refers to the latter approach as explicit pooling and  the former approach that we implemented as implicit pooling. 
Regardless of the nature of explicit (former) or implicit (latter) pooling methods discussed, budget constraints allow the decision maker to be less conservative and focus on the most probable set of scenarios.
%, and either can be adopted. 
Additional cross-channel and/or multi-period budget constraints can also be included in practice. 
We work with a simple (effective and data-driven) uncertainty set for ease of exposition, and simple modifications suffice to incorporate alternative choices of budget-constrained sets.

\section{Omnichannel  robust inventory optimization model}~\label{model}
The adaptive multi-period RO model can now be formulated as follows:
% \begin{align}
% Z^{R}(\bar{\bf {I}}) = \max_{{\bf x}^t({\bf D}[0:t-1]) \in {\mathbb{Z}^+}} \   \min_{{\bf D}[0,T] \in U_{\mathbb{Z}} }\ 
%              \max_{\{{\bf s}_t,{\bf y}_t,{\bf I}_{t+1}\} \in F\left({\bf I}_t, {\bf D}_t, {\bf x}^t\right) } \ & \ 
%  \sum_{t \in \ca{T}}  O({\bf s}_t,{\bf y}_t,{\bf I}_{t+1},{\bf D}_t, {\bf x}^t) \\
% \eqlabel{initialInv} \textrm{subject to. } & \quad  {\bf I}_{0l} = \bar{{\bf I}}_{l} \qquad \forall  \ l \in \Lambda 
%    \end{align}
\begin{align}
Z^*_{R}(\bar{\bf {I}}) = \max_{{\bf x}_t({\bf D}[0:t-1]) \geq 0} \   \min_{{\bf D}[0,T]  \in U%_{\mathbb{Z}} 
} \ 
             \max_{\substack{\{{\bf s}_t,{\bf y}_t,{\bf I}_{t+1}\} \in\\ F\left({\bf I}_t, {\bf D}_t, {\bf x}_t({\bf D}[0:t-1])\right) }} \ & \ 
 \sum_{t \in \ca{T}}  O({\bf s}_t,{\bf y}_t,{\bf I}_{t+1},{\bf D}_t, {\bf x}_t({\bf D}[0:t-1])) \\
\eqlabel{initialInv} \textrm{subject to. } & \quad  {\bf I}_{0l} = \bar{{\bf I}}_{l} \qquad \forall  \ l \in \Lambda 
   \end{align}
where the objective is: %\notePH{I have a question here why don't prior papers not consider the dependence on prior fulfillment or states. For example, my positioning in period 2 will be based on my left over inventory. I know we do not handle it but we need it for formalism sake}
\begin{align}   
\label{objective} \nonumber  O({\bf s}_t,{\bf y}_t,{\bf I}_{t+1},{\bf D}_t, {\bf x}_t) = & 
   \sum_{l \in \Lambda} \Big[p^b_{tl}  s^b_{tl} - b^b_{tl} (D^b_{tl} -  s^b_{tl}) \Big] + \sum_{z \in Z} \left[
                 \sum_{l \in \Lambda}  (p_t^o -c_{lz} ) y_{tlz} -  b_t^o \left(D_{tz}^o - \sum_{l \in \Lambda} y_{tlz}\right)\right] \\
                 &  - \sum_{l \in   \Lambda}\left[  h_l I^0_{t+1,l} + %\sum_{l' \in \Lambda \cup S}  C_{l'l} x^t_{l'l}
                 C_{l} x_{tl}
                 \right] %\\
% \nonumber =     & 
%    \sum_{l \in \Lambda} \left[(p^b_{tl} + b^b_{tl}) s^b_{tl} - b^b_{tl} D^b_{tl} 
%               +   \sum_{z \in Z} (p_t^o +  b_t^o  -c_{lz} ) y_{tlz} \right]-  \sum_{z \in Z} b_t^o D_{tz}^o  \\
%                  &  - \sum_{l \in   \Lambda}\left[  h_l I^0_{t+1,l} + %\sum_{l' \in \Lambda \cup S}  C_{l'l} x^t_{l'l}
%                   C_{l} x_{tl}\right]     \qquad  \textrm{(by simplification)}       
                \end{align}
and the fulfillment set $ F({\bf I}_t, {\bf D}_t, {\bf x}^t) $ is the set $\{{\bf s}_t,{\bf y}_t,{\bf I}_{t+1}\} \geq 0 $ that satisfies the following constraints:
\begin{align}
%\left\lbrace \begin{tabular}{ll}
\label{walkinSales} & s^b_{tl} \leq {D}_{tl}^b 
                                 &&  \forall \ l \in \Lambda,\\
\label{ecomSales} &\sum_{l \in\Lambda} 
  y_{tlz} \leq {D}_{tz}^o  && \forall \    z \in Z,\\
%& s^b_{tl}  + \sum_{z\in Z}
%  y_{tlz} + I_{t+1,l}^0 = I_{tl}^0  + \bar{I}_{l}^{t+1}  \mathbbm{1}_{t<L_l}  + \sum_{l' \in \Lambda \cup S}  x^{t-L_{l'l}}_{l'l}\mathbbm{1}_{t \geq L_{l'l}}  -
%\sum_{l' \in \Lambda } x^t_{l l'}  \hspace{-0cm } & \forall \ l \in \Lambda\\
%\end{tabular} \right \rbrace\\
% \eqlabel{invBalance}&  s^b_{tl}  + \sum_{z\in Z}
%   y_{tlz} + I_{t+1,l}^0 = I_{tl}^0  + I_{tl}^1 + \sum_{l' \in \Lambda \cup S} x^t_{l'l}  \mathbbm{1}_{L_{l'l}=0} - \sum_{l' \in \Lambda } x^t_{l l'} && \forall \ l \in \Lambda,\\
% \eqlabel{stateupdate} & I^j_{t+1,l} = I^{j+1}_{tl}  \mathbbm{1}_{j < \max_{l'} L_{l'l}}+ \sum_{l' \in \Lambda \cup S} x^t_{l'l} \mathbbm{1}_{L_{l'l}=j} &&  \forall \ 1\leq j \leq \max_{l'} L_{l'l}, \ l \in \Lambda
\eqlabel{invBalance}&  s^b_{tl}  + \sum_{z\in Z}
  y_{tlz} + I_{t+1,l}^0 = I_{tl}^0  + I_{tl}^1 \mathbbm{1}_{L_{l} > 0} + x_{tl}  \mathbbm{1}_{L_{l}=0} && \forall \ l \in \Lambda,\\
\eqlabel{stateupdate} & I^j_{t+1,l} = I^{j+1}_{tl}  \mathbbm{1}_{j < L_{l}}+ x_{tl} \mathbbm{1}_{j = L_{l}} &&  \forall \ 1\leq j \leq L_{l}, \ l \in \Lambda
\end{align}
This model is an adaptive multi-stage optimization problem with 3 stages in each period: inventory purchase and distribution, followed by the adversary choosing the worst-case demand, and finally, the retail fulfillment. 
Observe that the inventory decision variables $x_{tl}$ are non-anticipatory and fully adjustable/adaptive (i.e., not static and depend on all prior realizations of demand ${\bf D}[0,t-1]$). The objective function (\ref{objective}) models the total profitability of the retailer from allocation, distribution and fulfillment. 
%The first and the third term represent the revenue from walk-in and online demands respectively \noteSS{first and second?- see comment in slack on reorg the description}. The second and the fourth term model the penalty from lost sales (unsatisfied demand) in the two channels respectively while the fifth term captures the holding cost from the left over inventory. Finally, the sixth term represents the purchase cost from supplier and/or the transportation cost to move inventory from one node to another. 
The first and second terms represents the profit from the store and e-commerce channels respectively and are obtained as the difference of the respective channel-specific revenue less the lost sales penalty cost. The third term sums the holding cost from the left over inventory, and the purchase cost from supplier and/or the transportation cost to move inventory from the supplier to the nodes. As far as the fulfillment constraint set, the constraints (\ref{walkinSales}-\ref{ecomSales}) ensures that the sales is less than the channel specific demand in every period and location. Constraint (\ref{invBalance}) enforces inventory balance for each location and time period. %In more detail, it states that the sum of the walk-in sales at a node, the e-commence fulfilled sales from the same node plus the left over inventory in that node is equal to the on-hand inventory plus that in the historical pipeline that arrived in that period in that node plus the inventory that will arrive based on decisions make in this problem less those that are shipped out of the same node. 
Finally, constraint (\ref{stateupdate}) updates the state (i.e., the inventory pipeline). 

%\noindent {\bf Business Rules:} 
The above formulation can be enhanced with business rules to capture retailer-specific operational requirements. %around inventory movement and fulfillment. 
For example, we can specify transportation capacities on $x_{tl}$ and fulfillment capacities on $y_{tlz}$ and create $y_{tlz}$ variables only for edges $lz$ originating from SFS eligible nodes.
%that are fulfillment activate (as not all stores are ship-from-store). 
If the retailer prefers warehouse fulfillment of online demand to stores, the cost $c_{lz}$ can be modified to reflect the additional savings, i.e., labor savings besides shipping cost. 
%\begin{henumerate}
% \item {\em Limitations on inventory movements:} In a pure supplier purchase setting or when movement of inventory only from warehouses to stores are allowed, or a combination of the above, the other links (e.g., store to store inventory movement) are redundant and need not be added to the model. There also maybe transportation capacity limits for these flows from certain nodes. \notePH{Should we modify the notation in the model to add only edges of interest? Also, add capacity constraints in original formulation and delete this} 
% \item {\em Fulfillment capacities and costs:} Some nodes may not be ship-from-store activated in which case, the corresponding $y_{tlz}$ need not be added to the model. We may also want to restrict the capacities of how much a node, for example, may be able to fulfill in a given period. \notePH{same comment as above and maybe reorder the list} Similarly, if the retailer prefers warehouse fulfillment of online demand to stores, the cost $c_{lz}$ can be modified to reflect the additional savings, i.e., labor savings besides just shipping cost. 
%Due to this reason, whenever there are warehouses and stores in the mix or even say stores in urban (i.e., Manhattan) vs. sub-urban areas, the fulfillment cost has complex structure that are not entirely distance based (a common assumption in the literature is an affine function of distance~\citep{??}). \notePH{Maybe this has to move elsewhere}
%\item {\em Fulfillment window goals:}
Some retailers set internal on-time delivery service targets, e.g., 80\% of all e-commerce fulfillment should be delivered in 2 days. Such a goal can be implemented as follows: $\sum_{\{lz\} \in A}y_{tlz} \geq 0.8 \sum_{l \in \Lambda, z \in Z} y_{tlz}$ where $A = \{l \in \Lambda, z \in Z| \textrm{delivery time between $l$ and $z$ is less than 2 days} \}$. 
%Here we impose it in every period and it can be modeled to the horizon of interest depending on the requirements. 
Such goals encourage the positioning of the inventory closer to the customer, an aspect that is important to consider in an omnichannel world. Today's customer has diverse delivery options including 2-hour delivery and curb-side pickup, and demand can immediately turn into lost sales if the product is not available to a customer within the specified time-window.
%\end{henumerate}
%

%Suppose all lead times are 0 and we are considering only a single-period, the

% Consider a single-period version of the problem with zero lead time. The 
% problem reduces to a two-stage robust optimization problem, which has been shown to be NP-hard~\citep{ben2004adjustable}. In a multi-period setting, the complexity only increases and to our knowledge, there is no efficient algorithm to solve the general fully adjustable version of the problem. Given this, we focus on developing practically effective heuristic approaches. 

For simplicity of exposition of the core {\em bimodal optimistic-robust} idea, we consider a static policy for the inventory decisions where the dependence of ${\bf x}_t$ on {demand ${\bf D}[0,t-1]$)} is dropped, %(note that this assumption does not impact immediate decisions ${\bf x}^0$ anyway), 
while keeping the adaptive nature of fulfillment decisions intact. The problem here reduces to a two-stage NP-hard problem~\citep{ben2004adjustable} that we solve via cut-generation techniques in Section~\ref{BendersSec}. 
%. We solve a practical variation of the problem using a combination of Benders Decomposition and mixed-integer programming techniques that we describe in more detail in Section~\ref{BendersSec}. 
\textcolor{\highlightcolor}{Note that the drawback of a static policy is its inability to adapt to new information in future periods, leading to myopic decisions. In contrast, while an adaptive policy is more responsive, it suffers from the curse of dimensionality when handling large state spaces. As mentioned in Section~\ref{settings}, decisions of the static policy are executed only for the current period, and the next static policy is recomputed in a rolling-horizon fashion, allowing for some adaptability. The optimistic-robust ideas proposed in this paper for static policies extend to alternative heuristic approaches, such as affine policies for adaptive decisions (see ~\citealt{ben2004adjustable}).}

\subsection{Static robust policy to omnichannel inventory management} For concreteness, the formulation is as follows:
\begin{align}
\label{staticRobustFormulation} Z^*_{SR}({\bf \bar{I}}) = \max_{{\bf x}_t \geq 0 } \   \min_{{\bf D}[0,T] \in U%_{\mathbb{Z}}
} \ 
             \max_{\{{\bf s}_t,{\bf y}_t,{\bf I}_{t+1}\} \in  F\left({\bf I}_t, {\bf D}_t, {\bf x}_t\right) } \ & \ 
 \sum_{t \in \ca{T}}  O({\bf s}_t,{\bf y}_t,{\bf I}_{t+1},{\bf D}_t, {\bf x}_t) %\\
% \textrm{subject to. } & \quad  I_{0l} = \bar{I}_{l}^0 &&  \hspace{-0cm } \forall % \ l \in \Lambda 
   \end{align}
% \begin{align}
% \label{staticRobustFormulation} Z^{SR}({\bf \bar{I}}) = \max_{{\bf x}^t \in {\mathbb{Z}^+}} \   \min_{{\bf D} \in U_{\mathbb{Z}}} \ 
%              \max_{\substack{\{{\bf s}_t,{\bf y}_t,{\bf I}_{t+1}\} \in\\ F\left({\bf I}_t, {\bf D}_t, {\bf x}^t\right) }} \ & \ 
%  \sum_{t \in \ca{T}}  O({\bf s}_t,{\bf y}_t,{\bf I}_{t+1},{\bf D}_t, {\bf x}^t) %\\
% % \textrm{subject to. } & \quad  I_{0l} = \bar{I}_{l}^0 &&  \hspace{-0cm } \forall % \ l \in \Lambda 
%    \end{align}
We simplify constraints \eqref{initialInv}, \eqref{invBalance}, and \eqref{stateupdate} by eliminating the inventory pipeline variables $I_{tl}^j$ for all $j$ by using the static inventory variables instead. Dropping the superscript $j$, we obtain:
\begin{align}
% \eqlabel{invBalance_new}&  s^b_{tl}  + \sum_{z\in Z}
%   y_{tlz} + I_{t+1,l} = I_{tl}  + \bar{I}_{l}^{t+1}  \mathbbm{1}_{t<L_l}  + \sum_{l' \in \Lambda \cup S} x^{t-L_{l'l}}_{l'l}\mathbbm{1}_{t \geq L_{l'l}} - \sum_{l' \in \Lambda } x^t_{l l'} && \forall \ l \in \Lambda,\\
  \eqlabel{invBalance_new}&  s^b_{tl}  + \sum_{z\in Z}
  y_{tlz} + I_{t+1,l} = I_{tl}  + \bar{I}_{l}^{t+1}  \mathbbm{1}_{t<L_l}  + x_{t-L_{l},l}\mathbbm{1}_{t \geq L_{l}}  && \forall \ l \in \Lambda,\\
 &  I_{0l} = \bar{I}_{l}^0 && \forall \ l \in \Lambda.
\end{align}
% \begin{proposition}~\label{convexity}
% The ${SR}({\bf \bar{I}})$ problem is convex in ${\bf x}_t$. Moreover, for any given ${\bf x}_t$, the optimal (worst-case) demands ${\bf D}$ are at the extreme points of $U$. \notePH{Need to prove}
% \end{proposition}

\noindent {\bf Insights into the solution structure:} %Consider a specific inventory decision $x$.
{\em Typical behavior of the static robust adversary (inner min-max problem) is to select skewed worst-case demands that do not align with inventory.} This can be gathered by observing the objective terms in Eq. \eqref{objective}. Barring the last term (noting that $x$ is fixed for the inner problem), each term results in a drop in profitability without a sale. Therefore, skewed inventory allocations $x$ are of little value in %robust optimization
the worst case as they result in no revenue and increased costs. So, {\em a typical robust solution tends to be a fair and uniform distribution of inventory across locations so that the adversary does not benefit from a skewed choice of demand.} This is a valuable feature of robust solutions %and how the model naturally induces uniformity and fairness, and is %
in contrast to their deterministic counterparts. This result was proved in~\citet{govindarajan2021distribution} for a special case.
% \begin{proposition}
% In a symmetric setting, the total optimal robust allocation quantity in~\ref{staticRobustFormulation} is step-wise decreasing in purchase cost. \notePH{Do we need this?}
% \end{proposition}

 Another aspect in the omnichannel environment is the {\em asymmetrical adversarial behavior that impacts the walk-in channel far more than the online channel}. Unlike walk-in demand, any online order can be satisfied using SFS fulfillment %(a recourse decision) 
 even when the inventory positioning is unfavorable. 
 %because regardless of bad positioning, the online demand can be fulfilled because of ship-from-store fulfillment (a recourse decision), while that is not the case for walk-in demand. 

\noindent {\bf Downsides of a robust solution:} {\em Robust models are overly focused on protecting downside and ignore potential upside and tend to undervalue a wide-range of allocations.} %For example, regardless of high levels of inventory at a location w.r.t other locations, the adversary, defined with Eq.~\eqref{eqn:uncertainty_set}, prefers to allocate 0 demand to a location, if it is feasible (lowest in general). 
Consider, for example, the robust model applied to 3 symmetric locations with just walk-in demand that satisfy a local budget of [0,3] in each location and a global bound of [1,6]. The adversarial sales with no initial inventory at all locations does not increase even when one of the locations has high inventory, i.e., allocation [0,0,0] has the same adversarial sales as allocation [0,0,3] for example (with [3,3,0] as the only adversarial sales in the latter and there are many more degenerate options including this for the former). In real-world data with several nodes, it is rare that the scenarios are so adversarial in nature that demand practically mis-aligns with inventory everywhere. On the contrary, with non-adversarial distributions {\em statistically there is a lower chance of lost sales in locations with higher inventory, which the robust model with its chosen objective and budget constraints fails to capture.} 
%{This results in endogenous outliers in the uncertainty set, an aspect that is magnified by the initial inventory variation across locations that eliminates degeneracy in worst-case solutions making it even harder to encounter.
{This results in endogenous outliers in the uncertainty set which are chosen as the optimal adversarial response. This outlier effect is magnified by the initial inventory variations across locations that increase the asymmetry and result in optimizing for adversarial scenarios that are rarely encountered. }
%eliminate any degeneracy in worst-case solutions making it even harder to encounter.}%\noteSS{the last part of previous sentence is not fully clear to me- does it mean that init.inv increases asymmetry and this results in a more 'unique' worst case scenario, i.e, rare? } %This particularly hurts the walk-in channel where there are no recourse fulfillment actions available. 
%The presence of initial inventory across locations (common in rolling horizon problems) magnifies this problem even more as post-hoc observing the worst-case scenarios is akin to searching for a needle in the haystack. %5(e.g., very low chance of zero walk-in demand when the inventory allocation is very high). %, and our inability to capture the lower chance of lost sales, can result in very conservative solutions.

We therefore propose an alternative RO approach that uses the idea of an `allied-adversary' that incorporates the optimism associated with a increase sell-through that can occur with higher allocations. We continue to work within the RO framework %(over stochastic programming) 
for ease of model specification and to avoid the solution dependence on several uncertain probability estimates across the network.

\subsection{Optimistic-robust optimization: An ally-adversary model}

A high-level idea in Kelly’s approach that motivates our model is to view inventory purchase from a supplier and its positioning as a sequence of investment contributions to a portfolio of store locations. Here the decision maker not only positions herself for the adversarial worst case but simultaneously explores profitable inventory purchase and positioning decisions that boost cumulative (expected) gain. %by injecting a degree of ‘optimism’ in her decision making.%
Toward this goal, we extend the aforementioned distribution-free unimodal RO model %described above that focus entirely on adversarial demands and 
and propose a distribution-free \textit{bimodal} RO formulation that also analyzes the most favorable store demands that are aligned with inventory allocations within the uncertainty set.

A rich variety of bimodal objective functions and models can be analyzed in such an ally-adversary setting. We take one step in this direction by demonstrating that by treating the degree of optimism, that we denote by $\lambda$, as a machine-learning hyper-parameter, one can gainfully identify a data-driven tradeoff between the best- and worst-case demand scenarios. % \noteSS{with the new results, we can finally identify a tradeoff between SAA and pure RO, similar to original Kelly}. 
The resultant demand pattern is obtained based on the chosen objective and the decision-maker’s degree of optimism $\lambda$ (akin to Kelly's investment fraction $f$) and will lie in between the two extremes of best- and worst- case demands. We describe our canonical approach below. See Appendix~\ref{extensions} for alternative ways to exploit a location-specific information edge. 

Let us denote the allied best-case demand by ${\bf D}^+$ and the adversarial worst-case demand by ${\bf D}^-$ respectively, both in set $U$, in response to an inventory decision $x$. We use the optimistic hyper-parameter $\lambda$ to identify the net demand at any location as a convex combination of best-case and worst-case demand. As $U$ is a polyhedral set, the resultant demand scenario is also in set $U$. In the proposed (static) Bimodal Inventory Optimization Model (BIO-$\lambda$), we identify the inventory decisions along with the optimistic demands jointly in the first stage. In the second stage, we solve a sub-problem similar to RO to identify the adversarial demands for the given allocations and optimistic demands at the fixed $\lambda$ parameter, and in the third stage, calculate the fulfilment decisions to the resultant demand. We formulate BIO-$\lambda$ as follows.
\begin{align}
\label{bioformulation} Z^*_{SBIO - \lambda}({\bf \bar{I}}) = \max_{{\bf x}_t \geq 0 , {\bf {D}}^+[0,T] \in U%_{\mathbb{Z}}
} \  \min_{{\bf D}^-[0,T] \in U%_{\mathbb{Z}}
} 
             \max_{\substack{\{{\bf s}_t,{\bf y}_t,{\bf I}_{t+1}\} \in\\ F\left({\bf I}_t, \lambda {\bf D}^+_t + (1-\lambda) {\bf D}^-_t, {\bf x}_t\right) }}  &  
 \sum_{t \in \ca{T}}  O({\bf s}_t,{\bf y}_t,{\bf I}_{t+1},\lambda {\bf D}^+_t + (1-\lambda) {\bf D}^-_t, {\bf x}_t) %\\
% \textrm{subject to. } & \quad  I_{0l} = \bar{I}_{l}^0 &&  \hspace{-0cm } \forall % \ l \in \Lambda 
   \end{align}
% \begin{align}
% Z^{DRO}({\bf I}) = \max_{{\bf x}^t \in {\mathbb{Z}^+}} \   \min_{, {\bf {D}}^+[0,T], {\bf {D}}^-[0,T] \in U_{\mathbb{Z}}, \lambda \in [0,1]} \ 
%              \max_{\substack{\{{\bf s}_t,{\bf y}_t,{\bf I}_{t+1}\} \in\\ F\left({\bf I}_t, \lambda {\bf D}^+_t + (1-\lambda) {\bf D}^-_t = \bar{D}, {\bf x}^t\right) }} \ & \ 
%  \sum_{t \in \ca{T}}  O({\bf s}_t,{\bf y}_t,{\bf I}_{t+1},\lambda {\bf D}^+_t + (1-\lambda) {\bf D}^-_t, {\bf x}^t) %\\
% % \textrm{subject to. } & \quad  I_{0l} = \bar{I}_{l}^0 &&  \hspace{-0cm } \forall % \ l \in \Lambda 
%    \end{align}

 If $\lambda$ is zero, we retrieve the original adversarial RO model (i.e., RO/pure RO is the same as BIO-0). From a computational perspective, static BIO-$\lambda$ is identical in complexity to the static pure RO problem. In particular, the optimistic part of the problem (finding ${\bf D}^+$) is a relatively easy problem and  
%The reminder of the model is similar to the robust models described above. 
is jointly solved with allocation, and requires no additional information beyond what is already specified for the pure adversarial part of the problem. Section~\ref{BendersSec} describes our solution strategy for the BIO model.

For a given allocation $x$, the best-case outcome for problem $BIO-1$ occurs when the demand is maximally aligned with the inventory distribution across locations (demand as close to the inventory levels as possible)
and the worst-case is when demand is maximally mis-aligned with the inventory (i.e., high demands when there is little or no inventory and vice-versa). See Appendix~\ref{closedFormSingle}, for a closed form solution of BIO-0 and BIO-1 for a single walk-in location setting. The following theorem provides insight into the structure of the ${BIO-\lambda}$ solution and connects the best- and worst-case solution via a $\lambda$-superposed line segment. 
\textcolor{\highlightcolor}{
\begin{theorem}~\label{superpositionTheorem}
Given $\lambda$, there exists an $x^*_{BIO-\lambda} $ such that $x^*_{BIO-\lambda} = \lambda x^*_{BIO-1} + (1-\lambda)x^*_{BIO-0}$ and the corresponding optimal objective value $Z^*_{BIO-\lambda}$ satisfies $Z^*_{BIO-\lambda} = \lambda Z^*_{BIO-1} + (1-\lambda)Z^*_{BIO-0}$. In other words, for any given $\lambda$, there exists an optimal allocation for $BIO-\lambda$ that is a linear superposition of the respective optimal quantities for the best case scenario (when $\lambda = 1$, i.e., $BIO-1$) and the robust scenario (when $\lambda = 0$, i.e., $BIO-0$). Moreover, the optimal objective value follows  the same superposition structure. 
\end{theorem}}

This theorem shows that %when initial inventory is 0, 
the ${BIO-\lambda} $ problem can be decomposed into independent best- and worst-case problems that could be easily recombined. This suggests an alternative computational strategy of first solving the independent problems (just once) and thereafter one can simply superpose these solutions for any $\lambda$. Table~\ref{tab:example0} summarizes the results for a 3-location network having only walk-in demands that empirically demonstrate that the optimal allocation to the BIO-$\lambda$ problem is a superposition of the best and worst cases. %\footnote{To solve this problem, we make use of the fact that for the worst case uncertainty, it suffices to consider the extreme points of $U$ (see Section~\ref{BendersSec}) which are enumerated in this small example. } 
\begin{table}[t]
\centering
\scalebox{0.75}{
% %\arrayS{1.1}
\begin{tabular}{@{}cccccccccc@{}}
\toprule[1.0pt]
Setting	& \phantom{a} &	\multicolumn{2}{c}{p = 0, b = 160}						& \phantom{a} &	\multicolumn{2}{c}{p = 160, b = 0}	& \phantom{a} &	\multicolumn{2}{c}{p = 80, b = 80}	\\\cmidrule[0.5pt]{3-4} \cmidrule[0.5pt]{6-7} \cmidrule[0.5pt]{9-10}
Method	&&	Allocation	&	Estimated Obj	&&	Allocation			&	Estimated Obj	&&	Allocation				&	Estimated Obj	\\\cmidrule[0.5pt]{1-1}\cmidrule[0.5pt]{3-4} \cmidrule[0.5pt]{6-7} \cmidrule[0.5pt]{9-10}
Pure RO	&&	[3,3,3]	&	-360	&&	[1,1,1]	&	40	&&	[1.75,1.75,1.75]	&	-130	\\
Optimistic	&&	[1,0,0]	&	-40	&&	[3,3,0]	&	720	&&	[3,3,0]	&	240	\\
BIO-25\%	&&	[2.5,2.25,2.25]	&	-280	&&	[1.5,1.5,0.75]	&	210	&&	[2.0625,2.0625,1.3125]	&	-37.5	\\
BIO-50\%	&&	[2,1.5,1.5]	&	-200	&&	[2,2,0.5]	&	380	&&	[2.375,2.375,0.875]	&	55	\\
BIO-75\%	&&	[1.5,0.75,0.75]	&	-120	&&	[2.5,2.5,0.25]	&	550	&&	[2.6875,2.6875,0.4375]	&	147.5	\\
%Stochastic Prog	&&	[2,2,2]	&	-240	&&	[2,2,2]	&	520	&&	[2,2,2] &	140	\\
\bottomrule[1.0pt]
\end{tabular}
}
\caption{Example of Theorem 1 in a symmetric 3-location walk-in only setting with $L=0$, $h=0$, $c=40$ and $U = \{(d_1,d_2,d_3)| 0\leq d_i \leq 3 \ \forall \ i = 1,2,3; \textrm{ and } 1\leq d_1+d_2+d_3 \leq 6 \}$.} 
\label{tab:example0}
\end{table}

%The  superposition observation in Theorem~\ref{superpositionTheorem} on the objective $Z^*_{BIO-\lambda}$ can be viewed as an in-sample result \noteSS{in and out of sample not defined yet..? - use Monte-Carlo term instead}. 
The questions that remain are (1) how does a BIO solution for a chosen $\lambda$ perform in a Monte-Carlo simulation and (2) how should this performance be measured? %The following set of corollaries establish a connection between the BIO objective to commonly known objectives of interest.
%
%
%Regardless of the initial inventory, the objective of BIO $Z^*_{BIO-\lambda}$ steadily increases with $\lambda$ as more demand is allied, but the out-of-sample performance, say  starts to deteriorate.  
%
%
%
% We begin with the sample average approximation (SAA) objective, wherein we replace the adversarial minimum in (\ref{staticRobustFormulation}) to an average over samples in set $U$. The corollaries provide an alternative way of arriving at an SAA optimal solution that maximizes the sample mean.
%
% \begin{corollary}~\label{SAAbiolambda}
% If the true distribution of ${\bf D}$ is restricted to the uncertainty set $U$, in general, we know that $Z^*_{BIO-0} \leq Z^*_{SAA} \leq Z^*_{BIO-1}$. So, %when $\bar{I}_{l}^0 = 0 \ 
% $\forall \ l \in \Lambda$, $\exists$ a $\lambda$ such that $Z^*_{BIO-\lambda} = Z^*_{SAA}$ and  $x^*_{BIO-\lambda} = x^*_{SAA}$.
% \end{corollary}
\color{\highlightcolor}
The above theorem suggests that we can perform a one-dimensional search over $\lambda \in [0,1]$, where each value corresponds to a BIO-$\lambda$ allocation, ensuring an approach that maintains data-driven and distribution-free properties. 
%Using an out-of-sample data set from the uncertainty set $U$, we can identify the sample optimal $\lambda$ for a user-specified objective (e.g., SAA) and determine the corresponding sample-optimal allocation. 
We formalize this result as a corollary without proof and is immediate from Theorem~\ref{superpositionTheorem}. %\footnote{This may be a 1-dimensional grid search and therefore may result only in a near optimal solution.} 
%That this one-dimensional search, especially with a convex objective like SAA is a convex optimization problem, and in fact a bisection search on the $\lambda \in [0,1]$ line segment will suffice. We state this as a corollary. 

% This corollary suggests that we can perform a one-dimensional search on $\lambda \in [0,1]$ (each of which gives us a BIO-$\lambda$ allocation) and with an out-of-sample data set from the uncertainty set $U$, we can identify the sample optimal $\lambda$ and hence the corresponding sample-optimal SAA allocation. %\footnote{This may be a 1-dimensional grid search and therefore may result only in a near optimal solution.} 
% Next we show that this one dimensional search, especially with a convex objective like SAA is a convex optimization problem, and in fact a bisection search on the $\lambda \in [0,1]$ line segment will suffice.

% \begin{corollary}~\label{SAAbiolambda_optLambdascoring}
% Given a data set from $U$, identifying the sample optimal $\lambda_{SAA}$ such that $Z^*_{BIO-\lambda_{SAA}} = Z^*_{SAA}$ is a single dimensional convex optimization problem using $x^*_{BIO-0}$ and $x^*_{BIO-1}$.% when $\bar{I}_{l}^0 = 0 \ \forall \ l \in \Lambda$. 
% \end{corollary}

\begin{corollary}~\label{biolambda_optLambdascoring}
Given a data set from $U$, identifying the sample optimal $\lambda^*$ that optimizes for a user-specified objective function (e.g., SAA) reduces to a single-dimensional optimization problem using $x^*_{BIO-0}$ and $x^*_{BIO-1}$. For a convex objective, it reduces to a bisection search on the $\lambda \in [0,1]$ line segment. % when $\bar{I}_{l}^0 = 0 \ \forall \ l \in \Lambda$. 
\end{corollary}

Note that this approach may not yield the globally optimal solution for the user-specified objective, as achieving true optimality may require knowledge of the underlying distribution or access to a large dataset or may need a location-specific $\lambda$ describe in Appendix~\ref{extensions}. Instead, it offers a simple, data-driven, distribution-free method that approximately optimizes the chosen objective by assessing out-of-sample performance on a relatively small validation set. This stands in contrast to the larger training dataset typically required to compute an exact SAA solution, for example, that maximizes mean returns~\citep{kleywegt2002sample,bertsimas2018robust}.
Using BIO, one could aim to balance optimistic and pessimistic (conservative) scenarios via out-of-sample analysis—enhancing average-case performance while preserving robustness.

\color{black}

The proposed BIO methodology is not restricted to inventory problems and can be applied to any compatible scarce resource allocation problem where RO is viable while noting that the above results may need to be validated in other settings. 

 Next we expand on the BIO-$\lambda$ formulation. In the general case, we will have to introduce allied demands in all channels as discussed in the formulation (\ref{bioformulation}). Noting the lop-sided {impact of the adversary on the channels, we limit the allied demands to  the walk-in channel only as it does not have a recourse opportunity. 
 As our chosen uncertainty set $U$ can be decomposed by channel, we are easily able to apply this restriction. } 
 %We are able to easily satisfy this assumption as our uncertainty set $U$ could be decomposed by channel. 
 The resultant BIO formulation exhibits a multi-class demand structure as shown below. 
\begin{align}
\nonumber Z^{SBIO - \lambda}({\bf I}) & = \max_{(x, s^+) \geq 0, {\bf {D}}^+ \in U%_{\mathbb{Z}
} \ \min_{{\bf {D}}^- \in U%_{\mathbb{Z}} 
} \ 
             \max_{(s^-, y,I)\geq 0}  \ 
 \sum_{t \in \ca{T}}
  \sum_{l \in \Lambda} \Big[(p^b_{tl} + b^b_{tl}) (s^{b+}_{tl} + s^{b-}_{tl} ) - b^b_{tl} (\lambda D^{b+}_{tl} +  (1-\lambda) D^{b-}_{tl})
               \Big] &&\\
\label{bioFormulation}  &  +  \sum_{t \in \ca{T}}\sum_{z \in Z} \left[(p_t^o +  b_t^o  -c_{lz} ) y_{tlz} - b_t^o D_{tz}^{o-} \right] 
                   - \sum_{t \in \ca{T}} \sum_{l \in   \Lambda}\left[  h_l I^0_{t+1,l} + C_l x_{tl}%\sum_{l' \in \Lambda \cup S}  C_{l'l} %x^t_{l'l}
                   \right] \\
%\nonumber Z^{BIO}({\bf I})  = \max_{\substack{x \in {\mathbb{Z}^+}, {\bf {D}}^+ \in U %\\ s^+,u^+ \geq 0} } &\ \min_{{\bf {D}}^- \in U } \ 
%             \max_{\substack{s^-,u^- \geq 0,\\y,I\geq 0}}  \ 
% \sum_{t \in \ca{T}} 
% \sum_{l \in \Lambda} \left[
%   p^b_{tl} (s^{b+}_{tl} +  s^{b-}_{tl})             + \sum_{z \in Z}
%                  (p_t^o -c_{lz} )y_{tlz}  \right] - \sum_{t \in \ca{T}} \sum_{z \in Z}  b_t^o u_{tz}^o &&\\
%& \qquad  - \sum_{t  \in \ca{T}} \sum_{l \in   \Lambda}\left[b^b_{tl} %(u^{b+}_{tl}+u^{b-}_{tl}) + \sum_{l' \in \Lambda \cup S}  C_{l'l} x^t_{l'l}
%                + h_l I_{t+1,l}\right]&&\\
\nonumber \textrm{subject to. } 
\ & s^{b+}_{tl} \leq \lambda D_{tl}^{b+} 
                                 && \hspace{-4cm} \forall \ l \in \Lambda, t \in \ca{T}\\
& s^{b-}_{tl}  \leq  (1-\lambda) D_{tl}^{b-} 
                                 && \hspace{-4cm} \forall \ l \in \Lambda, t \in \ca{T}\\
& \sum_{l \in
  \Lambda} y_{tlz} \leq D_{tz}^{o-}  && \hspace{-4cm} \forall \    z \in Z, t \in \ca{T}\\
& s^{b+}_{tl}+ s^{b-}_{tl}  + \sum_{z\in Z}
  y_{tlz} + I_{t+1,l} = I_{tl}  + \bar{I}_{l}^{t+1}  \mathbbm{1}_{t<L_l}  +  x_{t-L_{l},l}\mathbbm{1}_{t \geq L_{l}}  &&\hspace{-4cm} \forall \ l \in \Lambda,                                        t \in \ca{T}\\
% & s^{b+}_{tl}+ s^{b-}_{tl}  + \sum_{z\in Z}
%   y_{tlz} + I_{t+1,l} = I_{tl}  + \bar{I}_{l}^{t+1}  \mathbbm{1}_{t<L_l}  + \sum_{l' \in \Lambda \cup S} x^{t-L_{l'l}}_{l'l}\mathbbm{1}_{t \geq L_{l'l}} -\sum_{l' \in \Lambda } x^t_{l l'} \nonumber \\ & && \hspace{-5cm}\forall \ l \in \Lambda,
%                                               t \in \ca{T}\\
 \label{bio_initialInv} & I_{0l} = \bar{I}_{l}^0 && \hspace{-4cm} \forall \ l \in
                               \Lambda.
\end{align}
%Compared to the pure RO model, in the BIO model (because of the above results), without loss of generality we separate sales constraint for the optimistic and pessimistic part of the demand vector and the resultant optimistic sales and pessimistic sales and lost sales terms in the objective, treating them as multi-class demands. 
%Based on the above results, here the BIO model is formulated with a multi-class demand structure. 
We write separate sales constraints for the optimistic and pessimistic part of the demand vector by creating corresponding optimistic and pessimistic terms in the objective function.% for sales and lost sales. 

 Next, we provide comparative insight into the optimal solution structure of 
the (static) unimodal RO strategy versus its bimodal counterpart using a simple example  (leveraging the cut-generation approach in Section~\ref{BendersSec}). 
We constrain the allocations to be integer-valued in these experiments. While the alternative scoring objectives continue to lie in between the best and worst case, we cannot fully exploit the results of Theorem~\ref{superpositionTheorem} as the allocation space is now discrete. We simply re-solve the BIO-$\lambda$ problem for varying $\lambda$ and use out-of-sample data to identify a preferred $\lambda$. 

\begin{example}~\label{example}
Consider an omnichannel retailer with nodes, zones, fulfillment costs, uncertainty set and the true unknown demand distributions as depicted in Fig.~\ref{fig:example}. Suppose also $L=0$.
\begin{figure}[t]
  \begin{center}
     \includegraphics[scale=0.5]{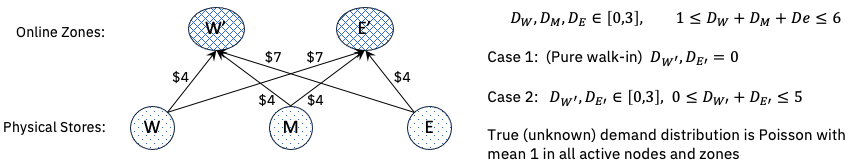}
 \caption{Example of an omnichannel retailer.}
 \label{fig:example}
\end{center}
\end{figure}

In Table~\ref{tab:example1}, we summarize the results (allocations, estimated objectives %\noteSS{achieved model obj or estimated from MC simul.?- i guess both as we have a lot more columns?} 
and their corresponding worst case/best-worst case demands scenarios) of the pure walk-in setting described in case 1 of Fig.~\ref{fig:example} when  $p = h = 0$, $b=160$ and $c = \$40$ for the pure RO and BIO-50\% methods. %Specifically, we include the recommended allocations, the estimated objectives and their corresponding worst case/best-worst case demands in the respective settings. 
Observe how the pure RO model invests in a lot inventory (9 units) to maximally reduce lost sales penalty, while the BIO model invests just 6 units.
Table~\ref{tab:example1} also includes the realized statistics of the objective obtained from a Monte Carlo simulation run on 1000 samples. Note that the estimated worst case and the simulated worst case profitability values do not match as the robust model assumes a finite polyhedral set to capture the most commonly occurring scenarios, as opposed to the unbounded Poisson distribution used by the simulation engine. Observe how the BIO model trades off worst-case performance (7.7\%, 11.1\% at the 5th, 10th quantile) compared to pure RO to improve the average case (21.8\%). %\noteSS{is 7.7pc compared to what baseline? 21.8pc gain over RO? how can we get this number from table-3 or fig2? 
%\noteSS{Also Fig.3 caption can be simplified/clarified better and maybe even mentioned in the main text}
\begin{table}[t]
\centering
\scalebox{0.7}{
%\arrayS{1.1}
\begin{tabular}{@{}cccccccccccc@{}}
  \toprule[1.0pt]
\multirow{2}{*}{Method}	& Reco. & Est. best-	& Est. Worst & Est. Worst & Est. &Rlze. & Rlze. 5th &	Rlze. 10th	& Rlze. & Rlze.	& Rlze. \\
	& Allocation	& -worst case	& Case	&  Case Obj & Obj	& Min.	& percentile &	percentile	& Median	& Max.	& Mean\\
\midrule[0.5pt]
Pure RO	& [3,3,3] & -	&	[3,3,0]	&-360 & -360	&-1000	&-520	&-360&	-360 &	-360	&-372.32\\
BIO-50\%	&[2,2,2]	& [0,0,1]-[3,3,0] &	[3,3,0]	&-560& -240	& -1040 &	-560	& -400	& -240	& -240	&-291.04\\ \bottomrule[1.0pt]
\end{tabular}}
\caption{Example of the recommend allocation and profit statistics of pure RO and BIO in the walk-in only setting. Allocation (alternatively, demands) are specified for nodes in the following sequence [M,W,E]. Expansion of the  acronyms used in the table include Reco. for Recommended, Est. for Estimated and Rlze. for Realized. } 
\label{tab:example1}
\end{table}
\begin{figure}[h]
  \begin{center}
     \includegraphics[scale=0.4]{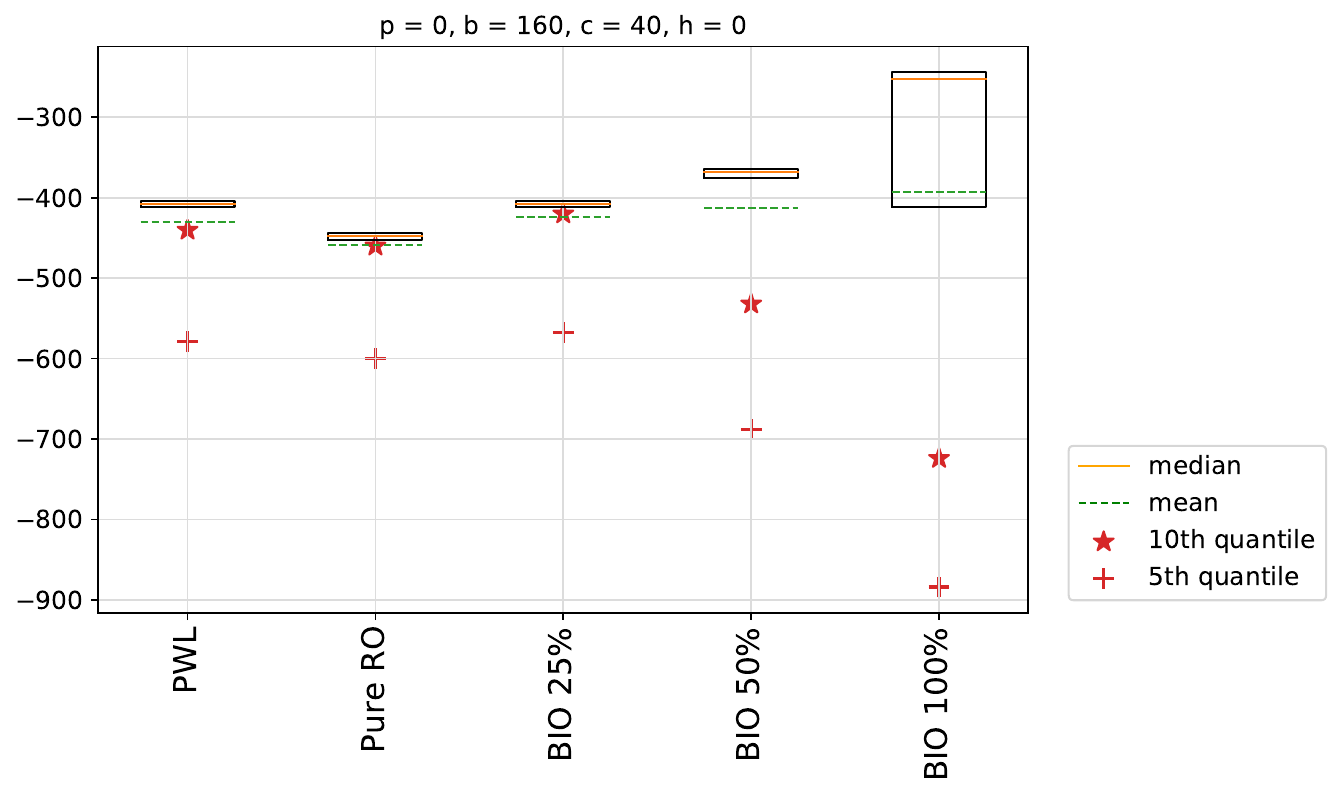}
    \includegraphics[trim = 0cm 0cm 5cm 0cm, clip,
     scale = 0.4]{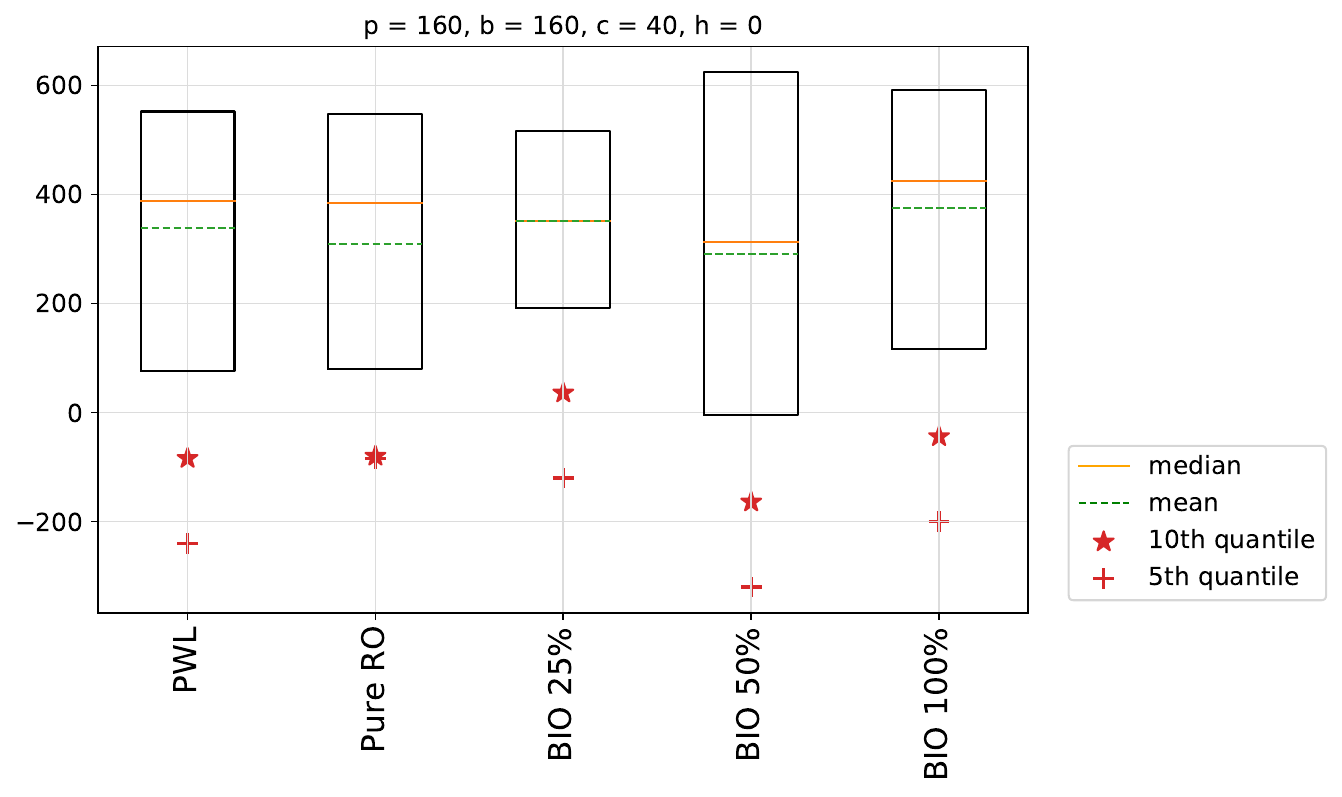}%\\
 \caption{The realized profitability box plots across different methods in case 2 of example 1.}
 \label{fig:boxplot_example2}
\end{center}
\end{figure}
\begin{table}[h]
\centering
\scalebox{0.67}{
%\arrayS{1.1}
\begin{tabular}{@{}lcccccccccccc@{}}
  \toprule[1.0pt]
Setting	& &	\multicolumn{5}{c}{ p = 0, b = 160, c = 40, h = 0}										& &	\multicolumn{4}{c}{p = 160, b = 160, c = 40, h = 0}									\\ \cmidrule[0.5pt]{3-7} \cmidrule[0.5pt]{9-13} 
Method	& &	PWL	&	%SAA	&
Pure RO	&	BIO 25\%	&	BIO 50\%	&	BIO 100\%	& &	PWL	&	%SAA	&
Pure RO	&	BIO 25\%	&	BIO 50\%	&	BIO 100\%	\\ \cmidrule[0.5pt]{3-7} \cmidrule[0.5pt]{9-13}
Reco. Allocation	& &	[2,4,4]	&	%[2,2,3]	&
[5,3,3]	&	[4,3,3]	&	[5,2,2]	&	[3,1,2]	& &	[2,4,4]	&	%[3,3,3]	&	
[2,2,2]	&	[2,3,2]	&	[4,3,1]	 & [3,3,3]\\
\multirow{2}{*}{Est. Best-Worse Case} & &	\multirow{2}{*}{-}	&	%\multirow{2}{*}{-}	&
\multirow{2}{*}{-}	&	[0,.5,.5]-	&	[0,1,0]-	&	[0,1,0]- 	& &	\multirow{2}{*}{-}	&	%\multirow{2}{*}{-}	&	
\multirow{2}{*}{-}	&	[2,3,1]- 	&	[3,3,0]- 	& [3,3,0]-\\
	& &		%&
 &		&	[3,3,0][3,2]	&	[3,3,0][3,2]	&	[1,1,1][3,2]	& &		&		%&
 &	 [0,1,0][0,0]	&	 [0,0,1][0,0]	& [1,1,1][0,0]\\
Est. Worst Case	& &	[3,3,0][3,2]	&	%[3,3,0][3,2]	&
[3,3,0][3,2]	&	[3,3,0][3,2]	&	[0,3,3][3,2]	&	[0,3,3][3,2]	& &	[3,0,0][0,0]	%&	[0,1,0][0,0]	
&	[3,0,3][3,2]	&	[0,1,0][0,0]	&	[0,0,3][0,0] & [0,1,0][0,0]	\\
Est. Worst Case Obj	& &	-586	%&	-935
&	-463	&	-579	&	-700	&	-1052	& &	-240	&	%-200	&	
-88	&	-120	&	-480	& -200\\
Rlze. Minimum	& &	-1551	&	%-1896	&
-1420	&	-1540	&	-1652	&	-2008	& &	-400	&	%-360	&	
-1052	&	-776	&	-640	& -360\\
Rlze. 5th percentile	& &	-579	&	%-768	&
-600	&	-568	&	-688	&	-884	& &	-240	&	%-200	&
-84	&	-120	&	-320	& -200\\
Rlze. 10th percentile	& &	-547	&	%-608	&
-460	&	-420	&	-532	&	-724	& &	-84	&	%-44	&
-80	&	36	&	-164	& -44 \\
Rlze. Mean	& &	-430	&	%-362	&
-459	&	-423	&	-413	&	-393	& &	338	&	%374	&
309	&	351	&	290	 & 374\\
Rlze. Median	& &	-408	&	%-288	&
-448	&	-408	&	-368	&	-252	& &	388	&	%424	&
385	&	352	&	312	& 424\\  \bottomrule[1.0pt]
\end{tabular}}
\caption{The recommend allocation and profit statistics across different methods in case 2 of example 1. Allocation/demands are specified for nodes/and zones in the sequence [M,W,E]/[M,W,E][W, E] respectively.} 
\label{tab:example2}
\end{table}

In Table~\ref{tab:example2}, we provide results for a more general scenario with e-commerce demands, case 2 in Fig.~\ref{fig:example}, when (1) $p = 0$, $b=160$ and (2) $p = b = 160$, where $c= 40$ and  $h = 0$. Besides the pure RO, BIO-$\lambda$ for a few different values of $\lambda$, we also include a baseline method, PWL which is a deterministic piece-wise linear network model described in Section~\ref{baselines}. %Observe how the worst case shifts to low demand value when the allocations are higher when $p=160$. Observe also how the best case in the best-worst case pair picks the minimum budget when $p=0$, but aligns with inventory when $p=160$. 
For ease of visualization, in Fig.~\ref{fig:boxplot_example2} we plot the realized profitability from the Monte Carlo simulation. We observe that a BIO-25\% improves the mean profitability and the robustness by 7.8\% and 5.3\% (at 5th quantile) respectively over pure RO when $p=0$ and trades off profitability against robustness (improves former by 13.5\% and drops latter by 42.8\% but changes direction and improves by 45\% at the 10th quantile) when $p=160$. In both these examples, we observe that the PWL baseline is competitive and outperforms pure RO across many metrics (5th, 10th quantiles and average case metrics) while BIO 25\% outperforms PWL and RO across all metrics. We also observed this behavior in real client data and is one of the findings that motivated the development of a BIO model. 
%Similar to the walk-in setting because we work with un-bounded Poisson distribution, the realized minimum is quite far from the estimated worst case objective, while the latter is close to the practical worst case of interest. 
These results suggest that the benefit of optimism can increase with higher $p$ when the other parameters remain the same. The non-convex behavior in the mean profitability between BIO-50\% and BIO-100\% when p=$160$ is attributed to the discreteness constraint on allocation. 
\end{example}

\section{Column and Cut Generation (CCG) Method to solve BIO-$\lambda$}\label{BendersSec}
We reformulate the inner min-max problem of BIO-$\lambda$ as a minimization problem by taking the dual of the innermost maximization problem: 
\begin{align}
\min_{{\bf {D}^-} \in U%_{\mathbb{Z}}
, \alpha, \beta \geq
  0, \gamma}
\ \nonumber & \sum_{t \in \ca{T}, l \in \Lambda} (1-\lambda)(\alpha_{tl}-b^b_{tl})
                  {D}_{tl}^{b-} + \sum_{t \in \ca{T}, z \in Z}
                  (\beta_{tz} - b^o_t) {D}_{tz}^{o-} + \sum_{l \in
                  \Lambda} \gamma_{0l}\bar{I}_{l}^0 \\
 \label{alphasub} & + \sum_{t \in \ca{T},l \in
                  \Lambda} \gamma_{tl} \left(\bar{I}_{l}^{t+1}
    \mathbbm{1}_{t<L_l} + x_{t-L_{l},l} \mathbbm{1}_{t   \geq  L_{l}} 
%    + \sum_{l' \in \Lambda \cup S} x^{t-L_{l'l}}_{l'l} \mathbbm{1}_{t   \geq  L_{l'l}}  -\sum_{l' \in \Lambda } x^t_{l l'} 
- s^{b+}_{tl}\right)\\
  \textrm{subject to. \ } \ & 
  \alpha_{tl} + \gamma_{tl} \geq p^b_{tl} + b^b_{tl}  && \hspace{-3cm}\forall t \in \ca{T}, l \in
                                               \Lambda\\
 \label{betasub}  &  \beta_{tz} + \gamma_{tl} \geq p^o_t + b^o_{t}- c_{lz}  &&\hspace{-3cm} \forall t \in \ca{T}, l \in
                                            \Lambda, z \in Z\\
 \label{gammasub}                 & \gamma_{tl}-\gamma_{t+1,l} \geq -h_l && \hspace{-3cm}\forall t \in \ca{T}, l
                                                          \in
                                                          \Lambda\\
 \label{gammasubzero}  & \gamma_{T,l} = 0 && \hspace{-3cm} \forall l \in \Lambda               \end{align}

Observe that both ${\bf x}$ and the optimistic walk-in sales variable $s^{b+}_{tl}$ are passed on to this problem, which we henceforth refer to as the sub-problem (SP). The variables $\alpha, \beta$ and $\gamma$ are the duals to the constraints in the primal problem (the innermost fulfillment problem). 
The inner problem in the above model is nonlinear and NP-Hard due to the bi-linear terms $\alpha_{tl}{D}_{tl}^{b+} $, $\beta_{tz} {D}_{tz}^o$. \textcolor{\highlightcolor}{We aim to leverage the column and cut generation (CCG) algorithmic framework to solve two-stage RO problems proposed by~\citet{zeng2013solving}. The CCG algorithm guarantees optimal convergence if the subproblem is solved exactly. To achieve the latter, we}
%(keeping aside the integrality of demand)%. 
%A local optimal solution to this problem can be obtained using a simple {\em alternating heuristic} (AH) that iterately switches between solving for the duals with fixed demand and then for demand by fixing the dual values at their new values until successive dual solutions converge. 
propose an exact mixed-integer linear reformulation that relies on the proposition below and % where we show that the extreme points of $U$ as described in Eq.(\ref{eqn:uncertainty_set}) are integral and (2) 
the fact that there exists an optimal (worst-case) sub-problem demand ${\bf D}^-$ which is at an extreme point of the uncertainty set $U$ (given the dual variables $\alpha, \beta, \gamma$, the problem is an LP with ${\bf D}^-$ as the variables). %also see~\citealt{zeng2013solving} and reference therein). 
\begin{proposition}~\label{Integral_ExtremePoints}
 The extreme points of the uncertainty set $U$ described in Eq.(\ref{eqn:uncertainty_set}) are integral with integral bounds. 
\end{proposition}
The proposition states that %for the simple quantile budget structure we assumed in Eq.(\ref{eqn:uncertainty_set}), 
so long as the bounds in Eq.(\ref{eqn:uncertainty_set}) (which are quantiles) are integral, the extreme points are integral.
Real-world demand realizations are discrete and therefore, the estimated quantiles of demands, and their aggregate values, are integral, making it easier to fulfill the conditions of the proposition. {For general uncertainty sets $U$ whose structure maybe different from Eq.(\ref{eqn:uncertainty_set}), the results of the proposition may not be true. In that case, one may choose to work with (1) a discrete uncertainty set, say $U_{\mathbb{Z}}$, or (2) the convex set generated by the integral points in $U$. Next, we propose an exact mixed-integer linear reformulation by %tapping into the fact that the optimal adversarial demand is indeed discrete.
\textcolor{\highlightcolor}{ \textcolor{\highlightcolor}{leveraging the discreteness of the optimal adversarial demand. Specifically, we enumerate all integral values within the uncertainty set $U$ which happens to also include the optimal adversarial demand, due to Proposition~\ref{Integral_ExtremePoints}. }}
%\noteSS{please see if clarity of prior para can be improved and made more succinct if possible. Are we simply saying that the Prop.1 result holds for continuous demands but need not be true for discrete demands, which is how they are in practice?. Also is the last line with option-1 and option-2 necessary and where we use this idea.}
%\notePH{Show that robust always lies on extreme point and extreme points are discrete as quantiles and underlying distribution for testing is discrete and so without loss of generality we can work with discrete sets}
\begin{proposition}~\label{RLT_subproblem}
The sub-problem can be re-formulated exactly as follows. Here, let $\overline{D}_{tlk}^b$ for $k \in K^b_{tl}$ and $\overline{D}_{tzk}^o$ for $k \in K^o_{tz}$ be the discrete values of the uncertainty that are feasible to the first set of box constraints in Eq. \eqref{eqn:uncertainty_set} in set $U$ and $M = \max_{t,l}\{
p^b_{tl} + b^b_{tl} ,  p^o_t +b^o_t - c_{lz}\} + |T+1|\max_l h_l$.%\footnote{Note that this $M$ computation can be made tighter by having it tailored by time period and node or zone in the implementations, more specifically with $M_{tl} = p^b_{tl} + b^b_{tl} + (T-t)h_l $, $M_{tz} = p^o_{t} + b^o_{t} + \max_{l \in \Lambda} \{(T-t)h_l-c_{lz}\} $.}  
\begin{align}
 Z^{SP({\bf x}, {\bf s}^{b+})} = \min_{\substack{\alpha, \alpha^*, \beta, \beta^* \geq
  0,\\
  \gamma, w \in \{0,1\}}}
\quad \nonumber & \sum_{\substack{t \in \ca{T}, l \in \Lambda,\\ i \in I^b_{tl}}} 
                  (1- \lambda_t)\overline{D}_{tlk}^b (\alpha^*_{tlk} - b^b_{tl} w^b_{tlk})
                + \sum_{\substack{t \in \ca{T}, z \in Z,\\ i \in I^o_{tz}}}
                  \overline{D}_{tzk}^o  (\beta^*_{tzk} - b^o_{t} w^o_{tzk})
                              \\
  & \hspace{-1cm} + \sum_{l \in
                  \Lambda} \gamma_{0l}\bar{I}^{0}_{l} + \sum_{t \in \ca{T}, l \in
                  \Lambda} \gamma_{tl} \left(\bar{I}_{l}^{t+1}
    \mathbbm{1}_{t<L_l}  + x_{t-L_{l},l} \mathbbm{1}_{t   \geq  L_{l}} 
%    + \sum_{l' \in \Lambda \cup S} x^{t-L_{l'l}}_{l'l} \mathbbm{1}_{t   \geq  L_{l'l}}  -\sum_{l' \in \Lambda } x^t_{l l'} 
- s^{b+}_{tl}\right)\\
\textrm{subject to.} \ \quad & \textrm{Eqs.} (\ref{alphasub}),(\ref{betasub}), (\ref{gammasub}),(\ref{gammasubzero})\\
%  \alpha_{tl} + \gamma_{tl} \geq p^b_{tl} + b^b_{tl}  && \hspace{-4cm} \forall t \in \ca{T}, l \in \Lambda\\
% &  \beta_{tz} + \gamma_{tl} \geq p^o_t + b^o_{tl}- c_{lz}  && \hspace{-4cm} \forall t \in \ca{T},l \in \Lambda, z \in Z\\
%& \gamma_{tl}-\gamma_{t+1,l} \geq -h_l && \hspace{-4cm} \forall t \in \ca{T}, l \in \Lambda\\
%& \gamma_{T,l} = 0 && \hspace{-4cm}\forall l \in \Lambda  \\
\label{Malphastar}&  \alpha^*_{tlk} \leq M  w^b_{tlk} ; \quad \beta^*_{tlk} \leq M
                                                                    w^o_{tzk} && \hspace{-4cm} \forall t \in \ca{T},l \in
                                            \Lambda, z \in Z,i\\
%\alpha^*_{tlk} \geq  \alpha_{tl} - (1-w^b_{tlk} )M;  \quad \beta^*_{tlz} \geq  \beta_{tz} - (1-w^o_{tzk} )M; \\
 % \sum_i  w^b_{tlk} = 1, \ \  \sum_i  w^o_{tzk} = 1\\
 \label{alphastar} &  \sum_{k \in K^b_{tl}}  \alpha^*_{tlk}  = \alpha_{tl} ; \quad  \sum_{k \in K^o_{tz}} \beta^*_{tzk}
   = \beta_{tz}  && \hspace{-4cm} \forall t \in \ca{T},l \in
                                            \Lambda, z \in Z\\
\label{wsuperposition} &  \sum_{k \in K^b_{tl}} w^b_{tlk} =1; \quad    \sum_{k \in K^o_{tz}}
  w^o_{tzk} = 1  && \hspace{-4cm} \forall t \in \ca{T},l \in
                                            \Lambda, z \in Z\\
 \label{wonlinebudget}&\bar{\bar{D}}_{t}^{oL} \leq \sum_{ z\in Z, k \in K^o_{tz}}\overline{D}_{tzk}^o
  w^o_{tzk} \leq \bar{\bar{D}}_{t}^{oU} && \hspace{-4cm} \forall t \in \ca{T}\\
\label{wbrickbudget}  &  \bar{\bar{D}}_{t}^{bL} \leq \sum_{ l\in \Lambda, k \in K^b_{tl}}\overline{D}_{tlk}^b
  w^b_{tlk} \leq \bar{\bar{D}}_{t}^{bU} && \hspace{-4cm} \forall t \in \ca{T}
\end{align}
\end{proposition}

Note that $M$ can be tightened by having it tailored by time period and node or zone in an actual implementation (see appendix). %, more specifically with $M_{tl} = p^b_{tl} + b^b_{tl} + (T-t)h_l $, $M_{tz} = p^o_{t} + b^o_{t} + \max_{l \in \Lambda} \{(T-t)h_l-c_{lz}\} $.}  
% A typical approach in two-stage RO papers~\citep{zeng2013solving,bertsimas2018scalable} is to linearize the KKT conditions assuming a continuous uncertainty model and use a generic big-$M$ parameters to reformulate the sub-problem as an MIP. \textcolor{\highlightcolor}{\citet{simchi2019constraint} assume integral coefficients in network flow problems and exploit it to obtain 
% %exploit the integrality assumption of the coefficients in their network flow problem to obtain 
% discrete dual variables to then generate an MIP.} Unlike these methods, we rely on the discreteness of the uncertainty set and RLT methods to re-formulate the sub-problem as a MIP.
Our method capitalizes on the discrete nature of the optimal uncertainties and employs the Relaxation and Linearization Technique (RLT, \citealt{sherali1998reformulation}) to \textcolor{\highlightcolor}{exactly} re-formulate the sub-problem as an MIP. This contrasts with  traditional methods such as those in \citet{zeng2013solving,bertsimas2018scalable}, which linearize the KKT conditions under continuous uncertainty assumptions using generic big-M parameters to derive MIP reformulations, or the approach in \citet{simchi2019constraint}, which leverages integral coefficients for network flow problems to construct MIP reformulations.

\textcolor{\highlightcolor}{Propositions~\ref{Integral_ExtremePoints} and~\ref{RLT_subproblem} together establish that
the sub-problem for the BIO formulations can be solved exactly with MIP solvers despite being NP hard. The CCG algorithm converges to the optimal solution in finite iterations as long as the sub-problem is solved exactly~\citep{zeng2013solving}.} To solve two-stage RO problem, we integrate the MIP sub-problem within the CCG algorithm.  %An important aspect of our sub-problem is that feasibility is guaranteed regardless of the values of its inputs $({\bf x}, {\bf s}^{b+})$, i.e., the solution space is non-empty {\em relatively complete recourse}. This is true even in the presence of additional business rules such a fulfillment capacities or fulfillment window goals, as a zero fulfillment is always a feasible solution. The CCG framework by~\citet{zeng2013solving} only address problems that satisfies this property, which suffices for us. For generalizations, refer to \citet{bertsimas2018scalable}'s extension of the CCG framework.
Here, the master problem is formulated using only a finite set of feasible ${\bf D}^{-}$ uncertainty values, but with a structure identical to the original problem. The solution method can be viewed as the master generating a sequence of beneficial allocations, and the sub-problem adding corresponding adversarial demand `cuts' to master and the iterations stop when the two problems converge in their achieved objectives.
 %\textcolor{\highlightcolor}{As mentioned earlier, CCG algorithm converges to the optimal solution in finite iterations as long
%as the sub-problem is exact~\citep{zeng2013solving}. Propositions~\ref{Integral_ExtremePoints} and~\ref{RLT_subproblem} together establish that
%the sub-problem for the BIO formulations can be solved exactly with MIP solvers (even though
%they are NP hard).}
  %This is the case for the BIO formulations we study due to Propositions~\ref{Integral_ExtremePoints} and~\ref{RLT_subproblem} together. 
 In the interest of brevity, we include the master problem formulation and the corresponding CCG algorithm's adaptation to the BIO-$\lambda$ problem in Appendix~\ref{CCGmaster}.  %\notePH{Comment about how optimistic can be continuous sets and adversarial is discrete set}

\section{Computational Experiments}~\label{experiments}
In this section, we use real-life data to compare the practical performance of different inventory optimization policies developed in this paper relative to baseline methods, and analyze their impact on total profitability and other metrics. 

%\subsection{Real-world data description and problem setting}
We obtained data from a leading omnichannel retailer operating over 150 locations and multiple DCs {(warehouses)} in the continental United States. {The dataset includes historical sales records from transaction-log and e-commerce orderline data, meta-data such as product and location attributes, and rate cards used for calculating the fulfillment costs. }
All locations are SFS equipped, i.e., their store inventory can be used to fulfill online orders. We selected four jackets subcategories in the apparel category during the second sales peak (the first being the holiday season) of the 
calendar year. We analyze 23 SKUs corresponding to the top 5\% best-sellers that account for more than 45\% of the total sales volume during the 3 week testing period.
Fig.~\ref{fig:salestrends} depicts the chain-level sales and the corresponding e-commerce share for these select SKUs over a 3 month peak period. 
Over 60\% of these sales 
occur during the testing period. The spike in walk-in and e-commerce sales (and the e-commerce share) during this peak  underscores the need for accurate inventory positioning to meet omnichannel demands as it can significantly impact the retailer margins for the entire season.
%indicating that more shoppers lean into the online channel during the peak season.  

We assume a weekly replenishment cycle and a lead time of one-week. 
\begin{figure}[t]
  \begin{center}
     \includegraphics[trim = 1.7cm 1cm 0cm 0cm, clip,scale=0.5]{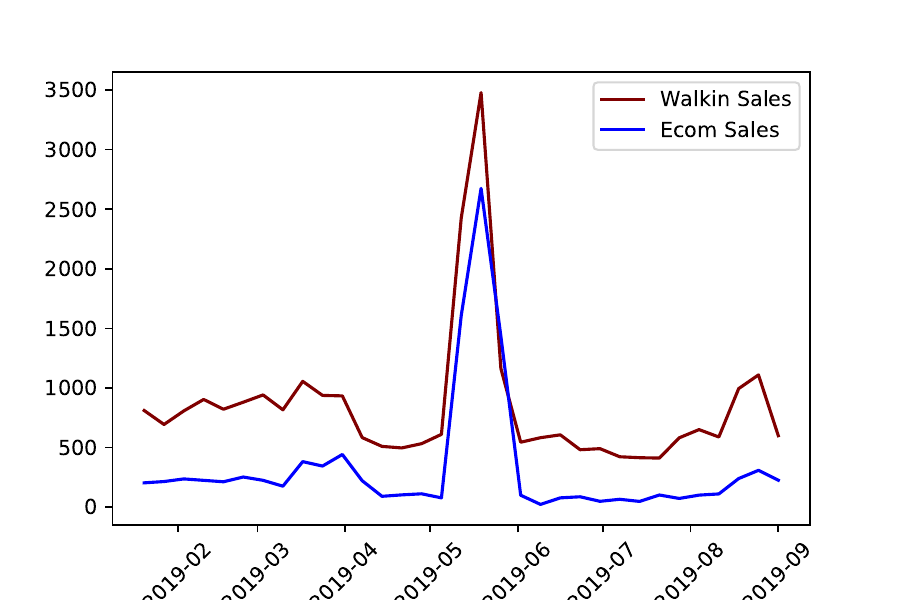}
    \includegraphics[trim = 0cm 1cm 0cm 0cm, clip,
     scale = 0.5]{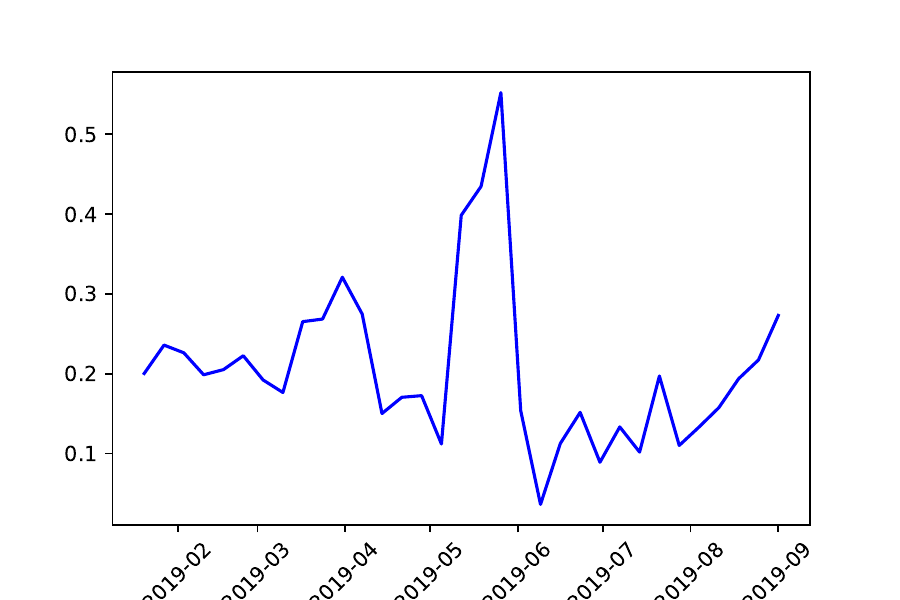}
 \caption{Actual sales trends by channel (left) and the actual e-commerce sales fraction of the total sales (right) for select SKUs around sales peak. %\textcolor{red}{Discuss with Shiva}
 }
 \label{fig:salestrends}
\end{center}
\end{figure}

\subsection{Inputs for inventory management}
%We cluster the zip3s across the continental US into 22 zones based on similar node-to-zip3 lowest unit fulfillment costs due to low rate of sales.

We cluster the individual zip3 sales across the continental US into 22 zones based on similar node-to-zip3 lowest unit fulfillment costs in order to deal with the low sales rate at the zip3 level. The labor cost by origin node (lower in DCs compared to metropolitan area) are additionally incorporated to obtain the total fulfillment cost. %Labor costs are lower for DCs compared to stores and for metropolitan area stores compared to suburban stores.

We employ a time-series multi-task uncensored demand forecasting model at the SKU-location-channel-week level 
with hierarchical location-based reconciliation to generate a daily refresh of the mean weekly demand forecast for the next few weeks.  Specifically two neural network-based models, %\footnote{Several different models and baselines were build before converging on a neural-net model}, 
one for SKU-store-Demand-week forecasts, and the other for SKU-zone-OnlineDemand-week predictions
were developed using {18.7 and 3 million observations respectively
corresponding to daily observations between 2016 to 2019 for SKUs in the chosen sub-categories across locations and zones.} 
%18,680,513 and 2,978,053 observations respectively
% 18.7 Mil. and 3 Mil. observations respectively
% corresponding to daily observations between 2016 to 2019 for 1,431 SKUs and 138 nodes / 22 zones. \noteSS{risk in giving exact number 1431 may bias the reviewer to ask to analyze more than 23 skus?} 
Multi-tasking allows the model to learn across multiple time-series based on common attributes (common sub-patterns in hierarchies and corresponding historical sales, prices, and calendar and seasonal attributes) under sparse data conditions and captures correlations automatically. 
Sales data is uncensored within this procedure to account for lost sales due to stock-outs and produce an unconstrained demand forecast. Reconciliation ensures consistency between chain-level and location-specific forecasts, an aspect that is important in our problem context. The forecasting model minimizes a Poisson loss function (along with the regularization and incorporation hierarchical aggregate series) 
%and reconciliation terms) 
keeping in mind the low rate of sales and the discreteness of demand, among other factors. This loss function also enables us to aggregate the local mean demands into a chain-level mean demand. 

To generate the required local and chain level quantiles of the uncertainty set, we use the 5th and the 95th quantile of the Poisson distribution. For Monte Carlo sample paths, we use a random number generator with an underlying Poisson distribution with the estimated mean. To be distribution-free, a quantile loss function at multiple levels of the hierarchy could be alternatively used instead of the Poisson loss function. {We adopt the former approach to be consistent in our modeling and our sample generation procedure as Monte Carlo samples are required for the out-of-sample computational comparisons we perform in this paper.} %\noteSS{we give the quantile loss function as an alternative that could be used, but then say this is what we use- pls check}
%\noteSS{a first-look/naive reviewer may question why we are "assuming Poisson" when we claim our method is distribution-free. A line on this upfront may help clarify} 
%\notePH{Ask Brian - how many data samples were used for this forecasting}

\subsection{Baseline Inventory Optimization Methods}~\sectlabel{baselines}
\textcolor{\highlightcolor}{Three} fast heuristic methods described below are used to generate the baseline KPIs for comparison. The methods rely on location-specific demand quantiles, \textcolor{\highlightcolor}{evaluated using an assumed distribution--Poisson for perfect information and Uniform for imperfect information.} The critical-quantile we use is the location-channel specific ratio of the margin 
%(price minus item cost minus average shipping cost, if online)
(price less the item cost, and less the average shipping cost for the ecommerce channel)
to the price of the item. 

\begin{henumerate}
\item {\em Basestock-based distribution heuristic:} We use the basestock heuristic to order supply at each store location for walk-in demand and next at the chain level for online demand. In the latter case, we consider the warehouse inventory as well as excess inventory in stores (i.e., that above critical-quantile store demand) to derive the current inventory position.  The chain level e-commerce supply is distributed amongst warehouses in-proportion to their individual basestock order quantities (where in, demand is obtained by mapping zones to their nearest warehouses).

\item {\em Piecewise-linear (PWL) network-based heuristic: } We extend the deterministic version of the pure robust model to allow for two classes of demand at each location: the mean demand, and the excess beyond the mean and up to the critical quantile (if higher). The price and penalties for the second class are associated with a discount factor that is calibrated to prioritize the satisfaction of all mean demands over any excess.
%such that first all the mean demands are met prior to meeting the second class portion of the demand. 

\item \textcolor{\highlightcolor}{{\em Basestock-based network distribution heuristic:}  This heuristic resembles the basestock heuristic described above but differs in key aspects. Instead of pooling online orders at the chain level, it pools demand network-wide, ordering only the incremental quantity beyond walk-in orders for online demand. The additional supply is then split, with a fraction $\theta$ allocated to stores and the rest to warehouses. Within each group, the distribution across nodes is proportional to their respective base-stock order quantities, where DC demand is obtained by mapping zones to their nearest warehouses. The incremental order quantity for online demand is capped so it does not exceed what would result from pooling at the chain level}\footnote{\textcolor{\highlightcolor}{This restriction is necessary because, for low-sales-rate items, pooling at the network level across thousands of stores could generate excessive incremental demand—mainly for walk-ins that were not driven by store-level orders—resulting in unsold inventory and subpar performance.}}. 
\end{henumerate} 
\textcolor{\highlightcolor}{For the basestock heuristics, we use a capped variation that limits the order quantity to the cycle safe stock level, as it known to outperform the vanilla variant in lost sales settings~\citep{xin2021understanding}. This observation holds across all our experiments, so we adopt it as the default.} %(Linwei Xin 2020).
%$$OrderQty = \max(0, \min(Cycle \ SSL, (CyclePlusLead \ SSL -Current \ %Inventory)))$$
The heuristic methods do not rigorously model demand uncertainty but instead incorporate its key elements in a practical manner, serving as competitive benchmarks for the real-data sets analyzed (as we will see). \textcolor{\highlightcolor}{Unlike the RO and BIO models, these heuristics do not guarantee feasible solutions in the presence of side-constraints, often required in practical applications, such discussed in Section \ref{model}.}

\subsection{Experimental design and setup}
We perform two different computational experiments under various initial conditions and parameter settings to assess the performance of the inventory management policies and their suitability for real-world applications. The first experiment involves a batch Monte-Carlo simulation, where the optimizer is executed only once to evaluate its performance in isolation across various scenarios. %This experimental set-up is aligned with the design \noteSS{design meaning?} and assumptions of the optimizer.
The second experiment is a business value assessment over an extended time horizon using 
%In the second experiment, we developed and used 
a transaction-level Monte Carlo simulation. Here, the inventory optimizers were executed on a weekly basis, and demands were fulfilled order-by-order on-the-fly. This simulation mimics a real-life order fulfillment system.

A two-week planning horizon is used for all the inventory optimizers in both experiments. As we solve several large instances of the two-stage inventory optimization model which involve NP-hard sub-problems, we employ a heuristic version of Algorithm~\ref{alg:pureRO_Benders}, by employing an {\em alternating heuristic} (AH) to warm start and polish the sub-problem solution along with practical safeguards %such as %to achieve high quality solutions within a reasonable time including 
to limit the maximum run time (300 seconds) and the number of Benders iterations (20). %We also add shallow cuts into the master by solving the sub-problem approximately (using tolerance and time limits) to speed the convergence. We enhance the shallow cuts by `warm starting' the sub-problem solver using the {\em alternating heuristic} (AH) and to polish the sub-problem solution again with the AH at termination. 
The AH iteratively switches between solving for the duals with fixed demands and then for demand by fixing the dual values at their new values until successive dual solutions converge.

\subsubsection{Batch Monte-Carlo simulations with optimizer executed just once:} 
The optimizers are executed for a peak sales week for different hyper-parameter settings and starting inventory positions. We evaluate the resultant allocations of each optimizer by solving the innermost maximization problem described in (\ref{staticRobustFormulation})  and generate a profitability distribution of the achieved objective %(the profitability gain) 
using the Monte-Carlo samples.

In the first experiment, we analyze the sensitivity of the profitability gain to the parameter settings by evaluating the gain at three different parameter settings where (1) $p_t=0, b_t=p_{0t}$, (2) $p_t=p_{0t}, b_t=p_{0t}$ and (3) $p_t=p_{0t}, b_t=0$, where $p_{0t}$ is the price of the item in week $t$. For these SKUs, the margin is on average 50\% at full price. The SKU prices drop by 20\% between the lead time and cycle time week. We assume that $h = 0$. %that the holding cost is zero. %as each unsold item is equivalent to a %wasted purchase. 
%lost sales opportunity. 
For each of these parameter settings, we consider three different starting inventory positions to recreate different network conditions: (1) zero - all locations have zero inventory; (2) mean - all locations have inventory equal to mean of the lead time demand including DCs that have the preferred zones mapped to them; %(3) zeroDC-excessImbalanceStores - the DCs have zero inventory and but across the network there is sufficient inventory to meet demand but it is not in the right place and is randomly allocated. Note that here the optimizer aims to order required amount for the walk-in demand that would be otherwise lost as there is sufficient inventory for e-commerce demand in the network; 
and (3) excess - DCs have excess inventory by SKU, and stores have inventory that is sufficient at the category level but is imbalanced across SKUs and locations (inventory equal to total mean walk-in demand is randomly allocated across SKUs and stores). %Effectively, (3) is a pure walk-in setting where the optimizer only satisfies in-store demand. 
\textcolor{\highlightcolor}{We implement the aforementioned baseline methods—Basestock, PWL, and BS Net$_{\theta}$—under perfect distributional information (i.e., Poisson) and apply the  network heuristic under imperfect distributional information (i.e., Uniform), denoted as BSU Net$_{\theta}$. We evaluate $\theta= 0$ and $\theta=1$ to analyze the impact of network pooling and the incremental value of maintaining this inventory at stores\footnote{\textcolor{\highlightcolor}{Since $\theta=1$ consistently gave near-optimal results when tuned, we present only this case.}}. Additionally, we evaluate distribution-free}
% the pure robust method (denoted by `Pure RO') 
BIO models for different $\lambda$ values (denoted by BIO 5\%, BIO 10\%, BIO 25\%, BIO 50\% and BIO 75\%), and BIO-0, i.e., the pure robust method (Pure RO). 

For the BIO model analysis, we split the (1000) Monte Carlo samples into a cross-validation (80\%) and a hold-out test data set (20\%). Doing so enables us to identify the most profitable $BIO-\lambda$ on average at the SKU level via a one-dimensional grid search (across  $\lambda$ in \{5,10,25,50,75\}) using the cross-validation data set and evaluate performance of this optimized $BIO-\lambda$ on the hold-out data. We refer to this ML based BIO model as `BIO-best'. For appropriate comparison, results of all models are reported on the hold-out data only. %For comparison, we also present results for a uniform demand distribution of the Monte-Carlo samples in Appendix~\ref{UniformDistr} and the findings are consistent with the results presented here. 

 %A summary of the results of the experiments over the 23 SKUs on test-data are provided in 
 Fig.~\ref{fig:results_zero_new} %for the different settings. This figure 
 provides the total ordering decision %by node type and the resultant profitability distribution for the different settings and methods aggregated over all  SKUs. 
 per node type and resulting profitability distribution for various settings and methods aggregated over all SKUs. Table~\ref{tab:results_standalone} complements this by reporting the percentage improvement in mean %and 5th quantile 
 profitability gains compared to the basestock method. \textcolor{\highlightcolor}{The 5th quantile percentage gains closely follow the insights from mean performance gains, also as illustrated in Fig.~\ref{fig:results_zero_new}, so we do not report them separately.} 
 
 %, complementing the figure.
 %complements this figure and reports the percentage improvement of the mean profitability gain, as well as the 5th quantile profitablility gain (viewed as a practical worst-case gain), of the different methods over the basestock method. 
 %Note that barring sub-graph (i), all methods use a Poisson-distribution based Monte-Carlo sampling scheme while sub-plot (i) uses a uniform distribution with appropriately adjusted budget constraints. 

\begin{figure}[t]
  \begin{center}
   \hspace{1cm} $p=p_{0t}, b= 0$ \hspace{3cm} $p=p_{0t},b=p_{0t}$ \hspace{3cm} $p = 0, b=p_{0t}$\\ 
    \hspace{-4.8cm}  \rotatebox{90}{\hspace{0.9cm}Zero Init. Inv.}
    \includegraphics[trim = 4cm 0cm 3cm 0cm, clip,
     scale = 0.27]{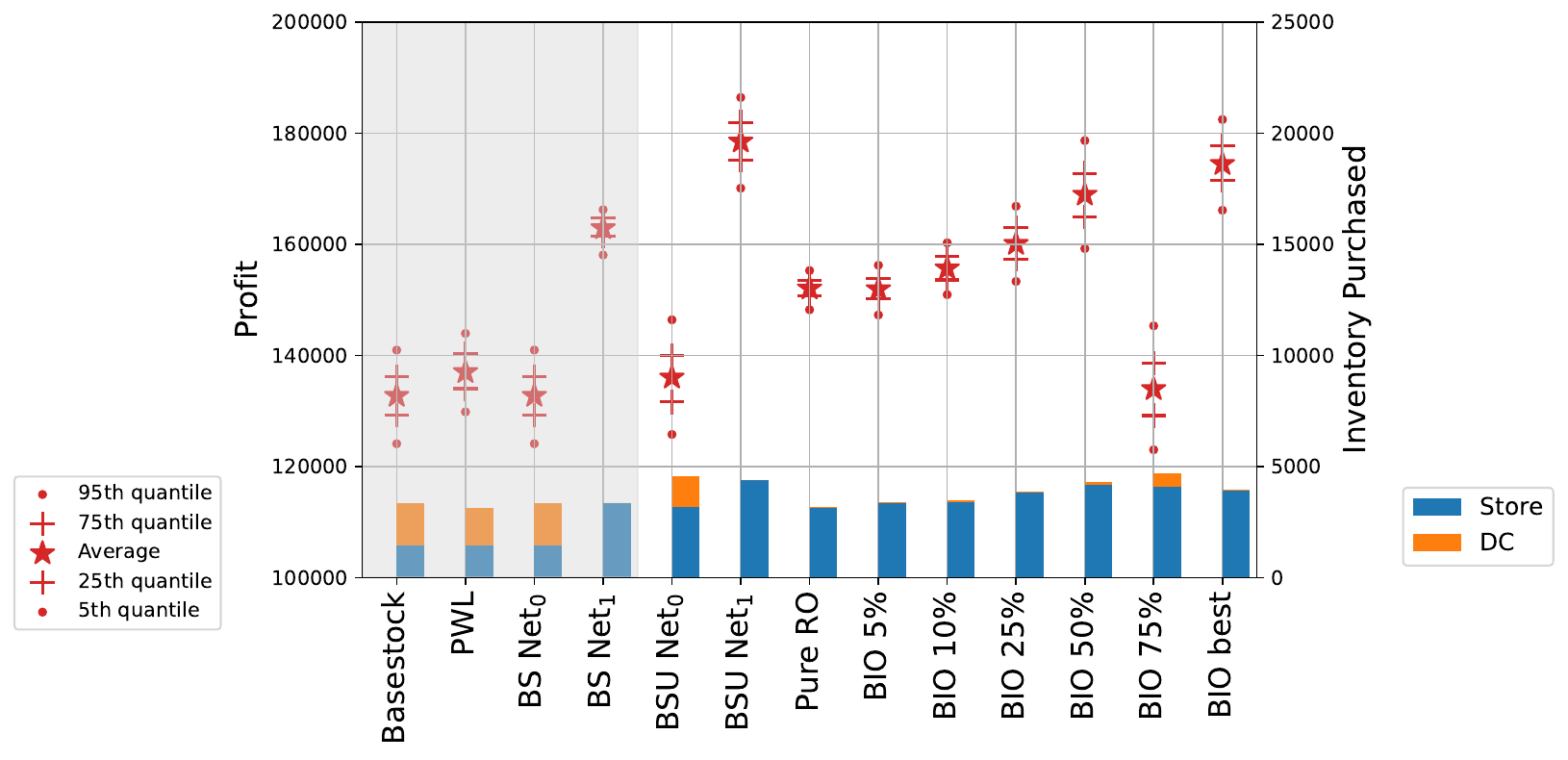}  \hspace{-0.25cm} 
    \includegraphics[trim = 4cm 0cm 3cm 0cm, clip,
     scale = 0.27]{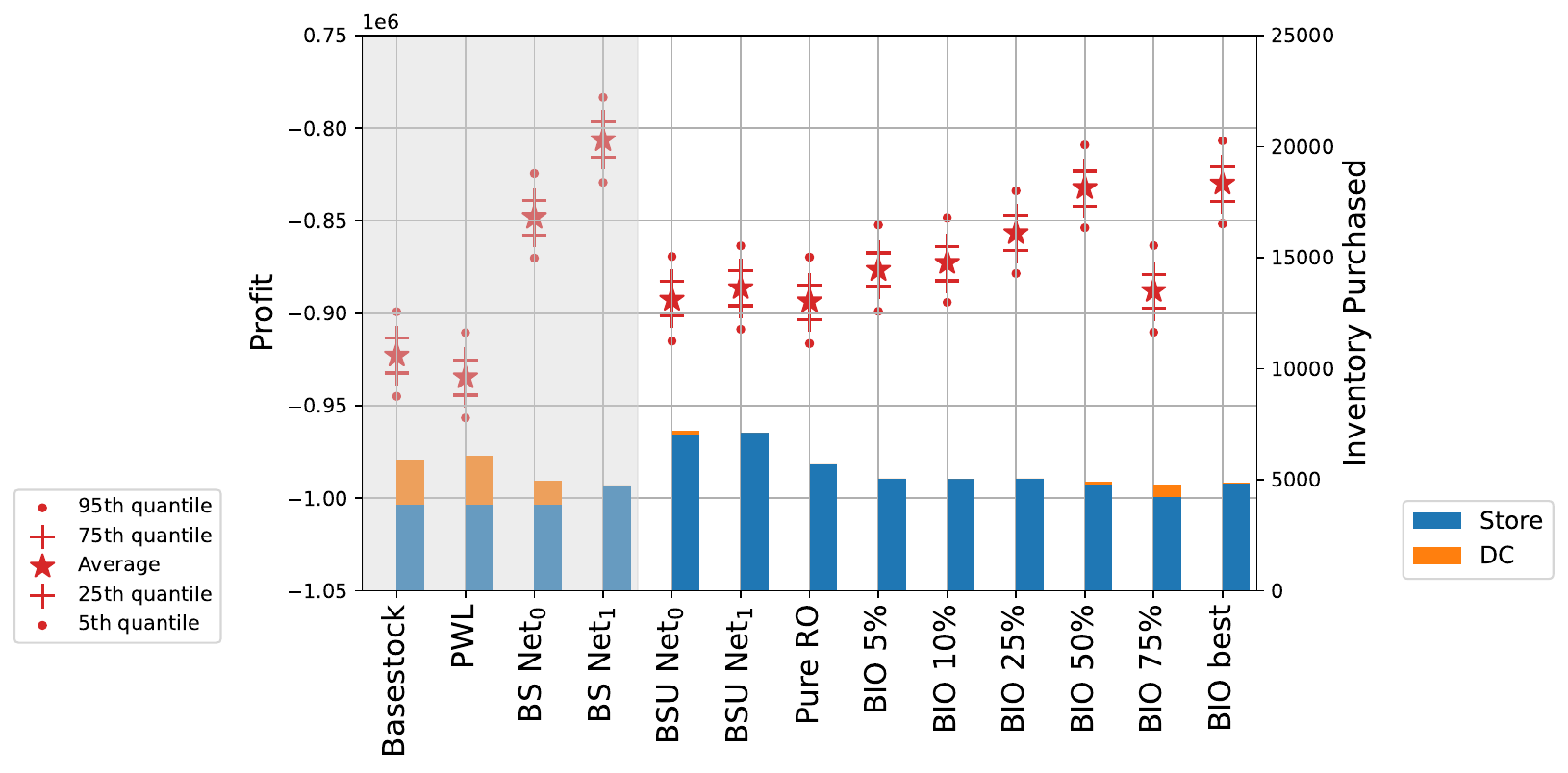}  \hspace{-0.25cm} 
      \includegraphics[trim = 4cm 0cm 3cm 0cm, clip,scale=0.27]{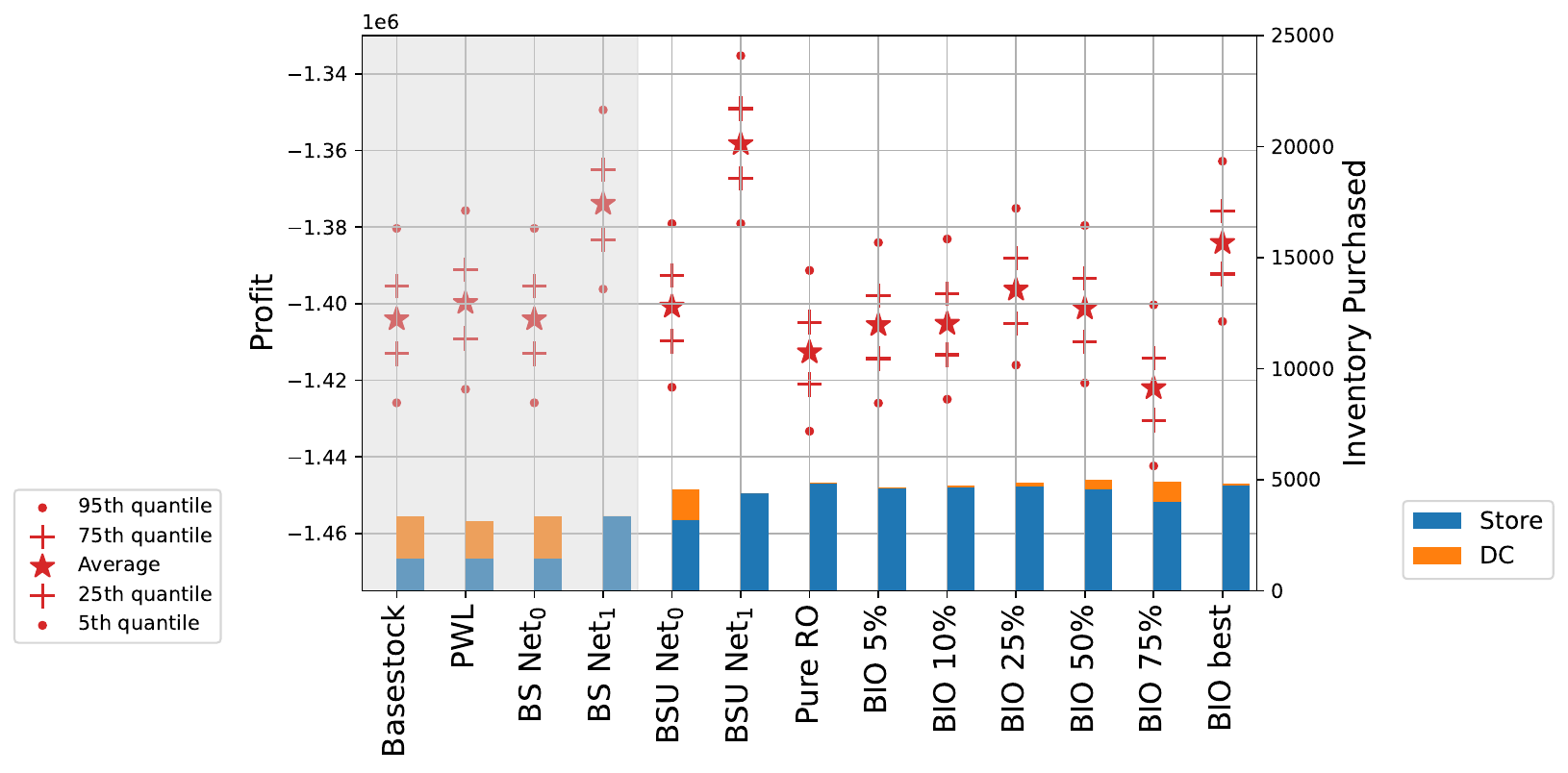}  
            \hspace{-11.9cm}
     \includegraphics[trim = 0.1cm 2.2cm 23cm 8cm, clip,
    scale = 0.27]{Bzero11_new2.pdf}  \hspace{4cm}
        \includegraphics[trim = 24.5cm 2.5cm 0.1cm 8cm, clip,
    scale = 0.27]{Bzero11_new2.pdf}\\ \vspace{-0.6cm}
 \hspace{1cm} (a) \hspace{4.8cm} (b) \hspace{4.8cm} (c)\\ \vspace{0.2cm}
       \hspace{-4.8cm} \rotatebox{90}{\hspace{0.9cm}Mean Init. Inv.}  
       \includegraphics[trim = 4cm 0cm 3cm 0cm, clip,
     scale = 0.27]{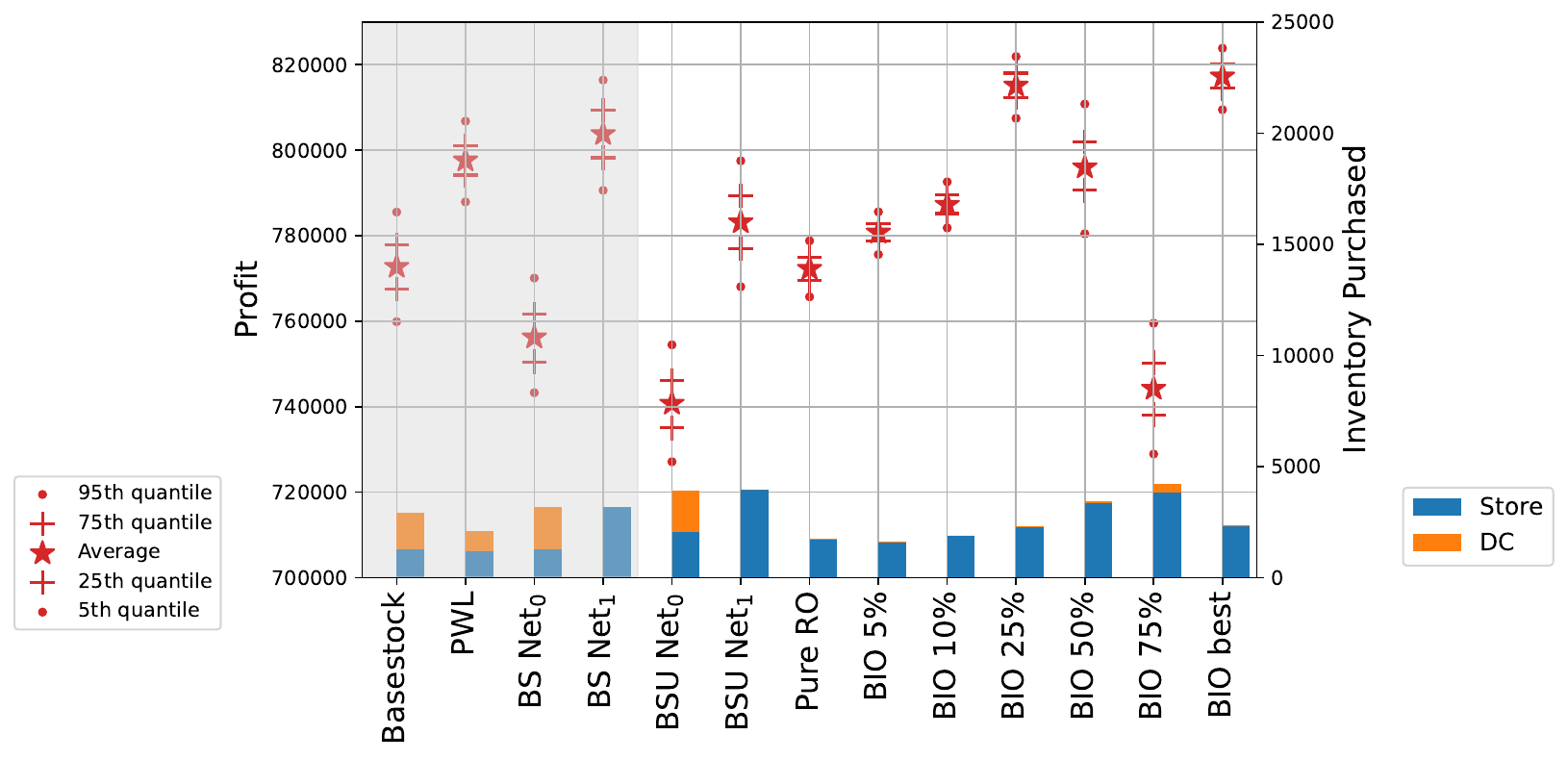} \hspace{-0.25cm} 
    \includegraphics[trim = 4cm 0cm 3cm 0cm, clip,
     scale = 0.27]{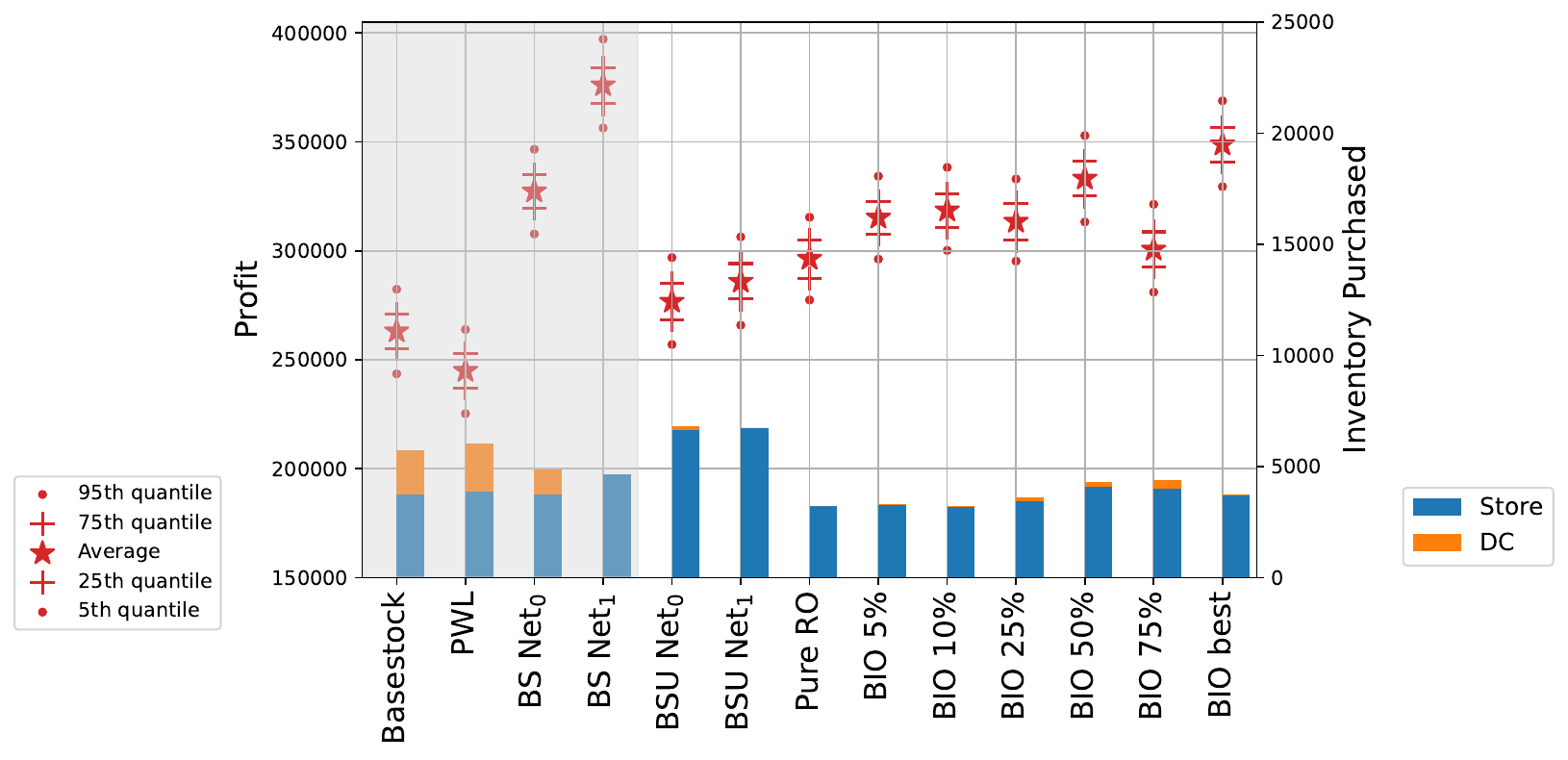} \hspace{-0.25cm} 
 \includegraphics[trim = 4cm 0cm 3cm 0cm, clip,scale=0.27]{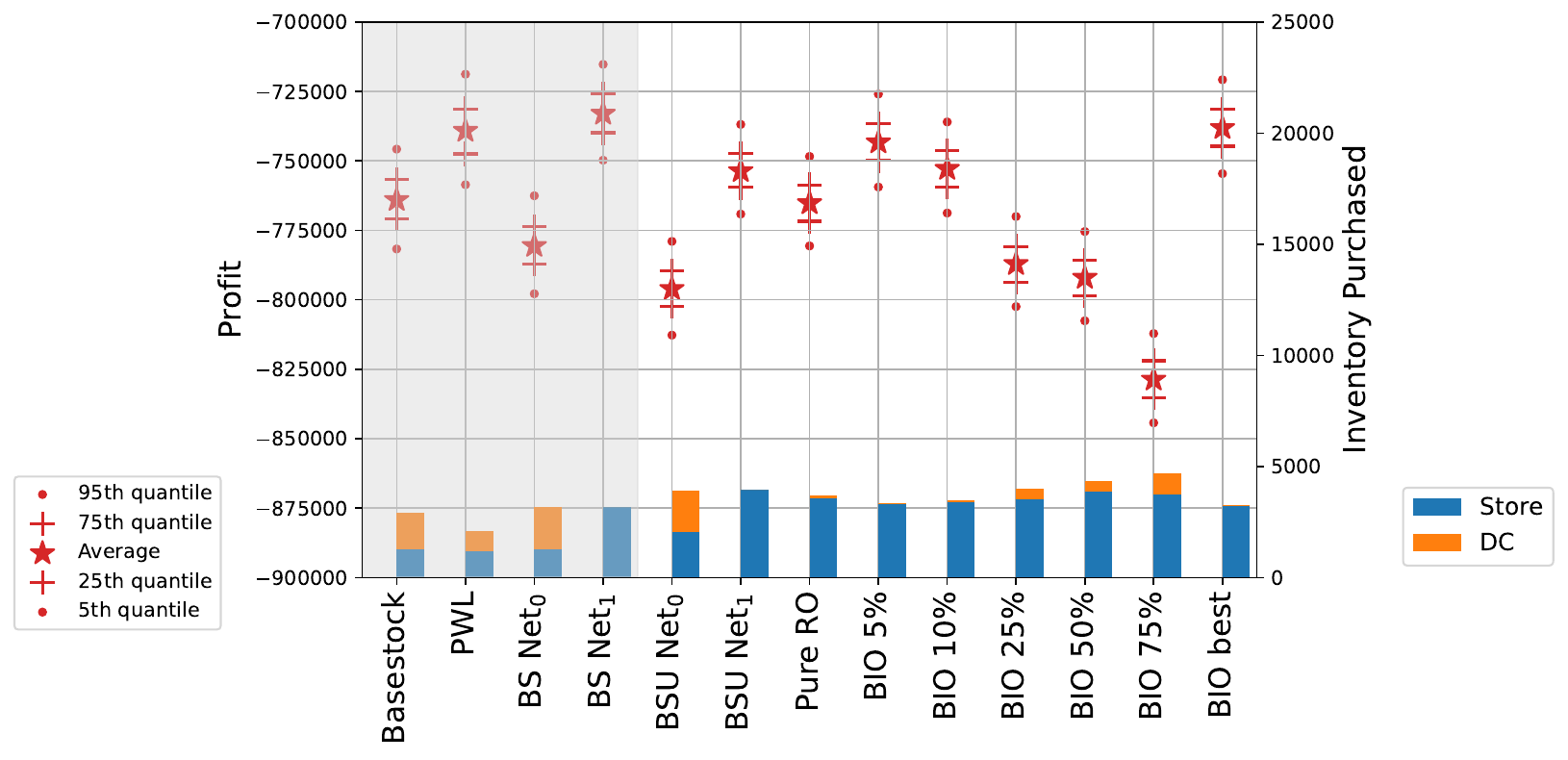}
            \hspace{-11.9cm}  
     \includegraphics[trim = 0.1cm 2.2cm 23cm 8cm, clip,
    scale = 0.27]{Bzero11_new2.pdf}  \hspace{4cm}
        \includegraphics[trim = 24.5cm 2.5cm 0.1cm 8cm, clip,
    scale = 0.27]{Bzero11_new2.pdf}  \\
    \vspace{-0.6cm}
 \hspace{1cm} (d) \hspace{4.8cm} (e) \hspace{4.8cm} (f)\\
  \hspace{-4.8cm} \vspace{0.2cm} \rotatebox{90}{\hspace{0.7cm} Excess Init Inv.}  \includegraphics[trim = 4cm 0cm 3cm 0cm, clip,
     scale = 0.27]{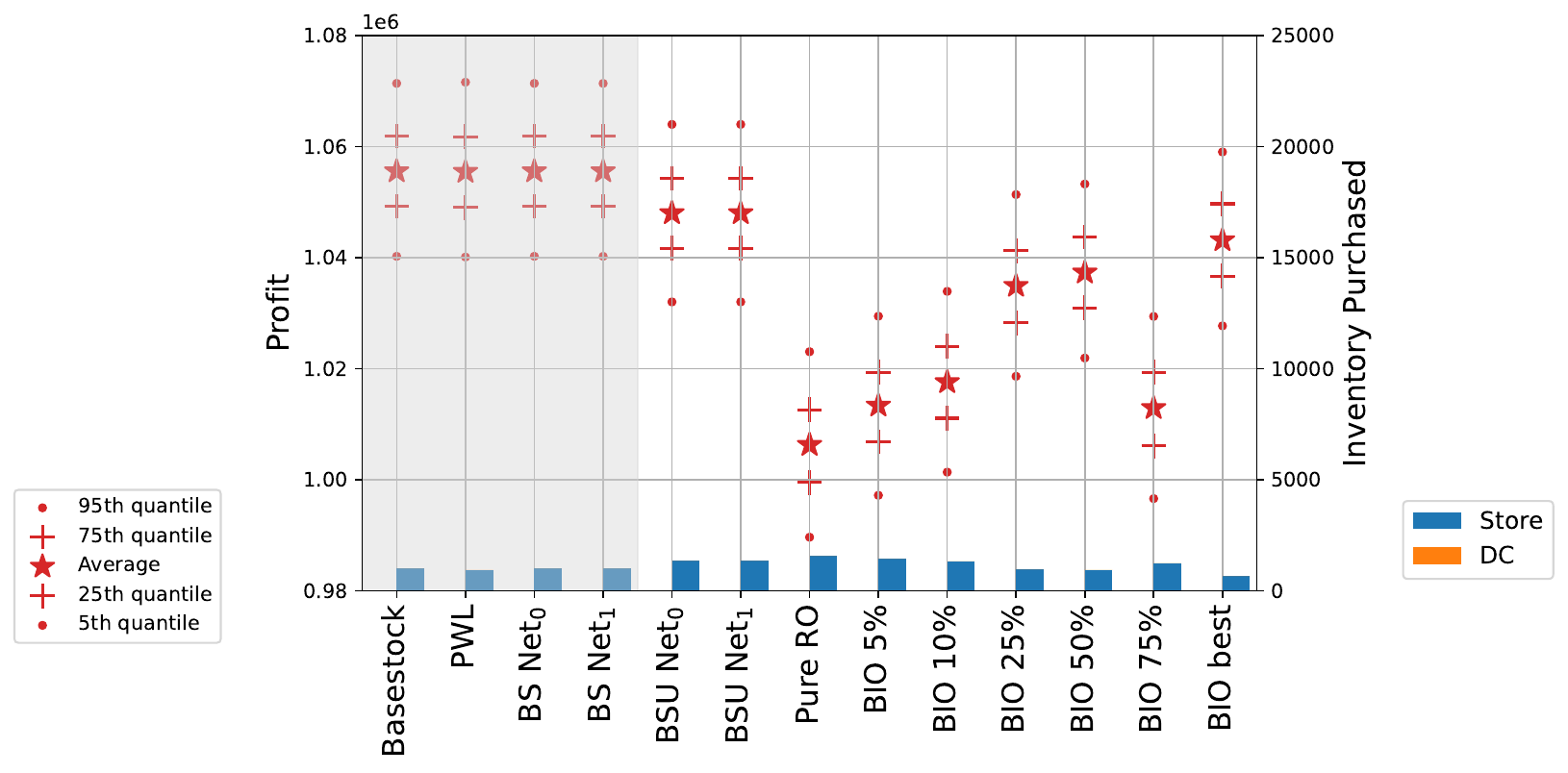} \hspace{-0.25cm} 
    \includegraphics[trim = 4cm 0cm 3cm 0cm, clip,
     scale = 0.27]{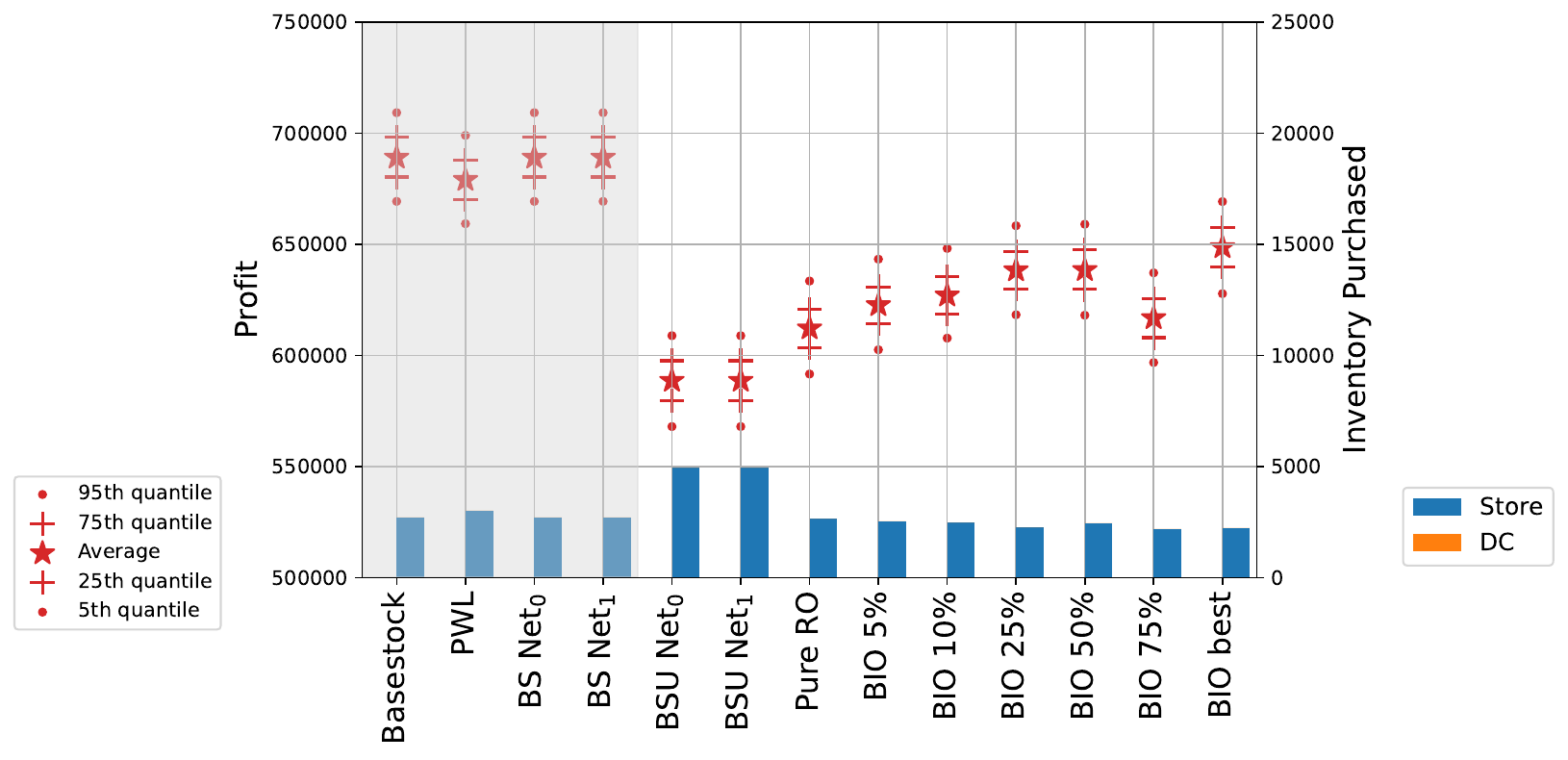} \hspace{-0.25cm} 
 \includegraphics[trim = 4cm 0cm 3cm 0cm, clip,scale=0.27]{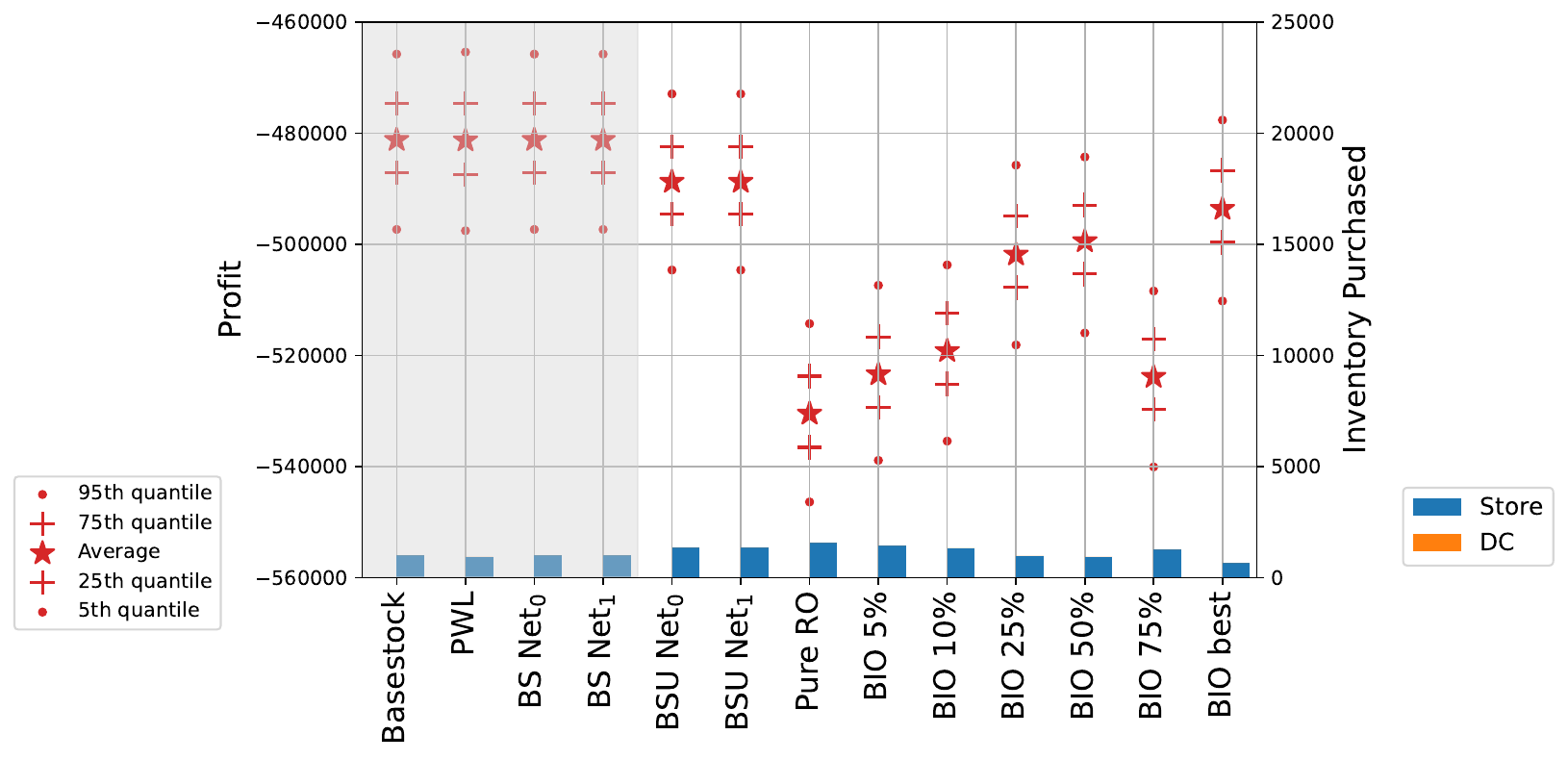} \hspace{-11.9cm}  
     \includegraphics[trim = 0.1cm 2.2cm 23cm 8cm, clip,
    scale = 0.27]{Bzero11_new2.pdf}  \hspace{4cm}
        \includegraphics[trim = 24.5cm 2.5cm 0.1cm 8cm, clip,
    scale = 0.27]{Bzero11_new2.pdf}  \\
    \vspace{-0.6cm}
  \hspace{1cm} (g) \hspace{4.8cm} (h) \hspace{4.8cm} (i)%\\
 %  \hspace{-4.8cm} \vspace{0.2cm} \rotatebox{90}{\hspace{0.7cm} Mean Init., Unif}  \includegraphics[trim = 4cm 0cm 3cm 0cm, clip,
 %     scale = 0.27]{Figures/Bmean10_Uniform_new2.pdf} \hspace{-0.25cm} 
 %    \includegraphics[trim = 4cm 0cm 3cm 0cm, clip,
 %     scale = 0.27]{Figures/Bmean11_Uniform_new2.pdf} \hspace{-0.25cm} 
 % \includegraphics[trim = 4cm 0cm 3cm 0cm, clip,scale=0.27]{Figures/Bmean01_Uniform_new2.pdf} \hspace{-11.9cm}  
 %     \includegraphics[trim = 0.1cm 2.2cm 23cm 8cm, clip,
 %    scale = 0.27]{Figures/Bzero11_new2.pdf}  \hspace{4cm}
 %        \includegraphics[trim = 24.5cm 2.5cm 0.1cm 8cm, clip,
 %    scale = 0.27]{Figures/Bzero11_new2.pdf}  \\
 %    \vspace{-0.6cm}
 % \hspace{1cm} (j) \hspace{4.8cm} (k) \hspace{4.8cm} (l)
 % % \hspace{-4.8cm} \vspace{0.2cm} \rotatebox{90}{\hspace{1cm} p = 1, b= 1}  \includegraphics[trim = 4cm 0cm 3cm 0cm, clip,
 % %     scale = 0.27]{Figures/BexStorezeroDC11_new2.pdf} \hspace{-0.25cm} 
 % %    \includegraphics[trim = 4cm 0cm 3cm 0cm, clip,
 % %     scale = 0.27]{Figures/Bact11_new2.pdf} \hspace{-0.25cm} 
 % % \includegraphics[trim = 4cm 0cm 3cm 0cm, clip,scale=0.27]{Figures/Bmean11_Uniform_new2.pdf} \hspace{-11.9cm}  
 % %     \includegraphics[trim = 0.1cm 2.2cm 23cm 8cm, clip,
 % %    scale = 0.27]{Figures/Bzero11_new2.pdf}  \hspace{4cm}
 % %        \includegraphics[trim = 24.5cm 2.5cm 0.1cm 8cm, clip,
 % %    scale = 0.27]{Figures/Bzero11_new2.pdf}  \\
 % %    \vspace{-0.6cm}
 % % \hspace{1cm} (g) \hspace{4.8cm} (h) \hspace{4.8cm} (i)
 \caption{Total inventory allocation across stores and DCs %for different standalone optimizers under different settings, 
 and resultant total profitability distribution in the Monte Carlo simulation. Initial Inventory is zero in (a-c), mean in (d-f), and excess in DC in (g-i). Price and back order penalty multipliers ($p,b$) to the weekly item price are: $p=1, b=0$ in (a,d,g), $p=1, b=1$ in (b,e,h) and $p=0, b=1$ in (c,f,i). Grey-shaded settings: perfect distributional info; others: imperfect or distribution-free. 
 %Barring (j,k,l) which utilizes a uniform distribution, the rest utilize a Poisson distribution to generate samples for the MonteCarlo simulation. 
 }
 \label{fig:results_zero_new}
\end{center}
\end{figure}
\begin{table}[t]
    \centering
    \scalebox{0.6}{
%\arrayS{1.1}
    \begin{tabular}{@{}ccccccccccccccccc@{}}
     \toprule[1.0pt]
    \multirow{2}{*}{\parbox{2cm}{\centering  Initial Inv. Setting}}	&	\multirow{2}{*}{\parbox{2cm}{\centering Price Factor}}	&	\multirow{2}{*}{\parbox{2cm}{\centering Lost Sales Factor}}	&&	\multicolumn{3}{c}{Perfect Distr. Info.}	&  & \multicolumn{2}{c}{Imperfect Distr. Info.} & \multicolumn{6}{c}{Distribution-free}	
%& &	\multicolumn{10}{c}{\% Improvement in 5th quantile of the Total Profitability	}	
    \\ \cmidrule[0.5pt]{5-7} \cmidrule[0.5pt]{9-10} \cmidrule[0.5pt]{12-17} 
    &&&\phantom{a}
					 &	PWL	&	BS Net$_0$ & BS Net$_1$ && BSU Net$_0$ & BSU Net$_1$ && Pure RO	&	BIO 5\%	&	BIO 10\%	&	BIO 25\%	& BIO 50\% 	&	BIO best	%& \phantom{a} &	PWL	&	Pure RO	&	BIO 5\%	&	BIO 10\%	&	BIO 25\%	&	BIO best	
                     \\ \cmidrule{1-3} \cmidrule[0.5pt]{5-7} \cmidrule[0.5pt]{9-10} \cmidrule[0.5pt]{12-17} 
Zero	&	1	&	0	&&	3	&	0	&	23	&&	2	&	34	&&	15	&	14	&	17	&	21	&	27	&	31	\\
Zero	&	1	&	1	&&	-1	&	8	&	13	&&	3	&	4	&&	3	&	5	&	5	&	7	&	10	&	10	\\
Zero	&	0	&	1	&&	0	&	0	&	2	&&	0	&	3	&&	-1	&	0	&	0	&	1	&	0	&	1	\\
Mean	&	1	&	0	&&	3	&	-2	&	4	&&	-4	&	1	&&	0	&	1	&	2	&	5	&	3	&	6	\\
Mean	&	1	&	1	&&	-7	&	24	&	43	&&	5	&	9	&&	13	&	20	&	21	&	19	&	27	&	33	\\
Mean	&	0	&	1	&&	3	&	-2	&	4	&&	-4	&	1	&&	0	&	3	&	1	&	-3	&	-4	&	3	\\
Excess	&	1	&	0	&&	0	&	0	&	0	&&	-1	&	-1	&&	-5	&	-4	&	-4	&	-2	&	-2	&	-1	\\
Excess	&	1	&	1	&&	-1	&	0	&	0	&&	-15	&	-15	&&	-11	&	-10	&	-9	&	-7	&	-7	&	-6	\\
Excess	&	0	&	1	&&	0	&	0	&	0	&&	-2	&	-2	&&	-10	&	-9	&	-8	&	-4	&	-4	&	-3	\\
\bottomrule[1.0pt]
    \end{tabular}}
    \caption{Percentage improvement of the mean total profitability over the perfect info. basestock heuristic.}
    \label{tab:results_standalone}
\end{table}
\color{\highlightcolor}

Referring to Table~\ref{tab:results_standalone}, we observe that under perfect information, the capped basestock network heuristic significantly outperforms the vanilla basestock approach when $\theta = 1$ in both the mean and zero-inventory settings, highlighting the single-channel method's inability to effectively leverage cross-channel inventory. However, for $\theta = 0$, this advantage is only observed when $p=b=1$. This occurs because, in low sales-rate settings (even during a peak week), lower demand quantiles (when $p=0$ or $b=0$) lead to negligible order quantities, often zero. A slight increase in inventory at stores improves results, which explains the superior performance at higher quantiles (when $p=b=1$ or when $\theta =1$). This same reasoning explains why, under imperfect information (Uniform distribution), the capped heuristic performs better in the zero-inventory case. Across all settings, we note that when incremental online inventory pooled at the network level is pushed to stores ($\theta = 1$), both channels benefit, enhancing overall performance. 
All gains disappear in the excess inventory setting, as the problem simplifies to a single-channel in-store allocation problem that is separable by location. For each individual problem, the basestock method is known to be optimal for maximizing mean profitability.
%In contrast, the baseline models plan to exclusively fulfill online orders from DCs reflecting the inability of single channel or deterministic methods to effectively leverage the cross-channel inventory utilization.

While the perfect-information basestock network heuristic outperforms the imperfect-information and distribution-free settings, its advantage diminishes in realistic scenarios where demand distribution information is inherently imperfect. Therefore, for a practical comparison, we evaluate the imperfect distribution setting against BIO, a distribution-free approach. 
Across the zero, mean, and excess inventory settings, BIO-best outperforms BSU Net$_{\theta=1}$ (network basestock counterpart) heuristic by 0.4\%, 10.1\%, and 2.4\%, respectively, on average across price settings, resulting in an overall average gain of approximately 4.3\%. Notably, we observe significant gains in the $p=b=1$ setting as well as across all mean inventory settings (average +9\%, in 5 of 9 settings), while performance remains relatively low negative otherwise (average -1.6\%, in the rest). % Except for one case when $p \geq b$, the robust network BIO models consistently outperform BSU, while no clear dominance is observed when $p < b$.

We also observe that the pure RO method under-performs relative to the baselines in many parameter regimes. However, this performance gap is closed with a simple tuning of $\lambda$ via the BIO strategy (also as illustrated in Example~\ref{example}). %Notably, in the excess inventory setting, the pure RO approach performs poorly (up to -9\%), while BIO-best reduces this loss significantly (to -3\%), with the remaining gap attributable to differing modeling assumptions and input data.
This BIO tuning is essential for data-driven distribution-free approaches to effectively compete against simple baselines (serving as an alternative to a full-scale SAA approach, which would require a large number of demand scenarios—potentially correlated across all locations and channels—leading to both computational challenges as well as difficulties in scenario generation).
\color{black}
Specifically, we observe the BIO models with $\lambda >0$ enhance the performance of the pure RO models on average by 5-6\% (with BIO 10\%, 25\%) which increases to 12\% by tuning $\lambda$ for each SKU using BIO-best. \textcolor{\highlightcolor}{From Fig.~\ref{fig:results_zero_new}, we observe that a gradual increase in $\lambda$ leads to an increase in mean profitability without compromising practical robustness.} BIO-best does achieve the best overall performance across the mean and the lower quantiles amongst all the BIO models. {The out-of-sample performance worsens when we employ overly optimistic BIO settings (e.g., $\lambda>$ 50\%), indicating that a pragmatic range of smaller $\lambda$ values are preferable.}

From Fig.~\ref{fig:results_zero_new}, we observe that the ordering policy of the baseline models is invariant to $p+b$ but this is not the case for the robust models. The worst case demand magnitude tends to be large for high $b$ and relatively low for high $p$. We observe the same trend in the quantity of inventory purchased especially by pure RO models. %Additionally, the BIO-best model is less conservative and can be similar to either the baseline or the robust model, depending on which performs well but outperforms both by optimizing the inventory positions. 
{A key finding to highlight is that across all the omnichannel settings, the best BIO and RO models organically pool inventory by effectively fulfilling walk-in and online demand %uncertainties 
%\noteSS{don't follow- should it be pooling walk-in and fulfilment inventory? if we are indeed effectively pooling (demand) 'uncertainty' across channels, we need some explanation of how that occurs in the bio model} 
using inventory at the store} (indicated by a thin sliver or lack of an orange bar in Fig.~\ref{fig:results_zero_new} subplots a-f, which denotes the quantity allocated to the DCs). %In contrast, the baseline models plan to exclusively fulfill online orders from DCs reflecting the inability of single channel or deterministic methods to effectively leverage the cross-channel inventory utilization.
%\noteSS{does bio do more cross-channel pooling compared to unimodal RO- that is a key takeaway if true.}
 
 The average run time across pure RO (unimodal strategy using $\lambda = 0$) and various BIO instances ($\lambda > 0$) are similar as they employ the same run time and iteration limits. We observe an average ($\pm$ standard deviation) run time of 53$\pm$28 seconds with 15$\pm$3 Benders iterations. %\textcolor{red}{To compute global bounds for the CCG heuristic even with approximate solutions to the master and sub-problems, we leverage the LP relaxations and observe these to be ~ **\% on average at termination.} 

\subsubsection{Business Value Assessment:}~\label{BVA}
We use the different inventory models to derive a replenishment policy that is executed weekly in tandem with a fine-grained customer transaction level  Monte-Carlo simulator that generates sample paths to evaluate various KPIs. This experiment replicates a real-world setting wherein individual walk-in and online orders are individually processed in the order in which they arrive. %In Fig.~\ref{fig:whatifplanner} we provide a block diagram of the implementation of this large-scale what-if planner. 
%Customer-level orders by day
Daily customer-level orders are generated via an appropriate spread-down of week-location-level demands. Walk-in orders are met first, and e-commerce orders are processed next based on the cheapest fulfillment cost from a tiered set of origin nodes. The tiered node sets are created based on the inventory available to meet the mean daily predicted walk-in demand prior to the next replenishment. %\notePH{Check with Shiva}%\notePH{This line needs more justification so that we do not invite more experiments}
% \begin{figure}[h]
%   \begin{center}
%     \includegraphics[scale = 0.48]{Figures/WhatIf_Planner.png}
%  \caption{The what-if planner architecture used in the experiments. %\textcolor{red}{New term for envisaged profit}
%  }
% \label{fig:whatifplanner}
% \end{center}
% \end{figure}

We execute the what-if planner every week over a rolling horizon (the 3-week peak period) for 30 Monte Carlo sample paths. While the first standalone experiment required one optimization instance per SKU (23 instances in total), we now solve 30 optimization instances per SKU per week over the chosen time horizon, yielding  a total of 23*30*2=1380 instances; Two consecutive optimization runs are required in a 3 week period because of the one week lead time and noting that the recommended decision in the third week is redundant. We compare the performance of \textcolor{\highlightcolor}{different optimizers: Basestock, PWL, BS Net$_{\theta = 0}$, BS Net$_{\theta^*}$, BS Net$_{\theta = 1}$, in the perfect distributional information setting, BSU Net$_{\theta = 0}$, BSU Net$_{\theta^*}$, BSU Net$_{\theta = 1}$, in the imperfect distributional information case, 
along with pureRO, BIO 10\%, BIO 25\%, BIO 55\% and BIO 50\% in the distribution-free case. Here $\theta^*$ is obtained via a grid search from 0.6 to 1 in increments of 0.1, and we note that this turn out to be 0.9 and 1 for the perfect and imperfect scenarios respectively.} We start with zero inventory in the first week, which automatically translates to an imbalanced mean inventory setting for the following week. The retailer weighed lost sales and achieved sales equally and we employ a price and back order penalty of 1, along with a holding cost of 0 to reflect their preference. The results of the experiments are provided in Fig.~\ref{fig:whatifplanner_results} and Table~\ref{tab:planner_results} respectively. 

\color{\highlightcolor}
We observe from Table~\ref{tab:planner_results} that in the perfect information setting, the network heuristic outperforms the single-channel heuristic at any value of $\theta$ by a significant margin. Specifically, BS Net$_{\theta^*}$ achieves 46\% higher profitability than the vanilla basestock policy, highlighting the value of cross-channel inventory management. However, under imperfect information, the profitability of the same network heuristic BSU Net$_{\theta^*}$ decreases by 15\% compared to its counterpart BS Net$_{\theta^*}$. For a practical comparison, where demand distribution information is inherently unknown, we evaluate the imperfect information network heuristic, BSU Net$_{\theta^*}$, against the BIO models, which follow a distribution-free approach. We observe that profitability increases by up to 10\% and 11\% for BIO 25\% and BIO 35\%, respectively. This gain stems from improved inventory turnover, which increases by approximately 6 points, while the service level declines by 9 percentage points—a trend also observed with the best-case perfect information heuristic, where the service level drops by 6 percentage points, for a similar increase in inventory turn-over.
% pure RO outperform baselines
% %, basestock in particular, 
% by 8\% and 10\% in the realized average profit and inventory turn-over metrics, while BIO 10\% outperforms pure RO by 15\% and 35\% respectively. Note that even though this achieved profit go up with BIO 25\% from BIO 10\% (up to 25\% from 15\%), its 
% penalty-adjusted estimate drops from 29\% to 17\% for BIO 10\% (both w.r.t to pure RO as a baseline) suggesting a possible increase in stock-outs that can lead to a decline in market-share  and hurt long-term growth. This finding is also reflected in a 6\% drop in the predicted service level with BIO 25\% w.r.t to pure RO as compared to a drop of 2.4\% using BIO 10\%. 
\begin{figure}[t]
  \begin{center}
    \includegraphics[scale = 0.43]{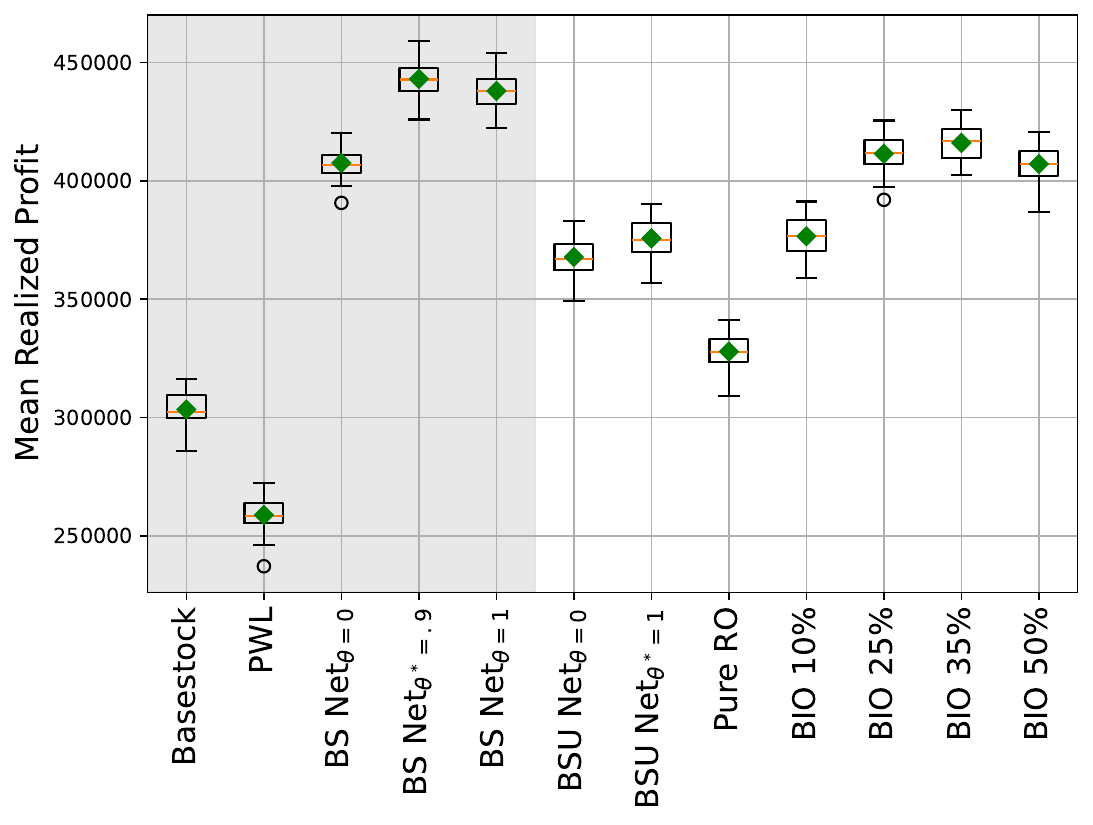}
    \includegraphics[scale = 0.43]{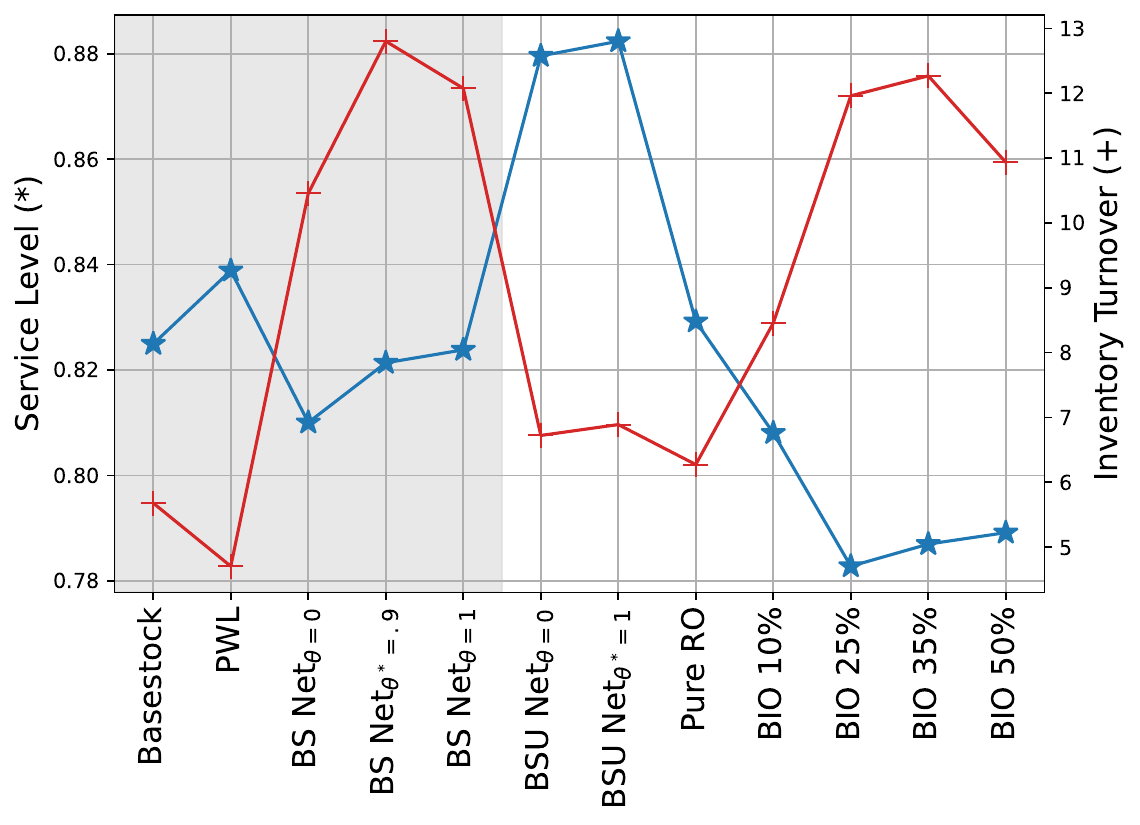}
 \caption{KPIs of different replenishment policies using the  transaction-level what-if simulator. Grey-shaded settings assume perfect distributional information; the rest reflect imperfect/distribution-free settings. }
 \label{fig:whatifplanner_results}
\end{center}
\end{figure}
\begin{table}[htb]
    \centering
    \scalebox{0.56}{
    \begin{tabular}{@{}ccccccccccccccc@{}}
    \toprule[1.0pt]
Distr. & Optimizer	&	Replenish	&	DC Reple- 	&	Walk-in &	Total 	&	Ship-from	&	Satisfied &	Shipping	& Purchase 	&	Excess Inv. Va-		&	E-Com 	&	Total 	&	Inv.	&	 Realized \\
Info. & &	Qty	&	nish Qty	&	Sales Qty	& Sales Qty	& Store Qty	&	Revenue  	& Cost 	& Cost &	lued at Cost 	&	Ser. Lvl.	& Ser. Lvl.	&	Tr.Ov. 	&  Profits \\
\midrule[0.4pt]
\multirow{5}{*}{Perfect}	&	Basestock	&	 12,449 	&	 2,937 	&	 6,390 	&	 9,206 	&	 59 	&	 1,312 	&	 13 	&	 996 	&	 255 	&	0.97	&	 0.82 	&	 5.68 	&	303	$\pm$	7	\\
	&	PWL	&	 13,341 	&	 2,955 	&	 6,545 	&	 9,361 	&	 147 	&	 1,338 	&	 13 	&	 1,066 	&	 313 	&	0.97	&	 0.84 	&	 4.71 	&	259	$\pm$	8	\\
	&	BS Net$_{\theta =0}$	&	 10,770 	&	 1,258 	&	 6,224 	&	 9,039 	&	 1,573 	&	 1,285 	&	 14 	&	 863 	&	 135 	&	0.97	&	 0.81 	&	 10.46 	&	408	$\pm$	7	\\
	&	BS Net$_{\theta^* =.9}$	&	 10,600 	&	 122 	&	 6,370 	&	 9,166 	&	 2,674 	&	 1,307 	&	 15 	&	 849 	&	 111 	&	0.96	&	 0.82 	&	 12.80 	&	443	$\pm$	8	\\
	&	BS Net$_{\theta =1}$	&	 10,718 	&	 -   	&	 6,401 	&	 9,193 	&	 2,793 	&	 1,312 	&	 15 	&	 859 	&	 119 	&	0.96	&	 0.82 	&	 12.08 	&	 438 	$\pm$	8	\\
\midrule[0.1pt]
\multirow{2}{*}{Imperfect}	&	BSU Net$_{\theta =0}$	&	 12,737 	&	 246 	&	 7,000 	&	 9,815 	&	 2,570 	&	 1,396 	&	 15 	&	 1,014 	&	 226 	&	0.97	&	 0.88 	&	 6.72 	&	368	$\pm$	8	\\
	&	BSU Net$_{\theta^* =1}$	&	 12,706 	&	 -   	&	 7,032 	&	 9,846 	&	 2,815 	&	 1,402 	&	 15 	&	 1,011 	&	 221 	&	0.97	&	 0.88 	&	 6.89 	&	 376 	$\pm$	8	\\
\midrule[0.1pt]
\multirow{5}{*}{Distr. free}	&	Pure RO	&	 12,205 	&	 43 	&	 6,446 	&	 9,253 	&	 2,764 	&	 1,321 	&	 15 	&	 978 	&	 232 	&	0.97	&	 0.83 	&	 6.27 	&	328	$\pm$	7	\\
	&	BIO 10\%	&	 11,152 	&	 41 	&	 6,211 	&	 9,017 	&	 2,764 	&	 1,286 	&	 15 	&	 894 	&	 167 	&	0.97	&	 0.81 	&	 8.46 	&	377	$\pm$	8	\\
	&	BIO 25\%	&	 10,198 	&	 53 	&	 5,927 	&	 8,735 	&	 2,756 	&	 1,244 	&	 15 	&	 818 	&	 113 	&	0.97	&	 0.78 	&	 11.96 	&	411	$\pm$	8	\\
	&	BIO 35\%	&	 10,216 	&	 91 	&	 5,976 	&	 8,782 	&	 2,715 	&	 1,252 	&	 15 	&	 821 	&	 113 	&	0.97	&	 0.79 	&	 12.27 	&	416	$\pm$	8	\\
	&	BIO 50\%	&	 10,419 	&	 219 	&	 5,997 	&	 8,806 	&	 2,591 	&	 1,254 	&	 15 	&	 832 	&	 124 	&	0.97	&	 0.79 	&	 10.93 	&	407	$\pm$	8	\\
    \bottomrule[1.0pt]
    \end{tabular}}
    \caption{Average observed metrics and KPIs across Monte Carlo  scenarios. The units of all revenue, cost and profits are in (K)\$. Service level and inventory turnover are referred to as ser. lvl. and inv. tr. ov. respectively. }
    \label{tab:planner_results}
\end{table}
These results underscore the benefits of network modeling in omnichannel settings, as well as the effectiveness of optimistic-robust methods, which are essential when distributional information is limited or imperfect. In such cases, pure robust models alone fail to achieve competitive performance--BIO 35\% yields 27\% higher profitability compared to pure RO.
\color{black}

We note that the achieved gap between the average and worst case (5th quantile) realized profit is nearly identical ($\sim$3\%, not shown in table) for both pure RO and BIO 25\% \textcolor{\highlightcolor}{and 35\%.} {Thus the overall results indicates that a calibrated bimodal inventory strategy can considerably boost average gain without compromising on the (practical) worst case performance. This performance gain can be attributed to BIO's ability to value different inventory positions using a `bifocal' optimistic and pessimistic lens that either improves or preserves the average and (practical) worst case KPIs, while being data-driven and distribution-free.} %\notePH{Question on why only one parameter setting and the favorable one. What does PWL perform this bad? What insights so the previous experiments give us in terms of the set up for this one. }

We recommend the use of the BIO model when it is compatible with the application (e.g., in an omnichannel case) and exercising caution when increasing the value of $\lambda$ beyond 50\% given the potential impact on profitability metrics.

\begin{APPENDICES}

\section{Extension to incorporate the movement of pre-existing inventory across nodes}~\label{DCtoStores}
We show how we extend the main model when in addition to supplier purchase, we consider the movement of pre-existing inventory from different nodes in the network
to other nodes (say warehouse to store allocation of fashion items). Here we refer to $S$ as the supplier node and define (1) $C_{l'l}$ as the per-unit transportation or purchase when shipping from node or supplier $l' \in \Lambda \cup S$ to node $l$ (and use it instead of $C_l$); (2) $L_{ll'}$ as the lead time between nodes $l$ and $l'$ where $L_{l}$, is the maximum lead time for node $l$, in particular, $\max_{l' \in \Lambda \cup S} L_{l'l}$, and (3) $x_{tl}^{l'}$ as the inventory allocated from node $l'\in \Lambda \cup S$ in period $t$ for node $l$.

The model changes as follows. In the objective~\ref{objective}, we replace term $C_{l} x^t_{l}$ by $\sum_{l' \in \Lambda \cup S}  C_{l'l} x^{l'}_{tl}$. Next, we replace constraints (\ref{invBalance}) and (\ref{stateupdate}) as follows:
\begin{align}
&  s^b_{tl}  + \sum_{z\in Z}
   y_{tlz} + I_{t+1,l}^0 = I_{tl}^0  + I_{tl}^1  \mathbbm{1}_{L_{l} >= 0} + \sum_{l' \in \Lambda \cup S} x^{l'}_{tl}  \mathbbm{1}_{L_{l'l}=0} - \sum_{l' \in \Lambda } x^{l}_{tl'} && \forall \ l \in \Lambda,\\
& I^j_{t+1,l} = I^{j+1}_{tl}  \mathbbm{1}_{j < L_l}+ \sum_{l' \in \Lambda \cup S} x^{l'}_{tl} \mathbbm{1}_{L_{l'l}=j} &&  \forall \ 1\leq j \leq L_{l}, \ l \in \Lambda
\end{align}
Also, in the static robust formulation, constraint (\ref{invBalance_new}) can be rewritten as:
\begin{align}
&  s^b_{tl}  + \sum_{z\in Z}
   y_{tlz} + I_{t+1,l} = I_{tl}  + \bar{I}_{l}^{t+1}  \mathbbm{1}_{t<L_l}  + \sum_{l' \in \Lambda \cup S} x_{t-L_{l'l},l}^{l'}\mathbbm{1}_{t \geq L_{l'l}} - \sum_{l' \in \Lambda } x^l_{t l'} && \forall \ l \in \Lambda.
\end{align}

\section{Optimistic-robust modeling to exploit an information edge}~\label{extensions} 
%\noteSS{let's call this out briefly in the main paper as it is closest to the spirit of kelly :)}
%\section{Alternative models for optimistic-robust modeling}~\label{extensions} 
There are alternative ways of developing an allied-adversary. %One approach is to just work with just two independent scenarios $D^+$ and $D^-$ and work with the  superposition of the optimistic and the pessimistic gain (objective) each computed separately. Empirically, we observed that the generalization to new data was better when we operated in the demand space compared to the objective. %This may be because the tuning of the optimistic fraction was quite sensitive to distribution of the resultant objective, while less sensitive to resultant demand. } 
%Even with convex combinations of demand scenarios 
For example, we can have $\lambda$ being time-dependent, channel dependent or even location-dependent. For the uncertainty set we adopt, time or channel dependent $\lambda$ does not significantly affect our results. On the other hand, a location-specific $\lambda$ enables us to model use cases where we can allow pure optimism in some locations, say at most $Q$ in number, and only permit pure adversarial demands in others. The specific optimistic locations (or location clusters) can be configured by the user based on an information `edge', i.e., extraneous demand signals that indicate an increased chance of a sales spike. 
%This can be viewed broadly as a location specific optimism and closely follows the spirit of Kelly's formula. Another continuous variation is to allow the master problem to provide restrictions (lower bounds) to the adversarial demands. The optimistic fraction in these methods is controlled via the upper bound of the budget constraint, the fraction that optimistic can utilize and the remaining fraction that the adversarial can utilize. In either of these methods, 
\textcolor{\highlightcolor}{Note that by tuning a $\lambda$ parameter by location using a validation dataset, we can obtain an SAA-like approximation within our BIO setting. On the other hand,} the challenge with introducing location dependence is that the resultant demand may not longer belong to the uncertainty set we use in our formulation. In an iterative algorithm like CCG, 
the previously elicited adversarial demands together with the current optimistic demand in the master, are not guaranteed to lie within the set $U$ (even though they were feasible to some of the past realizations of the optimistic demand). \textcolor{\highlightcolor}{An elastic uncertainty set can be introduced to mitigate this limitation, and the resultant BIO approximation can be practically useful in its ability to exploit an information edge.}  %One option to side step this issue is to restrict the budgets to be separate for optimism and pessimism but it restricts the resultant scenario. %(e.g., $[\lambda \bar{\bar{D}}^{bL}_{t}, \lambda \bar{\bar{D}}^{bU}_{t}] $) and separate for pessimism (e.g., $[(1-\lambda) \bar{\bar{D}}^{bL}_{t}, (1-\lambda) \bar{\bar{D}}^{bU}_{t}] $). But this on the other hand, restricts the resultant scenario.
%These options can be viewed as variants of the core idea and will need additional exploration. 

% \subsection{Proof of Proposition~\ref{convexityBIO}}
% \begin{proposition*}
% {\em For any given feasible $({\bf x}_t, {\bf D}^+)$, there exists an optimal (worst-case) demand ${\bf D}^-$ which is at an extreme points of $U$. 
% Moreover, finding the optimal $({\bf x}, {\bf D}^+)$ in the ${SBIO-\lambda}({\bf \bar{I}})$ problem is convex optimization problem, even though the inner minimization problem over ${\bf D}^-$ is non-convex.} \notePH{Need to prove}
% \end{proposition*}
% {\em Proof:} \hfill \halmos
\section{Closed form solutions for a single location walk-in setting}\label{closedFormSingle}
\begin{proposition}
{\em For a single location problem with time-horizon $T=1$ and zero lead time, if the uncertainty set $U = [D_{\min}, D_{\max}]$ and initial inventory $I^0 = 0$, the optimal pure robust and best case solutions and their respective objective functions are obtained as follows assuming $p+b\geq c>0$:}
\begin{align}
x^*_{BIO-1} &= \left\{\begin{array}{ll} D_{\max} & \textrm{ if } p\geq c\\ 0 & \textrm{otherwise.} \end{array}\right.,   && Z^*_{BIO-1} = \left\{\begin{array}{ll} (p-c) D_{\max} & \textrm{ if } p\geq c\\ (p-c)D_{\min} & \textrm{otherwise.} \end{array}\right. \\
x^*_{BIO-0} &= \frac{(p+h)D_{\min} +b
            D_{\max}}{(p+b+h)} ,   && Z^*_{BIO-0} = \frac{(p+b-c)(p+h)D_{\min} -(h+c) b
            D_{\max}}{(p+b+h)} 
\end{align}
% x^*_{BIO-0} &= \left\{\begin{array}{ll} \frac{(p+h)D_{\min} +b
%             D_{\max}}{(p+b+h)} & \textrm{ if } p+b\geq c\\ 0 & \textrm{otherwise.} \end{array}\right.,   && Z^*_{BIO-0} = \left\{\begin{array}{ll} \frac{(p+b-c)(p+h)D_{\min} -(h+c) b
%             D_{\max}}{(p+b+h)} & \textrm{ if } p+b\geq c\\ -bD_{\max} & \textrm{otherwise.} \end{array}\right. 
% \end{align}
\end{proposition}
With proposition~\ref{superpositionTheorem}, we can obtain  $x^*_{BIO-\lambda}$ and $Z^*_{BIO-\lambda}$ in closed form. Also, we know that the closed form newsvendor solution in this case is $x^*_{SAA} = VaR_D\left(\frac{p+b-c}{p+b+h}\right)   $ \\
{\em Proof: }
Let us understand the best case and worst demand in a simple
 setting of a single location and single time period. 
%First let us
% begin with the average case setting. We also include all types of unit rewards and costs: $p$ price of item, $b$ lost sales penalty, $h$ holding
% cost, and $c$ purchase cost of item. Here the objective can be written as 
% \begin{align*}
% & \max_{x} \  E_D \left[ \max_{s\leq x, s\leq D} ps - b (D-s) -h(x-s) -cx\right]\\
%  \implies & \max_{x}  \ E_D \left[ (p+b+h) \min\{x,D\}- bD -(h+c)x \right]\\
%  \implies& \max_{x} \ E_D \left[ (p+b+h) \left(x-(x-D)^+\right) - bD -(h+c)x\right]\\
%  \implies& \max_{x} \ E_D \left[ (p+b-c)x- (p+b+h)(x-D)^+ - bD\right]\\
%  \implies& \max_{x} \ E_D \left[ (p+b-c) \left(D +(x-D)^+ - (D-x)^+ \right) - (p+b+h)(x-D)^+ - bD\right]\\
%   \implies& \max_{x} \ E_D \left[ (p-c) D - (p+b-c) (D-x)^+  - (h+c)(x-D)^+ \right]\\
%  \implies & x^* = VaR_D\left(\frac{p+b-c}{p+b+h}\right)   
% \end{align*}
Now let us consider the pure robust setting. 
\begin{align*}
& \max_{x} \min_{D_{\min}<=D <= D_{\max}} (p-c) D - (p+b-c) (D-x)^+  - (h+c)(x-D)^+\\
 \implies & \max_{x} \left\{\begin{array}{ll}(p+h)D_{\min}-(h+c)x \ \ &\textrm{if }   x >= \frac{(p+h)D_{\min} + b
            D_{\max}}{(p+b+h)}\\ (p+b-c) x - bD_{\max}
                                                     &\textrm{otherwise}\end{array}\right.\\
 \implies & x^* =  \frac{(p+h)D_{\min} +b
            D_{\max}}{(p+b+h)} \textrm{ if } p+b >= c\textrm{ and  Opt Obj =}  \frac{(p+b-c)(p+h)D_{\min} -(h+c) b
            D_{\max}}{(p+b+h)} 
\end{align*}
Now let us study the best case problem:
\begin{align*}
& \max_{x} \max_{D_{\min}<=D <= D_{\max}} (p-c) D - (p+b-c) (D-x)^+  - (h+c)(x-D)^+\\
 \implies & \max_{x} \ \ (p-c)x \quad \textrm{where} \ \ D_{\min} \leq x \leq D_{\max}\\
 \implies & x^* = D_{\max} \textrm{ if } p\geq c,  \textrm{ else} \ \ 0.
\end{align*} \hfill \halmos
\color{\highlightcolor}
\section{Proof of Theorem~\ref{superpositionTheorem}}
\begin{theorem*}
 {\em Given $\lambda$, there exists an $x^*_{BIO-\lambda} $ such that $x^*_{BIO-\lambda} = \lambda x^*_{BIO-1} + (1-\lambda)x^*_{BIO-0}$ and the corresponding optimal objective value $Z^*_{BIO-\lambda}$ satisfies $Z^*_{BIO-\lambda} = \lambda Z^*_{BIO-1} + (1-\lambda)Z^*_{BIO-0}$. In other words, for any given $\lambda$, there exists an optimal allocation for $BIO-\lambda$ that is a linear superposition of the respective optimal quantities for the best case scenario (when $\lambda = 1$, i.e., $BIO-1$) and the robust scenario (when $\lambda = 0$, i.e., $BIO-0$). Moreover, the optimal objective value follows  the same superposition structure. }
\end{theorem*}
{\em Proof:} The above result holds when $\lambda$ equals 0 or 1 immediately. So, we will now consider a case when $0< \lambda <1 $. Let us refer to the BIO-$\lambda$ problem as $P^{BIO-\lambda}$. We first show bi-directional feasibility, that is: (1) any feasible solution of the problems $P^{BIO-0}$ and $P^{BIO-1}$, can be combined to arrive at a feasible solution for $P^{BIO-\lambda}$ and (2) any feasible solution to $P^{BIO-\lambda}$ can be disaggregated to arrive at feasible solutions to  $P^{BIO-0}$ and $P^{BIO-1}$.

We use the following notation for feasible solutions for $P^{BIO-\lambda}$,  $P^{BIO-0}$ and $P^{BIO-1}$ respectively: $\{x,D^+_{\lambda}, D^-_{\lambda},s, y, I\}$, $\{x^+,D^+, s^+, y^+, I^+\}$ and $\{x^-, D^-,s^-, y^-, I^-\}$. If $\{x,D^+_{\lambda}, D^-_{\lambda},s, y, I\}$ is a feasible solution to $P^{BIO-\lambda}$, we know that the following equations are satisfied $\forall \ t\in \ca{T}$:
\begin{align}
%\left\lbrace \begin{tabular}{ll}
\label{walkinSalesL} & s^b_{tl} \leq \lambda {D}_{tl}^{b+} + (1-\lambda) {D}_{tl}^{b-}
                                 &&  \forall \ l \in \Lambda,\\
\label{ecomSalesL} &\sum_{l \in\Lambda} 
  y_{tlz} \leq \lambda  {D}_{tz}^{o+} + (1-\lambda)  {D}_{tz}^{o-}  && \forall \    z \in Z,\\
  \eqlabel{invBalance_newL}&  s^b_{tl}  + \sum_{z\in Z}
  y_{tlz} + I_{t+1,l} = I_{tl}  + \bar{I}_{l}^{t+1}  \mathbbm{1}_{t<L_l}  + x_{t-L_{l},l}\mathbbm{1}_{t \geq L_{l}}  && \forall \ l \in \Lambda,\\
 &  I_{0l} = \bar{I}_{l}^0 && \forall \ l \in \Lambda.
\end{align}
If $\{x^+,D^+, s^+, y^+, I^+\}$ and $\{x^-, D^-,s^-, y^-, I^-\}$ are feasible solutions to $P^{BIO-0}$ and $P^{BIO-1}$ respectively, we know that the following equations are satisfied $\forall \ t\in \ca{T}$:
\begin{align}
%\left\lbrace \begin{tabular}{ll}
\label{walkinSales1} & s^{b+}_{tl} \leq {D}_{tl}^{b+}, \quad  s^{b-}_{tl} \leq {D}_{tl}^{b-} &&  \forall \ l \in \Lambda,\\
\label{ecomSales1} &\sum_{l \in\Lambda} 
  y_{tlz}^+ \leq{D}_{tz}^{o+}, \quad \sum_{l \in\Lambda} 
  y_{tlz}^- \leq {D}_{tz}^{o-}   && \forall \    z \in Z,\\
  \eqlabel{invBalance_new1}&  s^{b+}_{tl}  + \sum_{z\in Z}
  y_{tlz}^+ + I_{t+1,l}^+ = I_{tl}^+  + \bar{I}_{l}^{t+1}  \mathbbm{1}_{t<L_l}  + x_{t-L_{l},l}^+\mathbbm{1}_{t \geq L_{l}}  && \forall \ l \in \Lambda,\\
  \eqlabel{invBalance_new2}&  s^{b-}_{tl}  + \sum_{z\in Z}
  y_{tlz}^- + I_{t+1,l}^- = I_{tl}^-  + \bar{I}_{l}^{t+1}  \mathbbm{1}_{t<L_l}  + x_{t-L_{l},l}^-\mathbbm{1}_{t \geq L_{l}}  && \forall \ l \in \Lambda\\
 &  I_{0l}^+ = \bar{I}_{l}^0, \quad I_{0l}^- = \bar{I}_{l}^0 && \forall \ l \in \Lambda
\end{align}

It is easy to see that every feasible solution of the problems $P^{BIO-0}$ and $P^{BIO-1}$ can be combined  with scaled inventory and fulfillment variables to derive a feasible solution of  $P^{BIO-\lambda}$ as follows: 
$\{x,s,y, I\}_{\lambda}$ = $\lambda \{x^+,s^+, y^+, I^+\} + (1-\lambda)\{x^-,s^-, y^-,I^-\}$, $\{D^+\}_{\lambda} = \{D^+\}$ and $\{D^-\}_{\lambda} = \{D^-\}$ where all operations are element-wise.
% \begin{align}
%  & \lambda s^{b+}_{tl} + (1-\lambda) s^{b-}_{tl} = s^{b}_{tl}\\
%  & \lambda y_{tlz}^{+} + (1-\lambda) y_{tlz}^{-} = y_{tlz}\\
%  & \lambda I_{t+1,l}^+  + (1-\lambda) I_{t+1,l}^-  = I_{t+1,l}\\
%  & \lambda x_{t-L_{l},l}^+ + (1-\lambda) x_{t-L_{l},l}^- = x_{t-L_{l},l}
% \end{align}
And because the fulfillment and inventory variables are scaled, along with the demand, the resultant objective is a superposition of the two objectives\footnote{\textcolor{\highlightcolor}{Note that if we used a channel specific $\lambda$ (as we did in in equations~\ref{bioFormulation}-\ref{bio_initialInv}), then we use the same channel specific $\lambda$ for the fulfillment variables here. The proof is indifferent to this choice and works in the general case too.}}, i.e., 
$Z_{BIO-\lambda} = \lambda Z_{BIO-1} + (1-\lambda)Z_{BIO-0}$, where the objectives correspond to the feasible solutions being considered.

% Let us refer to the BIO-$\lambda$ problem as $P^{BIO-\lambda}$ and let $\{x,D^+, D^-,s, y, I\}_{\lambda}$ refer to a feasible solution of $P^{BIO-\lambda}$. It is easy to see that every feasible solution of the problems $P^{BIO-0}$ and $P^{BIO-1}$ can be combined  with scaled inventory and fulfillment variables to derive a feasible solution of  $P^{BIO-\lambda}$ as follows\footnote{Note that if we used a channel specific $\lambda$ (as we did in in equations~\ref{bioFormulation}-\ref{bio_initialInv}), then we use the same channel specific $\lambda$ for the fulfillment variables here. The proof is indifferent to this choice and works in the general case too.}: $\{x,s,y, I\}_{\lambda}$ = $\lambda \{x,s, y, I\}_{1} + (1-\lambda)\{x,s, y,I\}_{0} $, $\{D^+\}_{\lambda} = \{D^+\}_{1}$ and $\{D^-\}_{\lambda} = \{D^-\}_{0}$ where all operations are element-wise. And because the fulfillment and inventory variables are scaled, the resultant objective is a superposition as well.  

We will now show that any feasible solution to the problem $P^{BIO-\lambda}$ can be used to generate a feasible solution to the problems $P^{BIO-0}$ and $P^{BIO-1}$.  Note that the constraints of $P^{BIO-\lambda}$ can be viewed as linear superpositions of the respective
constraints of the two problems $P^{BIO-0}$ and $P^{BIO-1}$, barring the respective upper and lower bounds of the individual variables. So, there may be many feasible solutions to the disaggregated set of constraints and here we show how we generate one of these.
We let the demand remain the same across the problems, i.e., $\{D^+\}_{P^{BIO-1}} = \{D^+\}_{\lambda}$ and $\{D^-\}_{P^{BIO-0}} = \{D^-\}_{\lambda}$. We first generate sales and fulfillment variables to satisfy the constraints (\ref{walkinSales1}-\ref{ecomSales1}) that ensures that the sales are less than the channel-specific demand in every period and location. 
%demand-less- than-sales constraint \noteSS{better to give the number? constraints (\ref{walkinSales}-\ref{ecomSales}) ensures that the sales is less than the channel specific demand in every period and location. } 
The sales less than demand constraints are separable by period and so the following discussion regarding fulfillment variables $(s,y)$ is per period and we drop the subscript $t$ for convenience. We begin with the walk-in channel where we can further decompose the constraints by location. 
% If $\{s\}^+_{\lambda}$ and $\{s\}^-_{\lambda}$ exist we use them directly after scaling down by $\lambda$ and $(1-\lambda)$ respectively. If they are not present, 
we set the sales values as follows: $\{s^+\} = \min \left\{ \{D\}^{+}, \{s\}/ \lambda \right\}$ and $\{s^-\} =  (\{s\}  - \lambda \{s^+\})/(1-\lambda)$. Note that because of the way we set the optimistic sales, the pessimistic portion of the sales also satisfies the sales less than its respective demand constraint (\ref{walkinSales1}). For the $y$ variables, we first order the location in an enumerated sequence, say $l = 1,2,...|\Lambda|$ and then $y_{lz}^+ = \min\{ \max \{0,D^+_{z} -\sum_{l'<l} y_{l'z}^+ \} , y_{lz,\lambda}/\lambda\}$ and $y_{lz}^- =  (y_{lz,\lambda} - \lambda y_{lz}^+ )/(1-\lambda)$. Similarly to the walk-in case, these fulfillment values also satisfy the counterpart constraint (\ref{ecomSales1}). 

What remains to be identified are $\{x^+,I^+\}$ and $\{x^-,I^-\}$ for each location $l \in \Lambda$. %Without loss of generality, we set $I^+_T = 0, I^-_T = I_{T}/(1-\lambda)$ as $T$ is left over inventory after all the periods and the initial inventory $I_{0l}^+,I_{0l}^-$ which are set to $\bar{I}_{l}^0$. 
For simplified notation, we let $S^+_{tl} = s^{b+}_{tl} + \sum_{z\in Z} y_{tlz}^+$ and $S^-_{tl} = s^{b-}_{tl} + \sum_{z\in Z} y_{tlz}^-$. 
%Also, we use notation  $\hat{x}_{tl} = \bar{I}_{l}^{t+1}  \mathbbm{1}_{t<L_l}  + x_{t-L_{l},l}\mathbbm{1}_{t \geq L_{l}} $ and similarly define $\hat{x}_{tl}^+, \hat{x}_{tl}^-$. We first aim to identify $\{\hat{x}^+,I^+\}$ and $\{\hat{x}^-,I^-\}$ given $S^+_{tl}m S^-{tl}$ using the following equations. 
We skip subscript $l$ as the same logic follows for each $l\in \Lambda$.
% \begin{align}
% &  S^{+}_{t}  + I_{t+1}^+ = I_{t}^+ + \hat{x}^+_t, \quad S^{-}_{t}  + I_{t+1}^- = I_{t}^- + \hat{x}^-_t \quad \quad \forall \ t \in \{0,...,T-1\}\\
% %&  S^{-}_{t}  + I_{t+1}^- = I_{t}^- + \hat{x}^-_t \quad t \in \{0,...,T-1\}\\
%  & \lambda \hat{x}_{t}^+  + (1-\lambda) \hat{x}_{t}^-  = \hat{x}_{t} \quad \forall \ t \in \{0,1,...,T-1\}\\
%  & \lambda I_{t}^+  + (1-\lambda) I_{t}^-  = I_{t} \quad \forall \  t \in \{1,...,T-1\}\\
%    & I_{0}^+ = I_{0}^- = \bar{I}^0 %, \ I^+_T = 0, \ I^-_T = I_{T}/(1-\lambda)
% \end{align}
We first identify the optimistic variables $\{x^+,I^+\}$ with $S^+_{t}$ as inputs by solving the following LP. Here we choose an arbitrary finite $c>h>0$ (motivating purchasing as little inventory as needed with minimal holding to meet sales $S^+_{t}$):
\begin{align}
& \min_{I^+, {x}^+ \geq 0}  \ \ \sum_{t=0}^T \left[h I_{t+1}^+ + c \hat{x}^+_t\right] \\
&  S^{+}_{t}  + I_{t+1}^+ = I_{t}^+ + \bar{I}^{t+1}  \mathbbm{1}_{t<L}  + x_{t-L}^+\mathbbm{1}_{t \geq L_{l}} 
 \quad \forall \  t \in \{0,...,T-1\}\\
& \lambda  {x}^+_t  \leq {x}_t \quad \forall \  t \in \{0,...,T-1\}\\
& I_{0}^+  = \bar{I}^0
\end{align}
We know this is feasible as $P^{BIO-\lambda}$ was feasible and because $\lambda S^+_t \leq S_t$ per our earlier selection.
We then set $\{I^-\} = (\{I\}_{\lambda} - \lambda \{I^+\})/(1-\lambda)$ and $\{\hat{x}^-\} = (\{\hat{x}\}_{\lambda} - \lambda \{\hat{x}^+\})/(1-\lambda)$. Observe that by design $\{\hat{x}^+,I^+\}$ and $\{\hat{x}^-,I^-\}$ are non-negative and $\{\hat{x}^-, S^-, I^-\}$ satisfy the inventory balance constraints of $P^{BIO-0}$. This can be gathered by eliminating the terms related to $\{\hat{x}^+, S^+, I^+\}$  in $P^{BIO-\lambda}$. Moreover, because the $\lambda-$superposition of each feasible variable is maintained, the $BIO-0, BIO-1$ objectives also satisfy the $\lambda-$superposition, i.e.,  $Z_{BIO-\lambda} = \lambda Z_{BIO-1} + (1-\lambda)Z_{BIO-0}$.

% What remains to be identified are $\{x,I\}_1$ and $\{x,I\}_0$ for each location $l \in \Lambda$. Without loss of generality we set $\{I\}_1 = 0$ for all time periods and set $\{I\}_0 = \{I\}_{\lambda}/(1-\lambda)$. Given $\{s,y,I\}_1$ are fixed, using the inventory balance constraints (\ref{invBalance_new}) of problem  $P^{BIO-1}$, we can immediately get the $\{x\}_1$ values. And using this we can set $\{x\}_0$ as $\left(\{x\}_\lambda - \lambda \{x\}_1\right)/(1-\lambda)$. Observe that by design $\{x,I\}_1$ and $\{x,I\}_0$ are non-negative and $\{x, s, y, I\}_0$ satisfy the inventory balance constraints of $P^{BIO-0}$. This can be gathered by eliminating the terms related to $\{x, s, y, I\}_1$  in $P^{BIO-\lambda}$. 

% We leverage the following system of equations to identify them (1) $\lambda \{x,I\}_1+ (1-\lambda)\{x,I\}_0 = \{x,I\}_{\lambda}$ which are a system of $2T$ equations for each location and (2) a system of linear inventory balancing equations each problem $P^{BIO-0}$ and $P^{BIO-1}$ which are again $2T$ equations for each location with initial inventory 0 and sales variables as identified as above for each problem.  So, with $4T$ equations and $4T$ unknowns per location, we know this system has a unique solution as long as the equations are not inconsistent or linearly dependent. Note that none of them are linearly dependent or inconsistent as they have different variables in every equation to begin with. \notePH{Are these x and I's non-negative?}

Now having shown bi-directional feasibility, we begin with the optimal solution to problem  
$P^{BIO-\lambda}$. We then generate feasible solutions to problems $P^{BIO-0}$ and $P^{BIO-1}$. Each of these can be improved to its respective optimal solution which is again feasible to $P^{BIO-\lambda}$ which by definition cannot have a better objective than the optimal solution we began with. This means each step here resulted in an optimal transition and therefore there exists an optimal allocation $x$ to $P^{BIO-\lambda}$, i.e., $x^*_{BIO-\lambda} $ such that $x^*_{BIO-\lambda} = \lambda x^*_{BIO-1} + (1-\lambda)x^*_{BIO-0}$ and it corresponds to an optimal objective $Z^*_{BIO-\lambda}$ that satisfies $Z^*_{BIO-\lambda} = \lambda Z^*_{BIO-1} + (1-\lambda)Z^*_{BIO-0}$. Hence, the theorem. \hfill \halmos

\color{black}

\section{Proof of Proposition~\ref{Integral_ExtremePoints}}
\begin{proposition}
 The extreme points of the uncertainty set $U$ described in Eq.(\ref{eqn:uncertainty_set}) with integral bounds are integral. 
\end{proposition}
\begin{proof} {\em Proof:} Because the set $U$ in Eq.(\ref{eqn:uncertainty_set}) decomposes by channel and time period, without loss of generality we only present the result for the  walk-in channel and for a single time period by dropping the time subscript. To prove integrality, we consider the following LP where $c_l \ \forall l \in \Lambda$ can be chosen to be any vector of constants:
\begin{align}
    \min_{D_l} \  \  & \sum_{l\in \Lambda} c_l D_l\\
    & \bar{D}^{L}_{l}\leq {D}_{l} \leq \bar{D}^{U}_{l} \qquad \forall \ l \in \Lambda \\
    &\bar{\bar{D}}^{L}\leq \sum_{l \in \Lambda} {D}_{l} \leq \bar{\bar{D}}^{U} 
\end{align}
Consider any corresponding optimal solution of this LP. We want to show that there exists an extreme point solution to the LP that is integral and optimal assuming $\bar{D}^{L}_{l}, \bar{D}^{U}_{l} \ \forall l \in \Lambda$ and $ \bar{\bar{D}}^{L}, \bar{\bar{D}}^{U}$ are integral. 

Case 1: If all but one location say $l'$ have integral $D_l$ then $D_{l'}$ can also be chosen to be integral by decreasing or increasing to the nearest integer depending on whether $c_{l'} \leq 0 $  or vice versa, and without loss of generality, this move will only improve the optimal objective. 

Case 2: Next, suppose there exists at least one pair of locations that have non-integral solutions. Here, we can decrease the demand for the one with lower $c_l$ and increase the demand in the other location with higher $c_l$ (if the costs are equal, then pick any one location arbitrarily) until one of demand variables hits a local lower or upper bound without impacting the budget constraint. At this point we have decreased the number of locations with non-integral demands by 1 and improved the objective, if at all. If we repeat this procedure for every pair of non-integral demand locations, we will reduce the problem to the case of non-integrality at no more than one location, which can be resolved per Case-1 discussed earlier. 
%We proved above that in this case the single location can be made integral again without impacting the objective.
\hfill \halmos
\end{proof}

\section{Proof of Proposition~\ref{RLT_subproblem}}
\begin{proof} {\em Proof:} Constraints (\ref{wsuperposition}-\ref{wbrickbudget}) together with the binary nature of $w$ describe the discrete elements of the uncertainty set $U$. We know this is a just a sub-set of $U$ but contains the optimal adversarial demand solution to the problem due to Proposition~\ref{Integral_ExtremePoints}.
%Note in the budget constraints (\ref{wonlinebudget}-\ref{wbrickbudget}) are implicitly we are leveraging the superposition constraints: 
Note that we are implicitly enforcing the following superposition constraints in the budget restrictions (\ref{wonlinebudget}-\ref{wbrickbudget}) : 
$ {D}_{tl}^b = \sum_{k \in K^b_{tl}} \overline{D}_{tlk}^b
  w^b_{tlk}; \quad
    {D}_{tz}^o = \sum_{k \in K^o_{tz}}\overline{D}_{tzk}^o
  w^o_{tzk}$. We linearize the bilinear terms in the objective by setting $\alpha^*_{tlk}  = \alpha_{tl} w^b_{tlk} $ and
$\beta^*_{tlk}  = \beta_{tl} w^o_{tlk} $ and employ the relaxation and linearization technique (RLT) to exactly linearize these terms utilizing the binary nature of variables $w$ and the definition of $M$ to get constraints (\ref{Malphastar}-\ref{alphastar}). \hfill \halmos

Note that $M$ can be tightened by having it tailored by time period and node or zone in an actual implementation, more specifically with $M_{tl} = p^b_{tl} + b^b_{tl} + (T-t)h_l $, $M_{tz} = p^o_{t} + b^o_{t} + \max_{l \in \Lambda} \{(T-t)h_l-c_{lz}\} $.
\end{proof}
\section{Adaptation of the CCG method to solve BIO-$\lambda$ problem}~\label{CCGmaster}
We incorporate the exact MIP sub-problem given in Proposition~\ref{RLT_subproblem} within the column and cut generation (CCG) framework proposed by~\citet{zeng2013solving} to solve two-stage linear optimization problems. {An important aspect of our sub-problem is that its feasibility is guaranteed regardless of the values of its inputs $({\bf x}, {\bf s}^{b+})$, i.e., the solution space is non-empty, because zero fulfillment is always a feasible solution. This means we satisfy the {\em relatively complete recourse} assumption.}  This holds even in the presence of additional business rules such as fulfillment capacity constraints or fulfillment time-window goals, as a zero fulfillment is readily available as a feasible solution. The CCG framework by~\citet{zeng2013solving} only address problems that satisfies this property, which suffices for our setting. For generalizations, refer to \citet{bertsimas2018scalable}'s extension of the CCG framework.

%We include the master problem formulation and the corresponding CCG algorithm for completeness. 
%We refer to the inner minimization problem as the sub-problem and denote it by $SP({\bf x},{\bf s}^{b+})$. Note that for ease of notation we drop the dependence on ${\bf \bar{I}}$ which is the initial state of the system. 
In the CCG method, a master problem is formulated using a finite set of feasible ${\bf D}^{-}$ uncertainty values. Let $\{{\bf D}^1,\cdots, {\bf D}^j\}$ a set of feasible demand vectors for ${\bf D}^- \in U%_{\mathbb{Z}} 
$. We formulate the master problem as follows which we denote by $MP_j = MP(\{{\bf D}^1,\cdots, {\bf D}^j\})$:
{\allowdisplaybreaks[1]
\begin{align}
& Z^{MP_j} = 
 \max_{\substack{\eta, {\bf D}^+ \in U%_{\mathbb{Z}} 
 \\(x,s^+, s^-,y, I)\geq 0}}  \qquad \eta  + \sum_{t \in \ca{T}} \sum_{l \in \Lambda} 
   \left[(p^b_{tl} +b^b_{tl} ) s^{b+}_{tl} - b^b_{tl} \lambda D^{b+}_{tl} - 
    C_{l} x_{tl}
%   \sum_{l' \in \Lambda \cup S}  C_{l'l} x^t_{l'l}
                           \right]   \\
\nonumber   & \textrm{subject to}\\
\nonumber & \ \eta \leq  \sum_{t \in \ca{T},l \in \Lambda} \left[
   (p^b_{tl} +b^b_{tl}) s^{bi-}_{tl}             + \sum_{z \in Z}
                  (p_t^o +b^o_t -c_{lz}  )y_{tlz}^i  - b^b_{tl} (1-\lambda) D^{bi-}_{tl}  - h_l I_{t+1,l}^i\right] 
                 % \\& \hspace{2cm} 
                  - \sum_{t \in \ca{T},z \in Z}  b_t^o D_{tz}^{oi}
                 \\&   && \hspace{-4cm} \forall i < j\\
&  \ s^{b+}_{tl} \leq  \lambda {D}_{tl}^{b+}   && \hspace{-4cm} \forall \ l \in \Lambda, t \in \ca{T}\\
&  \ s^{bi-}_{tl} \leq   (1- \lambda) {D}_{tl}^{bi-}   && \hspace{-4cm} \forall \ l \in \Lambda, t \in \ca{T}, i <j\\
&  \ \sum_{l \in
  \Lambda} y^i_{tlz} \leq {D}_{tz}^{oi}  && \hspace{-4cm} \forall \    z \in
                                                   Z, t \in \ca{T},i<j\\
&  \ s^{b+}_{tl}+ s^{bi-}_{tl}  + \sum_{z\in Z}
  y_{tlz}^i + I^i_{t+1,l} = I^i_{tl}  + \bar{I}_{l}^{t+1}
  \mathbbm{1}_{t<L_l} + x_{t-L_{l},l} \mathbbm{1}_{t \geq  L_{l}} 
 % + \sum_{l' \in \Lambda \cup S} x^{t-L_{l'l}}_{l'l} \mathbbm{1}_{t \geq  L_{l'l}} - \sum_{l' \in \Lambda } x^t_{l l'} 
 && \hspace{-4cm} \forall \ l \in \Lambda,
                                               t \in \ca{T} , i<j\\
 &  \ I_{0l}^i = \bar{I}_{0l}^0&&  \hspace{-4cm} \forall \ l \in
                                                       \Lambda, i<j 
  \end{align}
}     
The CCG algorithm given below %in Algorithm~\ref{alg:pureRO_Benders} 
solves the master and the sub-problem alternatively until the desired tolerance-based stopping criterion is satisfied. This solution technique can be viewed as the master program generating a sequence of beneficial allocations, the sub-problem adding corresponding adversarial demand `cuts' to master and the iterations stop when the master problem converges in their achieved objective. %\noteSS{we may already have mentioned the prior comment in the main paper, skip if possible} %\notePH{Comment about how optimistic can be continuous sets and adversarial is discrete set}
\begin{algorithm}
\caption{Column and Cut Generation (CCG) Approach to inventory allocation}\label{alg:pureRO_Benders}
{\bf Input}: Stopping criteria tolerance $\epsilon$ and a small number $\delta$ (say $10^{-5}$)\\
{\bf Initialize}: $j = 0$, $LB^0 = -\infty$, $UB^0 = +\infty$, demand sample set $U^0 = \{{\bf D}^0\}$, ${\bf x}^* = {\bf 0} $
  \begin{henumerate}
\item j := j+1
\item Solve master problem $MP(U^j)$ to get $({\bf x}^j, {\bf s}^{b+,j}, \eta^j)$ and update $UB^j = Z^{MP_j}$.  
\item Solve sub-problem $SP({\bf x}^j, {\bf s}^{b+,j})$ to get ${\bf D}^{j}$ and  update $LB^j= \max \{LB^{j-1}, Z^{SP({\bf x}^j, {\bf s}^{b+,j})} + Z^{MP_j} - \eta^j\}$ and $U^j = U^{j-1} \cup \{{\bf D}^{j}\}$. Also, if $LB^j= Z^{SP({\bf x}^j, {\bf s}^{b+,j})} + Z^{MP_j} - \eta^j$, update ${\bf x}^* = {\bf x}^j$ 
\item If $\frac{UB^j-LB^j}{|LB^j| + \delta} \leq \epsilon$, return ${\bf x}^*$.
Otherwise go to step 1.
\end{henumerate}
\end{algorithm}
%\vspace{-0.8cm}
\begin{proposition}
\citep{zeng2013solving} Algorithm~\ref{alg:pureRO_Benders} converges to an optimal solution of SBIO-$\lambda$ problem in finite iterations. 
%\notePH{Need to show that Benders method here retrieves the exact optimal as a theorem. Why this proof is different from prior  work.}
\end{proposition} 

Note that the CCG method and its finite convergence to optimality is independent of the type of continuous or integral set chosen for the allocations $x$ and $D^+$. It assumes that an oracle is used to solve the sub-problem to optimality (or alternatively each time a new extreme point or discrete demand is elicited and incorporated).

%Due to practical computational resources, in our implementation on real-world large scale problem instances, 
%\noteSS{pls see the revised description below}
%We adapt our proposed solution methodology to implement a practical CCG heuristic that is able to find high quality allocations to large real-life instances within a reasonable run time. First, we add practical safeguards such as additional stopping criteria based on minimum and maximum number of iterations and a runtime limit. Next, 
Note also that we return the allocation ${\bf x}^*$, defined as the allocation that yields the best lower-bound computed this far, instead of ${\bf x}^j$, the current master problem solution.  %Essentially, we update the best lower bound estimate (instead of working the current estimate) which 
%This is different from prior papers that discuss the CCG method or its variants. 
The intuition is that the sub-problem objective captures the true value of an allocation and that might not increase through the iterations as the algorithm explores various first stage solutions. We include a pictorial representation of this intuition in Fig.~\ref{fig:benders_intuition}. With this modification and other practical safeguards one may impose such as tolerance or time limits, we end up finding the best solution with the available resources, and noting that more iterations/time leads to improvement in solution quality. 
\begin{figure}[t]
  \begin{center}
     \includegraphics[scale=0.3]{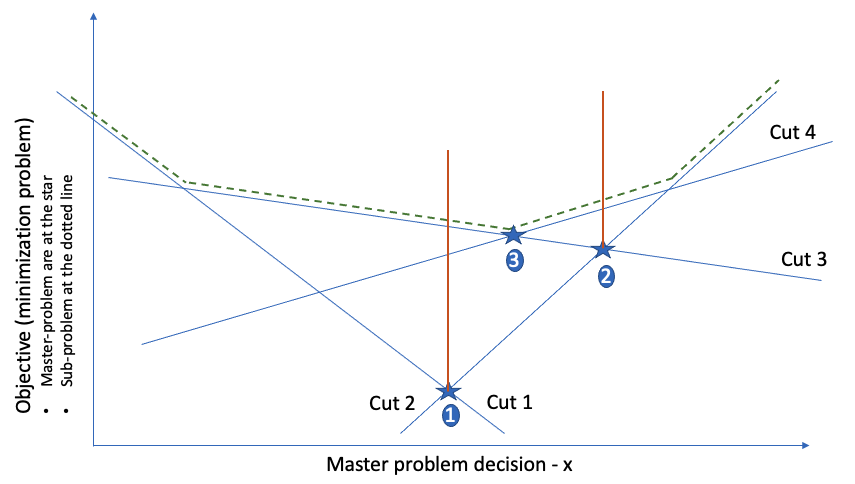}
 \caption{The intuition of how a cut-generation algorithm would proceed for a minimization problem. Observe how the master problem objective serves as a lower bound and the sub-problem objective as an upper bound. Also, observe that at the end of iteration 2, the best solution $x$ is that of iteration 1 as the sub-problem objective is better.}
 \label{fig:benders_intuition}
\end{center}
\end{figure}

\end{APPENDICES}
% Acknowledgments here
% \ACKNOWLEDGMENT{%
% % Enter the text of acknowledgments here
% }% Leave this (end of acknowledgment)

 \bibliographystyle{informs2014}
\bibliography{References}

\begin{thebibliography}{46}
\providecommand{\natexlab}[1]{#1}
\providecommand{\url}[1]{\texttt{#1}}
\providecommand{\urlprefix}{URL }

\bibitem[{Acimovic \protect\BIBand{} Graves(2015)}]{acimovic2015making}
Acimovic J, Graves SC (2015) Making better fulfillment decisions on the fly in
  an online retail environment. \emph{Manufacturing \& Service Operations
  Management} 17(1):34--51.

\bibitem[{Acimovic \protect\BIBand{} Graves(2017)}]{acimovic2017mitigating}
Acimovic J, Graves SC (2017) Mitigating spillover in online retailing via
  replenishment. \emph{Manufacturing \& Service Operations Management}
  19(3):419--436.

\bibitem[{Ali et~al.(2017)Ali, Michael, Ajay, Arun, Yingjie, Xuan, Christopher,
  Brian, Raymond, \protect\BIBand{} Yada}]{ali2017optimization}
Ali K, Michael D, Ajay D, Arun H, Yingjie L, Xuan L, Christopher M, Brian QL,
  Raymond T, Yada Z (2017) An optimization framework for a multi-objective
  omni-channel order fulfilment problem. \emph{IIE Annual Conference.
  Proceedings}, 2051--2056 (Institute of Industrial and Systems Engineers
  (IISE)).

\bibitem[{Bandi \protect\BIBand{} Bertsimas(2012)}]{bandi2012tractable}
Bandi C, Bertsimas D (2012) Tractable stochastic analysis in high dimensions
  via robust optimization. \emph{Mathematical programming} 134(1):23--70.

\bibitem[{Bandi et~al.(2019)Bandi, Han, \protect\BIBand{}
  Nohadani}]{bandi2019sustainable}
Bandi C, Han E, Nohadani O (2019) Sustainable inventory with robust
  periodic-affine policies and application to medical supply chains.
  \emph{Management Science} 65(10):4636--4655.

\bibitem[{Ben-Tal et~al.(2004)Ben-Tal, Goryashko, Guslitzer, \protect\BIBand{}
  Nemirovski}]{ben2004adjustable}
Ben-Tal A, Goryashko A, Guslitzer E, Nemirovski A (2004) Adjustable robust
  solutions of uncertain linear programs. \emph{Mathematical programming}
  99(2):351--376.

\bibitem[{Bertsimas et~al.(2011)Bertsimas, Brown, \protect\BIBand{}
  Caramanis}]{bertsimas2011theory}
Bertsimas D, Brown DB, Caramanis C (2011) Theory and applications of robust
  optimization. \emph{SIAM review} 53(3):464--501.

\bibitem[{Bertsimas et~al.(2018)Bertsimas, Gupta, \protect\BIBand{}
  Kallus}]{bertsimas2018robust}
Bertsimas D, Gupta V, Kallus N (2018) Robust sample average approximation.
  \emph{Mathematical Programming} 171:217--282.

\bibitem[{Bertsimas \protect\BIBand{} Shtern(2018)}]{bertsimas2018scalable}
Bertsimas D, Shtern S (2018) A scalable algorithm for two-stage adaptive linear
  optimization. \emph{arXiv preprint arXiv:1807.02812} .

\bibitem[{Bertsimas \protect\BIBand{} Sim(2004)}]{bertsimas2004price}
Bertsimas D, Sim M (2004) The price of robustness. \emph{Operations research}
  52(1):35--53.

\bibitem[{Bourlier(2020)}]{Bourlier2020}
Bourlier A (2020) \emph{Retail {D}ive}
  \urlprefix\url{https://www.retaildive.com/news/how-covid-19-prompted-new-thinking-for-omnichannel/592010/}.

\bibitem[{DeValve et~al.(2023)DeValve, Wei, Wu, \protect\BIBand{}
  Yuan}]{devalve2023understanding}
DeValve L, Wei Y, Wu D, Yuan R (2023) Understanding the value of fulfillment
  flexibility in an online retailing environment. \emph{Manufacturing \&
  service operations management} 25(2):391--408.

\bibitem[{Gallino \protect\BIBand{} Moreno(2014)}]{gallino2014integration}
Gallino S, Moreno A (2014) Integration of online and offline channels in
  retail: The impact of sharing reliable inventory availability information.
  \emph{Management Science} 60(6):1434--1451.

\bibitem[{Gallino et~al.(2017)Gallino, Moreno, \protect\BIBand{}
  Stamatopoulos}]{gallino2017channel}
Gallino S, Moreno A, Stamatopoulos I (2017) Channel integration, sales
  dispersion, and inventory management. \emph{Management Science}
  63(9):2813--2831.

\bibitem[{Gao \protect\BIBand{} Su(2017{\natexlab{a}})}]{gao2017omnichannel}
Gao F, Su X (2017{\natexlab{a}}) Omnichannel retail operations with
  buy-online-and-pick-up-in-store. \emph{Management Science} 63(8):2478--2492.

\bibitem[{Gao \protect\BIBand{} Su(2017{\natexlab{b}})}]{gao2017online}
Gao F, Su X (2017{\natexlab{b}}) Online and offline information for omnichannel
  retailing. \emph{Manufacturing \& Service Operations Management}
  19(1):84--98.

\bibitem[{Gorissen et~al.(2015)Gorissen, Yan{\i}ko{\u{g}}lu, \protect\BIBand{}
  den Hertog}]{gorissen2015practical}
Gorissen BL, Yan{\i}ko{\u{g}}lu {\.I}, den Hertog D (2015) A practical guide to
  robust optimization. \emph{Omega} 53:124--137.

\bibitem[{Gotoh et~al.(2023)Gotoh, Kim, \protect\BIBand{} Lim}]{gotoh2023data}
Gotoh Jy, Kim MJ, Lim AE (2023) A data-driven approach to beating {SAA}
  out-of-sample. \emph{Operations Research} .

\bibitem[{Govindarajan et~al.(2021{\natexlab{a}})Govindarajan, Sinha,
  \protect\BIBand{} Uichanco}]{govindarajan2021distribution}
Govindarajan A, Sinha A, Uichanco J (2021{\natexlab{a}}) Distribution-free
  inventory risk pooling in a multilocation newsvendor. \emph{Management
  Science} 67(4):2272--2291.

\bibitem[{Govindarajan et~al.(2021{\natexlab{b}})Govindarajan, Sinha,
  \protect\BIBand{} Uichanco}]{JolineNaval}
Govindarajan A, Sinha A, Uichanco J (2021{\natexlab{b}}) Joint inventory and
  fulfillment decisions for omnichannel retail networks. \emph{Naval Research
  Logistics (NRL)} 68(6):779--794.

\bibitem[{Harsha et~al.(2019{\natexlab{a}})Harsha, Subramanian,
  \protect\BIBand{} Ettl}]{harsha2019practical}
Harsha P, Subramanian S, Ettl M (2019{\natexlab{a}}) A practical price
  optimization approach for omnichannel retailing. \emph{INFORMS Journal on
  Optimization} 1(3):241--264.

\bibitem[{Harsha et~al.(2019{\natexlab{b}})Harsha, Subramanian,
  \protect\BIBand{} Uichanco}]{OCPX}
Harsha P, Subramanian S, Uichanco J (2019{\natexlab{b}}) Dynamic pricing of
  omnichannel inventories. \emph{Manufacturing \& Service Operations
  Management} 21(1):47--65.

\bibitem[{Iancu \protect\BIBand{} Trichakis(2014)}]{iancu2014pareto}
Iancu DA, Trichakis N (2014) Pareto efficiency in robust optimization.
  \emph{Management Science} 60(1):130--147.

\bibitem[{Jackson et~al.(2019)Jackson, Muckstadt, \protect\BIBand{}
  Li}]{jackson2019multiperiod}
Jackson PL, Muckstadt JA, Li Y (2019) Multiperiod stock allocation via robust
  optimization. \emph{Management Science} 65(2):794--818.

\bibitem[{Jasin \protect\BIBand{} Sinha(2015)}]{jasin2015lp}
Jasin S, Sinha A (2015) An lp-based correlated rounding scheme for multi-item
  ecommerce order fulfillment. \emph{Operations Research} 63(6):1336--1351.

\bibitem[{Jati et~al.(2023)Jati, Ekambaram, Pal, Quanz, Gifford, Harsha,
  Siegel, Mukherjee, \protect\BIBand{} Narayanaswami}]{jati2023hier}
Jati A, Ekambaram V, Pal S, Quanz B, Gifford WM, Harsha P, Siegel S, Mukherjee
  S, Narayanaswami C (2023) Hierarchical proxy modeling for improved {HPO} in
  time series forecasting. \emph{Proceedings of the 29th ACM SIGKDD Conference
  on Knowledge Discovery and Data Mining}, 891–900, KDD '23.

\bibitem[{Kelly(1956)}]{kelly1956new}
Kelly JL (1956) A new interpretation of information rate. \emph{{T}he {B}ell
  {S}ystem {T}echnical {J}ournal} 35(4):917--926.

\bibitem[{Kleywegt et~al.(2002)Kleywegt, Shapiro, \protect\BIBand{} Homem-de
  Mello}]{kleywegt2002sample}
Kleywegt AJ, Shapiro A, Homem-de Mello T (2002) The sample average
  approximation method for stochastic discrete optimization. \emph{SIAM Journal
  on optimization} 12(2):479--502.

\bibitem[{Lei et~al.(2018)Lei, Jasin, \protect\BIBand{} Sinha}]{lei2018joint}
Lei Y, Jasin S, Sinha A (2018) Joint dynamic pricing and order fulfillment for
  e-commerce retailers. \emph{Manufacturing \& Service Operations Management}
  20(2):269--284.

\bibitem[{Lim et~al.(2021)Lim, Jiu, \protect\BIBand{} Ang}]{lim2021integrating}
Lim YF, Jiu S, Ang M (2021) Integrating anticipative replenishment allocation
  with reactive fulfillment for online retailing using robust optimization.
  \emph{Manufacturing \& Service Operations Management} 23(6):1616--1633.

\bibitem[{Lim \protect\BIBand{} Wang(2017)}]{lim2017inventory}
Lim YF, Wang C (2017) Inventory management based on target-oriented robust
  optimization. \emph{Management Science} 63(12):4409--4427.

\bibitem[{Makridakis et~al.(2022)Makridakis, Spiliotis, \protect\BIBand{}
  Assimakopoulos}]{makridakis2022m5}
Makridakis S, Spiliotis E, Assimakopoulos V (2022) M5 accuracy competition:
  Results, findings, and conclusions. \emph{International Journal of
  Forecasting} 38(4):1346--1364, ISSN 0169-2070,
  \urlprefix\url{http://dx.doi.org/https://doi.org/10.1016/j.ijforecast.2021.11.013},
  special Issue: M5 competition.

\bibitem[{Nageswaran et~al.(2020)Nageswaran, Cho, \protect\BIBand{}
  Scheller-Wolf}]{nageswaran2020consumer}
Nageswaran L, Cho SH, Scheller-Wolf A (2020) Consumer return policies in
  omnichannel operations. \emph{Management Science} 66(12):5558--5575.

\bibitem[{NeilsonIQ(2022)}]{neilson}
NeilsonIQ (2022)
  \urlprefix\url{https://nielseniq.com/global/en/insights/success-story/2022/pivoting-to-long-term-growth-with-omnichannel-strategy/}.

\bibitem[{Oreshkin et~al.(2020)Oreshkin, Carpov, Chapados, \protect\BIBand{}
  Bengio}]{oreshkin2020nbeats}
Oreshkin BN, Carpov D, Chapados N, Bengio Y (2020) N-beats: Neural basis
  expansion analysis for interpretable time series forecasting.
  \emph{International Conference on Learning Representations}.

\bibitem[{Paterson et~al.(2011)Paterson, Kiesm{\"u}ller, Teunter,
  \protect\BIBand{} Glazebrook}]{paterson2011inventory}
Paterson C, Kiesm{\"u}ller G, Teunter R, Glazebrook K (2011) Inventory models
  with lateral transshipments: A review. \emph{European Journal of Operational
  Research} 210(2):125--136.

\bibitem[{Rahimian \protect\BIBand{}
  Mehrotra(2019)}]{rahimian2019distributionally}
Rahimian H, Mehrotra S (2019) Distributionally robust optimization: A review.
  \emph{arXiv preprint arXiv:1908.05659} .

\bibitem[{Rotando \protect\BIBand{} Thorp(1992)}]{rotando1992kelly}
Rotando LM, Thorp EO (1992) The {K}elly criterion and the stock market.
  \emph{The American Mathematical Monthly} 99(10):922--931.

\bibitem[{Salinas et~al.(2019)Salinas, Flunkert, \protect\BIBand{}
  Gasthaus}]{salinas2019deepar}
Salinas D, Flunkert V, Gasthaus J (2019) Deep{AR}: Probabilistic forecasting
  with autoregressive recurrent networks. \emph{International Journal of
  Forecasting} 8(2):136--153.

\bibitem[{Saénz et~al.(2022)Saénz, Revilla, \protect\BIBand{}
  Borrella}]{HBR2022}
Saénz MJ, Revilla E, Borrella I (2022) Digital transformation is changing
  supply chain relationships. \emph{Harvard Business Review} .

\bibitem[{Sherali \protect\BIBand{} Adams(1998)}]{sherali1998reformulation}
Sherali HD, Adams WP (1998) Reformulation-linearization techniques for discrete
  optimization problems. \emph{Handbook of Combinatorial Optimization:
  Volume1--3} 479--532.

\bibitem[{Simchi-Levi et~al.(2019)Simchi-Levi, Wang, \protect\BIBand{}
  Wei}]{simchi2019constraint}
Simchi-Levi D, Wang H, Wei Y (2019) Constraint generation for two-stage robust
  network flow problems. \emph{INFORMS Journal on Optimization} 1(1):49--70.

\bibitem[{Tagaras \protect\BIBand{} Cohen(1992)}]{tagaras1992pooling}
Tagaras G, Cohen MA (1992) Pooling in two-location inventory systems with
  non-negligible replenishment lead times. \emph{Management science}
  38(8):1067--1083.

\bibitem[{Xin(2021)}]{xin2021understanding}
Xin L (2021) Understanding the performance of capped base-stock policies in
  lost-sales inventory models. \emph{Operations Research} 69(1):61--70.

\bibitem[{Yang \protect\BIBand{} Qin(2007)}]{yang2007capacitated}
Yang J, Qin Z (2007) Capacitated production control with virtual lateral
  transshipments. \emph{Operations research} 55(6):1104--1119.

\bibitem[{Zeng \protect\BIBand{} Zhao(2013)}]{zeng2013solving}
Zeng B, Zhao L (2013) Solving two-stage robust optimization problems using a
  column-and-constraint generation method. \emph{Operations Research Letters}
  41(5):457--461.

\end{thebibliography}

% Appendix here
% Options are (1) APPENDIX (with or without general title) or 
%             (2) APPENDICES (if it has more than one unrelated sections)
% Outcomment the appropriate case if necessary
%
% \begin{APPENDIX}{<Title of the Appendix>}
% \end{APPENDIX}
%
%   or 
%
% \begin{APPENDICES}
% \section{<Title of Section A>}
% \section{<Title of Section B>}
% etc
% \end{APPENDICES}

% References here (outcomment the appropriate case) 

% CASE 1: BiBTeX used to constantly update the references 
%   (while the paper is being written).
%\bibliographystyle{informs2014} % outcomment this and next line in Case 1
%\bibliography{<your bib file(s)>} % if more than one, comma separated

% CASE 2: BiBTeX used to generate mypaper.bbl (to be further fine tuned)
%\input{mypaper.bbl} % outcomment this line in Case 2

\end{document}